\documentclass[pdftex, 10pt, a4paper,reqno]{amsart}

\usepackage{amsthm}
{\theoremstyle{definition}
\newtheorem{dfn}{Definition}[section]}
\newtheorem{prop}[dfn]{Proposition}

\newtheorem{thm}[dfn]{Theorem}
{\theoremstyle{definition}
\newtheorem{rem}[dfn]{Remark}}
\newtheorem{lem}[dfn]{Lemma}
\newtheorem{cor}[dfn]{Corollary}
\newtheorem{conj}[dfn]{Conjecture}

{\theoremstyle{definition}
}

\usepackage[top=29truemm,bottom=29truemm,left=35truemm,right=35truemm]{geometry}
\usepackage{amsmath,amssymb,amscd,bm,graphicx,tikz-cd,upgreek,dsfont,amsfonts,tikz,euscript,enumerate,longtable}

\usetikzlibrary{cd,matrix,arrows,shapes,calc}
\usepackage[all]{xy}
\usepackage[colorlinks=true,
    citecolor=darkpowderblue,
    linkcolor=alizarin,
    urlcolor=green(pigment)
    ,pagebackref 
    ]{hyperref}

\usepackage[mathlines]{lineno} 
\usepackage{amsmath}

\usepackage{etoolbox} 

\newcommand*\linenomathpatch[1]{%
  \cspreto{#1}{\linenomath}%
  \cspreto{#1*}{\linenomath}%
  \csappto{end#1}{\endlinenomath}%
  \csappto{end#1*}{\endlinenomath}%
}
\newcommand*\linenomathpatchAMS[1]{%
  \cspreto{#1}{\linenomathAMS}%
  \cspreto{#1*}{\linenomathAMS}%
  \csappto{end#1}{\endlinenomath}%
  \csappto{end#1*}{\endlinenomath}%
}

\expandafter\ifx\linenomath\linenomathWithnumbers
  \let\linenomathAMS\linenomathWithnumbers
  \patchcmd\linenomathAMS{\advance\postdisplaypenalty\linenopenalty}{}{}{}
\else
  \let\linenomathAMS\linenomathNonumbers
\fi

\linenomathpatch{equation}
\linenomathpatchAMS{gather}
\linenomathpatchAMS{multline}
\linenomathpatchAMS{align}
\linenomathpatchAMS{alignat}
\linenomathpatchAMS{flalign}

\usepackage{mathtools} 
\usetikzlibrary{positioning} 

\usepackage[outline]{contour}

\definecolor{alizarin}{rgb}{0.82, 0.1, 0.26}
\definecolor{azure(colorwheel)}{rgb}{0.0, 0.5, 1.0}
\definecolor{blue(pigment)}{rgb}{0.2, 0.2, 0.6}
\definecolor{denim}{rgb}{0.08, 0.38, 0.74}
\definecolor{mint}{rgb}{0.24, 0.71, 0.54}
\definecolor{parisgreen}{rgb}{0.31, 0.78, 0.47}
\definecolor{persiangreen}{rgb}{0.0, 0.65, 0.58}
\definecolor{seagreen}{rgb}{0.18, 0.55, 0.34}
\definecolor{shamrockgreen}{rgb}{0.0, 0.62, 0.38}
\definecolor{green(pigment)}{rgb}{0.0, 0.65, 0.31}
\definecolor{cadmiumgreen}{rgb}{0.0, 0.42, 0.24}
\definecolor{lightseagreen}{rgb}{0.13, 0.7, 0.67}
\definecolor{mediumseagreen}{rgb}{0.24, 0.7, 0.44}
\definecolor{pinegreen}{rgb}{0.0, 0.47, 0.44}
\definecolor{tealgreen}{rgb}{0.0, 0.51, 0.5}
\definecolor{darkpowderblue}{rgb}{0.0, 0.2, 0.6}
\definecolor{darkspringgreen}{rgb}{0.09, 0.45, 0.27}
\definecolor{deepjunglegreen}{rgb}{0.0, 0.29, 0.29}
\definecolor{dartmouthgreen}{rgb}{0.05, 0.5, 0.06}

\def\ang<#1>{\langle #1\rangle}
\def\bigang<#1>{\left\langle #1\right\rangle}

\makeatletter
\newcommand*{\defeq}{\mathrel{\rlap{%
                    \raisebox{0.3ex}{$\m@th\cdot$}}%
                    \raisebox{-0.3ex}{$\m@th\cdot$}}%
                     =}
\makeatother

\tikzset{
    labl/.style={anchor=south, rotate=90, inner sep=.5mm}
}

\numberwithin{equation}{section}

\newcommand{\Dmod}{\operatorname{Dmod}}

\newcommand{\Dcoh}{\operatorname{Dcoh}}
\newcommand{\Kcoh}{\operatorname{Kcoh}}
\newcommand{\Acoh}{\operatorname{Acoh}}

\newcommand{\Sing}{\operatorname{Sing}}
\newcommand{\Spec}{\operatorname{Spec}}
\newcommand{\Supp}{\operatorname{Supp}}
\newcommand{\Hom}{\operatorname{Hom}}

\newcommand{\Ext}{\operatorname{Ext}}

\newcommand{\Tot}{\operatorname{Tot}}
\newcommand{\RHom}{\operatorname{{\bf R}Hom}}

\newcommand{\End}{\operatorname{End}}

\newcommand{\Pic}{\operatorname{Pic}}

\newcommand{\Auteq}{\operatorname{Auteq}}
\newcommand{\Max}{\operatorname{Max}}

\newcommand{\coh}{\operatorname{coh}}
\newcommand{\Qcoh}{\operatorname{Qcoh}}

\newcommand{\Mod}{\operatorname{Mod}}

\newcommand{\fmod}{\operatorname{mod}}
\newcommand{\refl}{\operatorname{ref}}
\newcommand{\add}{\operatorname{add}}

\newcommand{\Db}{\operatorname{D^b}}

\newcommand{\D}{\operatorname{D}}

\newcommand{\CM}{\operatorname{CM}}

\newcommand{\codim}{\operatorname{codim}}

\newcommand{\id}{\operatorname{id}}

\newcommand{\semicolon}{;\hspace{1.1mm}}

\newcommand{\Ker}{\operatorname{Ker}}
\newcommand{\Cok}{\operatorname{Cok}}

\newcommand{\Res}{\operatorname{\sf Res}}

\newcommand{\fl}{\operatorname{\sf fl}}

\newcommand{\rep}{\operatorname{rep}}

\newcommand{\cA}{\mathcal{A}}

\newcommand{\cC}{\mathcal{C}}
\newcommand{\cD}{\mathcal{D}}
\newcommand{\cE}{\mathcal{E}}
\newcommand{\cF}{\mathcal{F}}
\newcommand{\cG}{\mathcal{G}}

\newcommand{\cK}{\mathcal{K}}
\newcommand{\cL}{\mathcal{L}}

\newcommand{\cN}{\mathcal{N}}
\newcommand{\cO}{\mathcal{O}}

\newcommand{\cS}{\mathcal{S}}
\newcommand{\cT}{\mathcal{T}}

\newcommand{\cV}{\mathcal{V}}

\newcommand{\bA}{\mathbb{A}}
\newcommand{\bC}{\mathbb{C}}

\newcommand{\bF}{\mathbb{F}}
\newcommand{\bG}{\mathbb{G}}

\newcommand{\bL}{\mathbb{L}}

\newcommand{\bP}{\mathbb{P}}

\newcommand{\bR}{\mathbb{R}}

\newcommand{\bZ}{\mathbb{Z}}

\newcommand{\bfP}{{\bf P}}

\newcommand{\g}{\mathfrak{g}}

\newcommand{\m}{\mathfrak{m}}

\newcommand{\X}{\mathfrak{X}}
\newcommand{\Y}{\mathfrak{Y}}

\newcommand{\scrA}{\EuScript{A}}
\newcommand{\scrB}{\EuScript{B}}
\newcommand{\scrC}{\EuScript{C}}

\newcommand{\scrF}{\EuScript{F}}

\newcommand{\scrH}{\EuScript{H}}

\newcommand{\scrK}{\EuScript{K}}
\newcommand{\scrL}{\EuScript{L}}
\newcommand{\scrM}{\EuScript{M}}

\newcommand{\scrS}{\EuScript{S}}
\newcommand{\scrT}{\EuScript{T}}

\newcommand{\scrV}{\EuScript{V}}
\newcommand{\scrW}{\EuScript{W}}

\newcommand{\dsG}{\mathds{G}}

\def\res{\mathop{\sf res}\nolimits}

\newcommand{\ST}{\operatorname{\sf{ST}}}

\newcommand{\sfk}{\Bbbk} 

\newcommand{\M}{\operatorname{\sf{M}}}
\newcommand{\N}{\operatorname{\sf{N}}}

\newcommand{\dR}{\operatorname{\bf{R}}\!}
\newcommand{\dL}{\operatorname{\bf{L}\!}}

\newcommand{\bsigma}{\boldsymbol{\Sigma}}
\newcommand{\bnabla}{\boldsymbol{\nabla}}

\newcommand{\simto}{\,\,\,\,\raisebox{4pt}{$\sim$}{\kern -1.1em\longrightarrow}\,}
\newcommand{\hookto}{\hookrightarrow}

\newcommand{\gwedge}{\raisebox{1pt}[0pt][0pt]{$\textstyle\bigwedge$}}

\renewcommand{\l}{\langle}
\renewcommand{\r}{\rangle}

\usepackage{setspace}
\newcommand{\marginparstretch}{0.6}
\let\oldmarginpar\marginpar
\renewcommand\marginpar[1]{\-\oldmarginpar[\framebox{\setstretch{\marginparstretch}\begin{minipage}{\marginparwidth}{\raggedleft\tiny #1}\end{minipage}}]{\framebox{\setstretch{\marginparstretch}\begin{minipage}{\marginparwidth}{\raggedright\tiny #1}\end{minipage}}}}

\newcommand{\relmiddle}[1]{\mathrel{}\middle#1\mathrel{}}

\makeatother

\newcommand{\LatticePoint}[1][]{%
\begin{tikzpicture}
\filldraw (0,0) circle (2pt);
\end{tikzpicture}
}

\newcommand{\EeightFour}[1][]{%
\begin{tikzpicture}[xscale=\EeightFourScalex,yscale=\EeightFourScaley]
\coordinate (A) at (0,0);
\node at (A) {\LatticePoint};
\coordinate (B) at (0,2);
\coordinate (C) at (0,4);
\coordinate (Ap) at (4,0);
\coordinate (Bp) at (4,2);
\coordinate (Cp) at (4,4);
\coordinate (b) at (2,4);
\coordinate (bp) at (2,0);
\draw (A)--(C);
\draw (Ap)--(Cp);
\draw (A)--(Ap);
\draw (B)--(Bp);
\draw (C)--(Cp);
\draw (b)--(bp);
\draw (A)--(b);
\draw (A)--(Cp);
\draw (C)--(bp);
\draw (C)--(Ap);
\draw (Ap)--(b);
\draw (Cp)--(bp);
\end{tikzpicture}
}

\usepackage{array,multirow,booktabs,longtable}
\makeatletter
\renewcommand{\p@enumi}{C}
\makeatother

\makeatletter
\@namedef{subjclassname@2020}{%
 \textup{2020} Mathematics Subject Classification}
\makeatother

\begin{document}

\pagewiselinenumbers

\title[Mutations of NCCR in GIT]{Mutations of noncommutative crepant resolutions in geometric invariant theory}

\author[W.~Hara]{Wahei Hara}
\address{ Kavli Institute for the Physics and Mathematics of the Universe (WPI), University of Tokyo, 5-1-5 Kashiwanoha, Kashiwa, 277-8583, Japan}
\email{wahei.hara@ipmu.jp}

\author[Y.~Hirano]{Yuki Hirano}
\address{Tokyo University of Agriculture and Technology, 2-24-16 Nakacho, Koganei, Tokyo 184-8588, Japan}\email{hirano@go.tuat.ac.jp}

\date{}

\begin{abstract} 
Let $X$ be a generic quasi-symmetric representation of a connected reductive group $G$. 
The GIT  quotient stack $\X=[X^{\rm ss}(\ell)/G]$ with respect to a generic $\ell$ is a (stacky) crepant resolution of the affine quotient $X/G$, and it is derived equivalent to a noncommutative crepant resolution (=NCCR)  of $X/G$. Halpern-Leistner and Sam showed that the derived category $\Db(\coh \X)$ is equivalent to  certain subcategories of $\Db(\coh [X/G])$, which are called  {magic windows}.
This paper studies equivalences between magic windows
that correspond to wall-crossings in a hyperplane arrangement
in terms of NCCRs.
We show that those equivalences coincide with derived equivalences between NCCRs induced by tilting modules, 
and that those tilting modules are obtained by certain operations of modules, which is called exchanges of modules. 
When $G$ is a torus, it turns out that the exchanges are nothing but  iterated Iyama--Wemyss mutations. 
Although we mainly discuss resolutions of affine varieties, 
our theorems also yield a result for projective Calabi-Yau varieties. 
Using techniques from the theory of noncommutative matrix factorizations, we show that Iyama--Wemyss mutations induce a group action of the fundamental group $\pi_1(\bP^1\,\backslash\{0,1,\infty\})$ on the derived category  of a Calabi-Yau complete intersection in a weighted projective space.
\end{abstract}

\maketitle{}

\section{Introduction}
\subsection{Backgrounds}
A crepant resolution is one of the best modifications of singularities.
This can be regarded as a higher dimensional analog of minimal resolutions of surface singularities, and in the terminology from minimal model theory, a crepant resolution can be paraphrased as a smooth minimal model of the singularity.

As a noncommutative analog of the notion of crepant resolutions,  Van den Bergh introduced noncommutative crepant resolutions (= NCCRs) \cite{vdbnccr, vdb}. 
In both commutative and noncommutative cases, the existence of such a resolution is not always true.
There are two large classes of singularities for which the study of NCCRs (and crepant resolutions) is well established.
One is the class of quotient singularities arising from quasi-symmetric representations of reductive groups, which was first studied in \cite{svdb}, and the other is the class of ($3$-fold) compound du Val singularities, studied in \cite{vdb3dflop,HomMMP}.
To investigate the latter class, Iyama and Wemyss \cite{IW} introduced an operation called \textit{mutation}, which produce a new NCCR from the original one.
By Kawamata \cite{kaw}, it is known that all minimal models (and hence all crepant resolutions) are connected by iterated flops, and mutations can be regarded as a noncommutative counterpart of flops. Indeed, it is proved in \cite{HomMMP} that the derived equivalences associated to 3-fold flops, which are established in \cite{bri, che}, correspond to derived equivalences associated to mutations of NCCRs. This interpretation and the technique of mutations of NCCRs provides the main ingredients for the study of  Bridgeland stability conditions for 3-fold flops \cite{HW,HW2}. 

The main purpose of this paper is to import such a technology established by \cite{IW} to deepen the study of NCCRs for quotient singularities arising from quasi-symmetric representations,
by studying the combinatorics associated to the representation,
and by accessing to the ideas from \cite{hl-s, svdb}.

\subsection{Exchanges and mutations of modifying modules} \label{intro: NCCR}
The present section, Section \ref{sect: qsrep}, and Section \ref{sect: wcfga} explain the setting of this paper, and recall some terminologies, notations, and known results that are needed to state our results.
The precise statements of main results are given in Section \ref{sect: intro main theorems}.

Let $R$ be a normal equidimensional Gorenstein ring. A finitely generated reflexive $R$-module $M$ is said to be {\it modifying} if the endomorphism ring $\End_R(M)$ is Cohen-Macaulay as an $R$-module. A {\it noncommutative crepant resolution} (=NCCR) of $R$ is the endomorphism ring $\Lambda=\End_R(M)$ of some modifying $R$-module $M$ such that the global dimension of $\Lambda$ is finite. If $\End_R(M)$ is an NCCR, we say that $M$ gives an NCCR.
The following is one of the central problems about NCCRs.

\begin{conj}[\cite{vdbnccr}]
Let $R$ be an equidimensional normal Gorenstein ring.
Then all crepant resolutions and all NCCRs of $R$ are derived equivalent.
\end{conj}

In relation to this derived equivalence problem, Iyama and Wemyss introduced the notion of mutation of NCCRs, which is defined in the following way.
Let $M,N$ and $L$ be finitely generated reflexive  $R$-modules. A {\it right $(\add L)_N$-approximation} of $M$ is a morphism
$
\alpha\colon K\to M
$
from $K\in \add L$ such that the induced morphism $\alpha\circ(-)\colon \Hom(N,K)\to \Hom(N,M)$ is surjective. A  {\it right $(\add L)_N$-exchange} of $M$ is the kernel of a  right  $(\add L)_N$-approximation of $M$.  Although right exchanges are not unique, we denote by $\varepsilon_{(L,N)}(M)$ one of right $(\add L)_N$-exchanges of $M$ by abuse of notation. By definition, there is an exact sequence
\[
0\to \varepsilon_{(L,N)}(M)\hookto K\xrightarrow{\alpha}M,
\]
where $\alpha$ is a right $(\add L)_N$-approximation of $M$. Dually  a {\it left $(\add L)_N$-exchange} of $M$ is defined to be the $R$-dual of a right $(\add L^*)_{N^*}$-exchange of $M^*$, and one of them is denoted by $\varepsilon^-_{(L,N)}(M)$. 
If $L=N$,  the right and left exchanges are simply written by $\varepsilon_N(M)$ and $\varepsilon^-_N(M)$, respectively. Suppose that $M=N \oplus N^c$ for some $R$-module $N^c$. The {\it right mutation} $\mu_{N}(M)$ and {\it left mutation} $\mu_N^-(M)$ of $M$ at $N$ are defined by
\begin{align*}
\mu_{N}(M)&\defeq N\oplus \varepsilon_N(N^c)\\
\mu_{N}^-(M)&\defeq N\oplus \varepsilon^-_N(N^c).
\end{align*}
For a positive integer $m$, an $m$-times iterated exchange is written by
\begin{align*}
\varepsilon_{(L,N)}^{\pm m}(M)&\defeq \underbrace{\varepsilon^{\pm}_{(L,N)}\bigl(\cdots\varepsilon^{\pm}_{(L,N)}(\varepsilon^{\pm}_{(L,N)}}_{m}(M))\cdots\bigr)
\end{align*}
and similarly
$\mu_N^{\pm m}(M)\defeq N\oplus \varepsilon_N^{\pm m}(N^c)$.
By \cite{IW}, if $M$ gives an NCCR of $R$, so do $\mu_N(M)$ and $\mu_N^-(M)$. Furthermore, $\End_R(M)$, $\End_R(\mu_N(M))$ and $\End(\mu^-_N(M))$ are all derived equivalent.  Thus mutation enables us to obtain a new NCCR from the original one, and it produces a canonical derived equivalence
\[
\Phi_N\colon \Db(\fmod \End_R(M))\simto \Db(\fmod \End_R(\mu^-_N(M))),
\]
which is called a {\it mutation functor} at $N$.

It is natural to ask if two given NCCRs are connected by (iterated) mutations or not.
It is known that, for many types of singularities, their natural NCCRs are actually connected by mutations \cite{hara17, hara22, HigNak, Nak, svdbNCBO, HomMMP}.
One of the main aims of this paper is to present a similar result for NCCRs associated to the quotient of quasi-symmetric representations, which are recalled in the next section.

\subsection{Quasi-symmetric representation} \label{sect: qsrep}
Let $G$ be a connected reductive group,  
$\M$  the character lattice of a maximal torus $T\subseteq G$, 
and $W$ the Weyl group of $G$. 
The Weyl group $W$ naturally acts on the lattice $\M$. 
Let $X$ be a $G$-representation with $T$-weights $\upbeta_1,\hdots,\upbeta_d$,  and assume that $X$ is {\it quasi-symmetric}, i.e. the  equality $\sum_{\upbeta_i\in L}\upbeta_i=0$ holds  for every 1-dimensional subspace $L$ in the real space $\M_{\bR} \defeq \M \otimes \bR$. 
Consider a $W$-invariant convex polytope 
\[
\bsigma\defeq\left\{\sum_{i=1}^d a_i\upbeta_i \middle| a_i\in[0,1]\right\}\subset \M_{\bR}.
\]
We always assume that $\bsigma$ linearly spans $\M_{\bR}$. 

For a generic element $\ell\in\M_{\bR}^W$, and consider the GIT quotient stack \[\X\defeq[X^{\rm ss}(\ell)/G].\]
To study the derived category of the stack $\X$, Halpern-Leistner and Sam \cite{hl-s} introduced \textit{magic window categories},
which are defined using a certain $W$-invariant convex polytope $\bnabla\subseteq\M_{\bR}$ and
a periodic locally finite hyperplane arrangement $\scrH^W$ in $\M_{\bR}^W$ (Definition \ref{def bnabla}).
More precisely, the magic window category $\scrM(\delta+\bnabla)$ is associated  to each element $\delta\in \M_{\bR}^W\backslash \scrH^W$, and it is defined as the thick subcategory of $\Db(\coh [X/G])$ generated by $G$-equivariant vector bundles of the form $V(\chi)\otimes \cO_X$ with $\chi\in (\delta+\bnabla)\cap \M^+$, where $V(\chi)$ denotes the irreducible $G$-representation with highest weight $\chi$, and $\M^+$ denotes the set of dominant weights. Furthermore,  \cite{hl-s} proved that the restriction functor 
  \[
  \res\colon \scrM(\delta+\bnabla)\simto\Db(\coh \X)
  \]
is an equivalence for all $\delta\in \M_{\bR}^W\backslash \scrH^W$.

Recall that $G$ acts on the ring $\sfk[X]$ of regular functions on $X$ by $(g\cdot \varphi)(x)\defeq \varphi(g^{-1}x)$, and the ring $R\defeq \sfk[X]^G$ of $G$-invariant regular functions defines a quotient singularity $X/G\defeq \Spec R$.
A quasi-symmetric representation $X$ is said to be {\it generic} if  ${\rm codim}(X\backslash X^{\rm ts},X)\geq 2$, where $X^{\rm ts}$ is a $G$-invariant open subspace of $X$ defined by 
\[
X^{\rm ts}\defeq\{x\in X\mid \mbox{the orbit $G\cdot x$ is closed and the stabilizer $G_x$ is trivial}\}.
\]
If $X$  is generic, then 
 the natural map 
 \[\varphi\colon \X\to \Spec R\] 
 is a (stacky) crepant resolution of $\Spec R$.
 
For each $\delta\in \M_{\bR}^W$, 
put $\scrC_{\delta}\defeq (\delta+\bnabla)\cap\M^+$. 
If $\delta\in \M_{\bR}^W\backslash \scrH^W$, 
the magic window $\scrM(\delta+\bnabla)$ has a canonical tilting object $\scrT_{\delta}\defeq \bigoplus_{\chi\in \scrC_{\delta}} V(\chi)\otimes \cO_X$, 
and hence $\scrV_{\delta}\defeq \res(\scrT_{\delta})$ is a tilting bundle on $\X$. 
In particular, the bundle $\scrT_{\delta}$ produces an equivalence
\[
\RHom(\scrT_{\delta},-)\colon \scrM(\delta+\bnabla)\simto \Db(\fmod \Lambda_{\delta}).
\]
between the magic window and the derived category of finitely generated right modules over $\Lambda_{\delta}\defeq\End(\scrT_{\delta})$.
Assume that  $X$ is generic. 
For a character $\chi\in\M$, the corresponding reflexive $R$-module $M(\chi)$ is defined by $M(\chi)\defeq \left(V(\chi)\otimes\cO_X\right)^G$. Then 
 $M_{\delta}\defeq \bigoplus_{\chi\in \scrC_{\delta}}M(\chi)$ is a modifying $R$-module, and  its endomorphism algebra  $\End_R(M_{\delta})\cong\Lambda_{\delta}$ is an NCCR of $R$ for every $\delta\in \M_{\bR}^W\backslash \scrH^W$.
 The next section explains what happens when the parameter $\delta$ is varied.

\subsection{Wall crossing and fundamental group action} \label{sect: wcfga} 
For each $\delta \in \M_{\bR}^W\backslash \scrH^W$, let $C(\delta)$ be the connected component of $\M_{\bR}^W\backslash \scrH^W$ that contains $\delta$.
We say that an ordered pair $(\delta,\delta')$ of elements $\delta,\delta'\in \M_{\bR}^W\backslash \scrH^W$ is an {\it adjacent pair} if $C(\delta)$ and $ C(\delta')$ are not equal and they share a codimension one wall.
Let $(\delta,\delta')$ be an adjacent pair, and $\X = [X^{\rm ss}(\ell)/G]$ the GIT quotient stack, where $\ell\defeq \delta'-\delta\in \M^W$. 
The ordered pair $(\delta,\delta')$ corresponds to a wall crossing in $\M_{\bR}^W\backslash \scrH^W$, and it induces an equivalence
 \begin{equation}\label{magic equiv intro}
 \scrM(\delta+\bnabla)\xrightarrow{\res} \Db(\coh \X)\xrightarrow{\res^{-1}}\scrM(\delta'+\bnabla)
 \end{equation}
of magic windows. In \cite{hl-s}, it is proved that the equivalences of magic windows of the form \eqref{magic equiv intro} defines a group action
\[
\uprho\colon \pi_1\bigl(\M_{\bR}^W\backslash \scrH^W\bigr)\to \Auteq \scrM(\delta+\bnabla).
\]
 This action is an analog of the group action on the derived category of 3-fold flops \cite{{IW9}}, which acts on  Bridgeland stability conditions for 3-fold flops as a deck transformation  of some regular covering map \cite{HW2}. Furthermore, the action $\uprho$ can be extended to a group action
\[
\widetilde{\uprho}\colon \pi_1(\scrK_X)\to\Auteq\scrM(\delta+\bnabla),
\]
where $\scrK_X\defeq \bigl(\M_{\bC}^W\backslash \scrH^W_{\bC}\bigr)/\M^W$ is the {\it stringy K\"ahler moduli space}  (= SKMS) of $X$.

\subsection{Main results} \label{sect: intro main theorems}

The main results in this paper interpret the equivalences in Section \ref{sect: wcfga} in terms of representation theory of NCCRs in Section \ref{intro: NCCR}.
Let $(\delta,\delta')$ be an adjacent pair in $\M_{\bR}^W\backslash \scrH^W$, and consider the GIT quotient stack $\X=[X^{\rm ss}(\ell)/G]$ with respect to  $\ell\defeq \delta'-\delta$. 
By the definition of adjacent pair, chambers $C(\delta)$ and $C(\delta')$ share the separating hyperplane $H$, and let $\delta_0\in H$ be the point where the line segment with endpoints $\delta$ and $\delta'$ meets $H$.    
Under this setting,
a set $\scrF_{(\delta,\delta')}$ of certain faces of the polytope $\delta_0+(1/2)\bsigma\subset \M_{\bR}$ and
direct summands $M_{\delta}^F \subset M_{\delta}$ and $L_{\delta}^F \subset M_{\delta_0}$
for each face $F\in \scrF_{(\delta,\delta')}$
are naturally defined.
Using them gives a direct sum decomposition of $M_{\delta}$:
\[
M_{\delta}=M_{\delta\cap \delta'}\oplus \bigoplus_{F\in \scrF_{(\delta,\delta')}}M_{\delta}^F, 
\]
where $M_{\delta\cap\delta'}\defeq \bigoplus_{\chi\in \scrC_{\delta}\cap\scrC_{\delta'}}M(\chi)$. 
Swapping $\delta$ and $\delta'$ in the decomposition above gives 
the decomposition of $M_{\delta'}$ with the index set $\scrF_{(\delta',\delta)}$.
In order to describe this decomposition in terms of $\scrF_{(\delta,\delta')}$,
 a bijection
\[
(-)^{\dagger}\colon \scrF_{(\delta,\delta')}\simto \scrF_{(\delta',\delta)}\semicolon F\mapsto F^{\dagger},
\]
will be constructed, and using this yields
\[
M_{\delta'}=M_{\delta\cap \delta'}\oplus \bigoplus_{F\in \scrF_{(\delta,\delta')}}M_{\delta}^{F^{\dagger}}.
\]
The following is our main result.

\begin{thm}[\ref{wall crossing}, \ref{main thm}]
Let $(\delta,\delta')$ be an adjacent pair in $\M_{\bR}^W\backslash \scrH^W$. 
Then the following hold.
\begin{itemize}
\item[(1)]  $T_{(\delta,\delta')}\defeq \Hom_R(M_{\delta},M_{\delta'})$ is a tilting $\Lambda_{\delta}$-module.
\item[(2)] The  diagram 
\[
\begin{tikzpicture}[xscale=1.3]
\node (A0) at (2,0) {$\Db(\fmod\Lambda_{\delta})$};
\node (A5) at (8,0) {$\Db(\fmod\Lambda_{\delta'})$};
\node (B0) at (2,1.5) {$\scrM(\delta+\bnabla)$};
\node (B4) at (4.75,1.5) {$\Db(\coh\X)$};
\node (B5) at (8,1.5) {$\scrM(\delta'+\bnabla)$};
\draw[->] (A0) -- node[above] {$\scriptstyle\RHom_{}(T_{(\delta,\delta')},-)$}(A5);
\draw[->] (B0) -- node[above] {$\scriptstyle\res$}(B4);
\draw[->] (B4) -- node[above] {$\scriptstyle\res^{-1}$}(B5);
\draw[->] (B0) -- node[left] {$\scriptstyle\RHom_{}(\scrT_{\delta},-)$}(A0);
\draw[->] (B5) -- node[right] {$\scriptstyle\RHom_{}(\scrT_{\delta'},-)$}(A5);
\end{tikzpicture}
\]
of equivalences is commutative.
\item[(3)] The module $M_{\delta}$ is obtained from $M_{\delta'}$ by applying  iterated exchanges of summands of $M_{\delta'}$. More precisely, $M_{\delta}$ can be represented by 
\[
M_{\delta\cap\delta'}\oplus \left(\bigoplus_{F\in \scrF_{(\delta,\delta')}} \varepsilon_{(L^F_{\delta},N^F_{\delta})}^{d_{F}^++\ell(w_0)-1}(M_{\delta'}^{F^{\dag}})\right),
\]
where $N_{\delta}^F\defeq M_{\delta}/M_{\delta}^F$, $d_F^+$ is some positive number associated to $F$, and $\ell(w_0)$ is the length of the longest element $w_0\in W$.
\end{itemize}
\end{thm}

If $G$ is an algebraic torus, the sets $\scrF_{(\delta,\delta')}$ and $\scrF_{(\delta',\delta)}$ have the unique elements $F$ and $F^*$ respectively. 
Set $N\defeq M_{\delta\cap\delta'}$, $d\defeq d_F^+$ and $d^*\defeq d_{F^*}^+$.  The following shows that $M_{\delta}$ can be obtained by iterated mutations of $M_{\delta'}$ at $N$.

\begin{thm}[\ref{mutation MMA toric}] \label{main 1.3}
Assume that $G$ is an algebraic torus. Then the following isomorphisms (up to additive closure) hold:
\begin{itemize}
\item[(1)] 
$
M_{\delta}\cong \mu_{N}^{d-1}(M_{\delta'})\cong\mu_{N}^{-d^*+1}(M_{\delta'})
$.
\vspace{1mm}\item[(2)]
$
M_{\delta'}\cong \mu_{N}^{d^*-1}(M_{\delta})\cong \mu_{N}^{-d+1}(M_{\delta}).
$
\end{itemize}
In particular, the iterated mutation of $M_{\delta}$ at $N$ is periodic, and the periodicity is $d + d^*- 2$.
\end{thm}

For each $i\geq 1$, put $M_i\defeq \mu_N^{-i}(M_{\delta})$ and 
$\Lambda_i\defeq \End_R(M_i)$. 
Note that the above theorem shows that $M_{d-1}=M_{\delta'}$. 
Recall that each mutation functor associates the equivalence
\[
\Phi_N\colon \Db(\fmod \Lambda_i)\simto \Db(\fmod\Lambda_{i+1}),
\]
and we write 
\[
\Phi_N^m\colon \Db(\fmod \Lambda_{\delta})\simto \Db(\fmod \Lambda_m)
\]
for the composition $\Db(\fmod \Lambda_{\delta})\xrightarrow{\Phi_N}\Db(\fmod \Lambda_1)\xrightarrow{\Phi_N}\cdots\xrightarrow{\Phi_N}\Db(\fmod \Lambda_m)$. A similar argument as in \cite{HigNak, hara17, hara22} shows the following.

\begin{thm}[\ref{composition mut}, \ref{toric main}] \label{main 1.4}
Assume that $G$ is an algebraic torus, and let $i\geq 1$. 
\begin{itemize}
\item[$(1)$]  There exists a tilting bundle $\scrV_i$ on $\X$ such that 
\[
\varphi_*(\scrV_i)\cong M_i.
\]

\item[$(2)$] The $\Lambda_i$-module $T_i\defeq \Hom(M_i,M_{i+1})$ is a tilting module, and the mutation functor 
\[
\Phi_N\colon\Db(\fmod \Lambda_i)\simto \Db(\fmod \Lambda_{i+1})
\] 
is given by $\RHom(T_i,-)$.
\item[$(3)$] The following diagram of equivalences commutes.
\[
\begin{tikzpicture}[xscale=1.3]
\node (A0) at (2,0) {$\Db(\fmod\Lambda_{\delta})$};
\node (A5) at (8,0) {$\Db(\fmod\Lambda_{\delta'})$};
\node (B0) at (2,1.5) {$\scrM(\delta+\bnabla)$};
\node (B4) at (4.75,1.5) {$\Db(\coh\X)$};
\node (B5) at (8,1.5) {$\scrM(\delta'+\bnabla)$};
\draw[->] (A0) -- node[above] {$\scriptstyle\Phi_N^{d-1}$}(A5);
\draw[->] (B0) -- node[above] {$\scriptstyle\res$}(B4);
\draw[->] (B4) -- node[above] {$\scriptstyle\res^{-1}$}(B5);
\draw[->] (B0) -- node[left] {$\scriptstyle\RHom_{}(\scrT_{\delta},-)$}(A0);
\draw[->] (B5) -- node[right] {$\scriptstyle\RHom_{}(\scrT_{\delta'},-)$}(A5);
\end{tikzpicture}
\]
\end{itemize}
\end{thm}

In Section \ref{section: THK}, it is observed that the above results on mutations descend to \textit{toric hyperK\"{a}hler varieties}.
Namely, the same statements as in Theorem \ref{main 1.3} and \ref{main 1.4} (1) hold for toric hyperK\"{a}hler varieties.

The above results also can be applied to the study of the derived category of projective Calabi-Yau varieties.
Combining the above results with the tilting theory for matrix factorizations established in \cite{ot,H3}, 
it is shown in Section \ref{section: CICY} that iterated mutations induce a group action of the fundamental group $\pi_1(\bP^1\,\backslash\{0,1,\infty\})$ on the derived category  of a Calabi-Yau complete intersection $Y\subset \bfP(a_1,\hdots,a_n)$ in a weighted projective space (Theorem \ref{window shift}). 
Furthermore, we show that the spherical twist autoequivalences associated to line bundles $\cO_Y(m)$  correspond to iterated Iyama--Wemyss mutations via noncommutative matrix factorizations (Corollary \ref{cor:spherical twist}).
This is done by generalising the results of Koseki-Ouchi \cite[Proposition 8.5]{ko} on Calabi-Yau hypersurfaces to Calabi-Yau complete intersections (Lemma \ref{lem for cor}).

\subsection{Notation}\label{notation} We work over an algebraically closed field $\sfk$ of characteristic zero. Throughout this paper, $G$ denotes a connected reductive group, and we fix  a maximal torus $T$ and a Borel subgroup $B$  of $G$; $T\subseteq B\subseteq G$. Let $n$ denote the dimension of $T$. We denote by $\Upphi$ the set of roots of $G$, and $\Upphi^{+}$ (resp. $\Upphi^-$) the set of positive (resp. negative) roots such that the weights in the Lie algebra of $B$ are negative roots. We define  $\rho\defeq (1/2)\sum_{\alpha\in \Phi^+}\alpha$.  We denote by  $\M$ the character lattice of $T$ and by $\N$ the lattice of one-parameter subgroups of $T$, and write $\M^+$ (resp. $\M_{\bR}^+$) for the dominant weights (resp. the dominant cone). For each  $\chi\in \M^+$, we write $V(\chi)$ for the irreducible representation of $G$ with highest weight $\chi$, and for any $G$-variety $S$, we write $V_S(\chi)\defeq V(\chi)\otimes \cO_S$ (resp. $V_{[S/G]}(\chi)$) for the $G$-equivariant locally free sheaf (resp. corresponding locally free sheaf on $[S/G]$) associated to $\chi$. We write $W$ for the Weyl group, and we write $w_0$ for the longest element of $W$. We fix  a $W$-invariant inner product $\l-,-\r\colon\M\times \N\to \bZ$. For any lattice $L\cong\bZ^n$, we write $L_{\bR}\defeq L\otimes\bR$ and $L_{\bC}\defeq L\otimes\bC$. 

Mainly during the study of the combinatorics of quasi-symmetric representations, 
this paper sets up masses of notations.
To sort out this pain, Appendix \ref{notation appendix} gives the list of notations that are frequently used among the discussions, and recalls rough definitions and the pages and lines where those notations are defined.

\vspace{3mm}

\noindent
\textbf{Acknowledgements.}
W.\,H.~would like to thank Prof.~Michael Wemyss for discussions and comments.
W.\,H.~was supported by EPSRC grant EP/R034826/1 and by ERC Consolidator Grant 101001227 (MMiMMa).
Y.\,H.~was supported by JSPS KAKENHI 19K14502.
The authors thank the referee for their careful reading and comments.


\section{Exchanges and Mutations of modifying modules}

\subsection{Noncommutative crepant resolution}

The present section recalls the definition of some basic notions that are studied in this article.

Let $R$ be a noetherian commutative ring of  Krull dimension $d$, and let $\Lambda$  be an $R$-algebra which is finitely generated as an $R$-module. Set $E\defeq \bigoplus_{\m\in\Max R}E(R/\m)$, where $E(-)$ denotes the  injective hull. This module  induces  Matlis duality 
\[
D\defeq\Hom_R(-,E)\colon \fl R\simto \fl R,
\]
where $\fl R$ denotes the category of finite length $R$-modules.
\begin{dfn} Let $n\in\bZ$. We say that $\Lambda$ is  {\it $n$-sCY} if there is a functorial isomorphism
\[
\Hom_{\D(\Mod \Lambda)}(M,N[n])\cong D\Hom_{\D(\Mod \Lambda)}(N,M)
\]
for all $M\in \Db(\fl\Lambda)$ and $N\in \Db(\fmod \Lambda)$.
\end{dfn}

\begin{prop}[{\cite[Proposition 3.10]{IR}}]
$R$ is $d$-sCY if and only if $R$ is Gorenstein and $\dim R_{\m}=d$ for every $\m\in\Max R$.
\end{prop}

\begin{dfn} Let $R$ be a normal $d$-sCY ring.
\begin{enumerate}
\item[(1)] A reflexive $R$-module $M$ is called a {\it modifying module} if $\End_R(M)$ is a maximal Cohen-Macaulay $R$-module.
\item[(2)] We say that a reflexive module $M$ gives a \textit{noncommutative crepant resolution (=NCCR)} $\Lambda = \End_R(M)$ if $M$ is modifying and the algebra $\Lambda$ has finite global dimension.
\end{enumerate}
\end{dfn}

\begin{rem}
Note that our definition of NCCR is different from the one in \cite{vdb} or \cite{IW}.
However, if $R$ is $d$-sCY, our definition is equivalent to other definitions.
See \cite[Lemma 4.2]{vdb} or \cite[Lemma 2.23]{IW}.
\end{rem}

\subsection{Exchanges/Mutations of modules}

This section recalls the notion of mutations introduced by Iyama--Wemyss  \cite{IW}. 
Let $R$ be a normal $d$-sCY algebra,  
and $M$, $N$, and $L$ reflexive $R$-modules.  
A {\it right $(\add L)_N$-approximation} of $M$ is a morphism
\[
\alpha\colon K\to M
\]
from $K\in \add L$ such that the induced morphism $\alpha\circ(-)\colon \Hom(N,K)\to \Hom(N,M)$ is surjective. If $L=N$, we just call $\alpha$ a {\it right $(\add L)$-approximation} of $M$. 
A right $(\add L)_N$-approximation $\alpha\colon K\to M$ of $M$ is said to be {\it minimal} if  any endomorphism $\varphi\in \End(K)$ satisfying $\alpha\circ \varphi=\alpha$  is an automorphism, and we say that  $\alpha$ is {\it reduced} if any direct summand $K'$ of $K$ is not contained in $\Ker(\alpha)$. Note that if a right approximation is minimal, it is reduced, and in the case when $R$ is complete local, the converse also holds.

\begin{rem}\label{app rem} 
Since $R$ is a normal domain,  the functor  $\Hom(L,-)\colon \refl R\simto \refl \End(L)$ is an equivalence \cite[Lemma 2.5\,(4)]{IW}, which implies that a right $(\add L)$-approximation of $M$ always exists. 
\end{rem}

\begin{dfn}
Let $R$ be a normal $d$-sCY, and let $M,N,L\in\refl R$.\vspace{1mm}\\
(1) A  {\it right $(\add L)_N$-exchange} of $M$ is the kernel of a  right  $(\add L)_N$-approximation of $M$. For a full subcategory $\cS\subseteq \refl R$, we denote by 
$\cE_{(L,N)}(\cS)$ \label{not:set of right exchanges} \linelabel{line not:set of right exchanges} 
the set of  right $(\add L)_N$-exchanges of $R$-modules in $\cS$. By abuse of notation,  we sometimes denote by  $\varepsilon_{(L,N)}(\cS)$ one of  objects  in $\cE_{(L,N)}(\cS)$.\vspace{1mm}\\
(2)
Dually a {\it left $(\add L)_N$-exchange} of $M$ is defined to be the $R$-dual of a right $(\add L^*)_{N^*}$-exchange of $M^*$. We write 
$\cE_{(L,N)}^{-}(\cS)$ 
\label{not:set of left exchanges} \linelabel{line not:set of left exchanges} 
for the set of  left $(\add L)_N$-exchanges of $M\in \cS$, and we denote by $\varepsilon_{(L,N)}^-(\cS)$ one of  objects in $\cE_{(L,N)}^{-}(\cS)$. Note that $\cE_{(L,N)}^{-}(\cS)=\bigl(\cE_{(L^*,N^*)}(\cS^*)\bigr)^*$.\vspace{1mm}\\
(3) A right (resp. left) exchange is said to be {\it reduced} if it is  the kernel of a reduced right (resp. left) approximation.

\end{dfn}

Since the kernel of a morphism between reflexive $R$-modules is again reflexive, the set
$\cE^{\pm}_{(L,N)}(\cS)$
is contained in $\refl R$.
For a positive integer $m$, set
\[\cE_{(L,N)}^{\pm m}(\cS)\defeq \underbrace{\cE^{\pm}_{(L,N)}\bigl(\cdots\cE^{\pm}_{(L,N)}(\cE^{\pm}_{(L,N)}}_{m}(\cS))\cdots\bigr),\]
and $\varepsilon_{(L,N)}^{\pm m}(\cS)$ denotes one of objects in $\cE_{(L,N)}^{\pm m}(\cS)$. 
For $M\in \refl R$, put $\cE_{(L,N)}^{\pm m}(M)\defeq \cE_{(L,N)}^{\pm m}(\{M\})$ and $\varepsilon_{(L,N)}^{\pm m}(M)\defeq \varepsilon_{(L,N)}^{\pm m}(\{M\})$.

\begin{lem}\label{exchange rem}
Notation is same as above.
\begin{itemize}
\item[$(1)$] If $L'\in\add L$, there is an inclusion \[\cE^{\pm}_{(L',N)}(\cS)\subseteq \cE^{\pm}_{(L,N)}(\cS),\] which remains  true when restricting to reduced exchanges.
\item[$(2)$]  If $N'\in \add N$,  there is an inclusion
\[\cE^{\pm}_{(L,N)}(\cS)\subseteq \cE^{\pm}_{(L,N')}(\cS),\]
which remains  true when restricting to reduced exchanges.
\item[$(3)$] For another full subcategory $\cS'\subseteq \refl R$, there is an inclusion 
\[
\cE^{\pm}_{(L,N)}(\cS)\oplus \cE^{\pm}_{(L,N)}(\cS')\subseteq\cE^{\pm}_{(L,N)}(\cS\oplus \cS').
\]
If $R$ is complete local, the similar inclusion also holds for reduced exchanges.
\end{itemize}
\begin{proof}
(1), (2) and the first assertion in (3) are obvious. The second assertion in (3) follows from the fact that, if $R$ is complete local, two approximations $\alpha\colon K\to M$ and $\alpha'\colon K'\to M'$ are reduced if and only if $\alpha\oplus \alpha'\colon K\oplus K'\to M\oplus M'$ is reduced. 
\end{proof}
\end{lem}

\begin{lem}
If $\Hom_R(N,M \oplus N)\in \CM R$, then 
$M\in \cE^-_{(L,N)}\cE_{(L,N)}(M)$ holds. 
Dually, assuming $\Hom_R(M \oplus N,N)\in\CM R$  implies that $M\in\cE_{(L,N)}\cE^-_{(L,N)}(M)$.
\begin{proof}
Assume that $\Hom(N,M \oplus N)$ is Cohen-Macaulay,
and consider an exact sequence
\[ 0 \to \Ker \alpha \to K \xrightarrow{\alpha} M, \]
where $\alpha$ is a right $(\add L)_N$-approximation.
Put $\mathbb{F} \defeq \Hom_R(N,-)$. 
Then the definition of the approximation gives an exact sequence
\[ 0 \to \mathbb{F}\Ker \alpha \to \mathbb{F}K \to \mathbb{F}M \to 0. \]
Now applying the functor $\Hom(-,\mathbb{F}R)$ to this sequence together with the reflexive equivalence proves that the dual sequence
\[ 0 \to M^* \to K^* \to (\Ker \alpha)^*\]
is exact.

 By the assumption, $\mathbb{F}M$ is Cohen-Macaulay, and  $\End(N)$ is  $d$-sCY  \cite[Lemma 2.22]{IW}. Applying the functor $\Hom(\mathbb{F}(-), \mathbb{F}N)$ to the original sequence and using the vanishing $\Ext^1_{\End(N)}(\mathbb{F}M, \End(N)) = 0$ imply that the sequence
\[ 0 \to \Hom(\mathbb{F}M, \mathbb{F}N) \to \Hom(\mathbb{F}K, \mathbb{F}N) \to \Hom(\mathbb{F}\Ker \alpha, \mathbb{F}N) \to 0 \]
remains exact.
Since all modules in the original sequence are reflexive, the reflexive equivalence and the duality yield an isomorphism
\[ \Hom_{\End(N)}(\mathbb{F}M, \mathbb{F}N) \simeq \Hom_R(N^*, M^*) \]
and similar isomorphisms for $K$ and $\Ker \alpha$, which imply the exactness of the sequence
\[ 0 \to \Hom(N^*,M^*) \to \Hom(N^*, K^*) \to \Hom(N^*, (\Ker \alpha)^*) \to 0 .\]
Thus the dual morphism
\[ K^* \to (\Ker \alpha)^* \]
is a right $(\add L^*)_{N^*}$-approximation with the kernel $M^*$, which proves the first assertion. The second assertion follows from a similar argument.
\end{proof}
\end{lem}

The following says that exchanging a direct summand of a modifying module gives a new  modifying module in  nice situations.

\begin{prop}\label{ex modi}
Let $N\oplus M$ be  a modifying $R$-module, and let $L\in \refl R$ such that  $\Hom(N,L)\in \CM R$. If $M'\in\cE^{\pm m}_{(L,N)}(M)$ is modifying, so is $N\oplus M'$. 
\end{prop}

Since $M'\in \cE^{\pm m}_{(L,N)}(M)$ is not modifying in general, neither $N\oplus M'$ is in general even if $N\oplus M$ is modifying and $\Hom(N,L)\in\CM R$.  The proof of Proposition \ref{ex modi} requires the following two standard lemmas.
We give the proofs for the convenience of the reader.

\begin{lem}\label{dual}
Let  $M\in \refl R$. The following equivalence holds.
\[
M\in \CM R  \Longleftrightarrow M^*\in \CM R.
\] 
\begin{proof}
We may assume that $R$ is local. Since $M$ is reflexive, it is enough to show the direction $(\Rightarrow)$. Since $R$ is Gorenstein, its injective dimension is finite. Thus the result follows from  \cite[Proposition 3.3.3 (b)]{bh}.
\end{proof}
\end{lem}

\begin{lem}\label{opposite hom}
Let $R$ be a Gorenstein normal ring, and let $M,N\in \refl R$. Then 
\[
\Hom_R(M,N)\in \CM R \Longleftrightarrow \Hom_R(N,M)\in \CM R.
\]
\begin{proof}
It is enough to prove the direction ($\Rightarrow$). Assume that $\Hom(M,N)\in \CM R$. Then Lemma \ref{dual} implies that $\Hom(M,N)^*\in \CM R$. But by Lemma \cite[Lemma 2.9]{IW}, there is an isomorphism $\Hom(M,N)^*\cong \Hom(N,M)$, which shows that  $\Hom(N,M)\in \CM R$.
\end{proof}
\end{lem}

\begin{proof}[Proof of Proposition \ref{ex modi}]
Assume that $m>0$ and $M'\in \cE^m_{(L,N)}(M)$. 
By definition, there are $(\add L)_N$-approximations $\alpha_i\colon L_i\to M_i$ such that $M_0=M$,  $M_i=\Ker (\alpha_{i-1})$ for $i=1,\hdots, m$ and $M'=\varepsilon^m_{(L,N)}(M)$. 
Then $\Hom(N,M_i)$ are Cohen-Macaulay for every $i=1,\hdots,m$.
Indeed, applying the functor $\Hom(N,-)$ to the exact sequence $0\to M_{i+1}\to L_i\xrightarrow{\alpha_i}M_i$ gives an exact sequence 
\[
0\to\Hom(N,M_{i+1})\to \Hom(N,L_i)\to \Hom(N,M_i)\to 0.
\]
Since $\Hom(N,L_i)\in \CM R$ by assumption, $ \Hom(N,M_i)\in \CM R$ implies that $\Hom(N,M_{i+1})$ is also Cohen-Macaulay. Since $\Hom(N,M_0)\in \CM R$, the claim follows by induction on $i$. In particular, $\Hom(N,M')\in \CM R$. By Lemma \ref{opposite hom}, $\Hom(M',N)$ is also  Cohen-Macaulay. Therefore, if $M'$ is modifying, so is $N\oplus M'$.

The proof for the case when $m<0$ is similar.
\end{proof}

If $L=N$, set
\[
\cE^{\pm}_N(\cS)\defeq \cE_{(L,N)}^{\pm}(\cS)\hspace{5mm}\mbox{and}\hspace{5mm}\varepsilon^{\pm}_N(\cS)\defeq \varepsilon_{(L,N)}^{\pm}(\cS).
\]
\begin{dfn}[\cite{IW}]
Suppose  $M=N\oplus N^c$ for some $N^c\in \refl R$, and put
\label{not: mutation} \linelabel{line not: mutation}
\begin{align*}
\mu_{N}(M)&\defeq N\oplus \varepsilon_N(N^c)\\
\mu_{N}^-(M)&\defeq N\oplus \varepsilon^-_N(N^c).
\end{align*}
The module $\mu_{N}(M)$  is called a {\it right mutation} of $M$ at $N$, and $\mu_{N}^-(M)$ is called a {\it left mutation} of $M$ at $N$. We say that a mutation $\mu^{\pm}_{N}(M)$ is {\it reduced} (resp. {\it minimal}) if $\mu^{\pm}_{N}(M)=N\oplus \varepsilon^{\pm}_{N}(N^c)$ with the exchange $\varepsilon^{\pm}_{N}(N^c)$ is reduced (resp. minimal).
For a positive integer $m$, left and right $m$-iterated mutations are denoted by
$
\mu_{N}^{m}(M)\defeq N\oplus \varepsilon^m_N(N^c)$ and $\mu_{N}^{-m}(M)\defeq N\oplus \varepsilon^{-m}_N(N^c)
$, respectively.
\end{dfn}

\begin{rem} Since a right approximation is not unique in general,  neither is  right/left mutation.  However,  right/left mutation is unique up to additive closure \cite[Lemma 6.2]{IW}, and  if $R$ is complete local, minimal mutations are unique up to isomorphism.
\end{rem}

If $N=L$ in Proposition \ref{ex modi}, the assumptions are automatically satisfied, and so an iterated mutation $\mu_N^d(N\oplus M)$ is  modifying. Furthermore,  the following fundamental properties of right/left mutations of modifying modules hold.

\begin{thm}[{\cite[Proposition 6.5, Theorem 6.8, Theorem 6.10]{IW}}]\label{IW mutation} Let $M\in \refl R$ be a modifying $R$-module.
\begin{itemize}
\item[$(1)$] $\mu_N(\mu_N^{-}(M))\cong M$ and $\mu_N^{-}(\mu_N(M))\cong M$ hold up to additive closure.
\item[$(2)$] $\End_R(M)$, $\End_R(\mu_N(M))$ and $\End_R(\mu_N^{-}(M))$ are all  derived equivalent.
\item[$(3)$] If $M$ gives an NCCR of $R$, so do $\mu_N(M)$ and $\mu_N^-(M)$.
\end{itemize}
\end{thm}

The derived equivalence of $\End_R(M)$ and $\End_R(\mu_N^-(M))$ in Theorem \ref{IW mutation}\,(2) is given by an equivalence induced by a tilting module, which can be explicitly constructed as follows.

Let $\alpha\colon K\to (N^c)^*$ be an $(\add N^*)$-approximation, and set
\[
V\defeq \Cok\left(\bF(N^c)\xrightarrow{\bF(\alpha^*)}\bF(K^*)\right) \hspace{3mm}\mbox{and}\hspace{3mm} Q\defeq \bF(N),
\]
where $\bF(-)\defeq\Hom_R(M,-)$.
Then it was proved in \cite{IW} that $V \oplus Q$ is a tilting $\End_R(M)$-module such that its endomorphism algebra is isomorphic to  $ \End_R(\mu_N^-M)$. Therefore, $V\oplus Q$ defines an equivalence  \label{not: mutation equiv} \linelabel{line not: mutation equiv}
\begin{equation}\label{mutation functor}
\Phi_N \defeq \RHom(V \oplus Q, -) \colon  \Db(\fmod \End_R(M)) \xrightarrow{\sim} \Db(\fmod \End_R(\mu_N^-M)), \end{equation}
which is called the {\it mutation functor} at $N$. For a positive integer $m$, by abuse of notation, write 
\begin{equation}\label{iterated mutation}
\Phi_N^m\colon \Db(\fmod \End(M))\simto \Db(\fmod \End(\mu_N^{-m}(M)))
\end{equation}
for the composition $\Db(\fmod \End(M))\xrightarrow{\Phi_N}\cdots\xrightarrow{\Phi_N}\Db(\fmod\End(\mu_N^{-m}(M)))$.

\begin{rem}
If $R$ is a three dimensional Gorenstein terminal singularity, the tilting module $V\oplus Q$ is isomorphic to $\Hom_R(M, \mu_N^-M)$. 
This simpler description of the tilting module $V \oplus Q$ is also true for the toric examples of this paper (Proposition \ref{composition mut}\,(2)),
but not true in general (see Remark \ref{counterexample for 2.21}).
\end{rem}

\subsection{Tilting bundles and mutations} 

This section discusses tilting bundles over algebraic stacks.
We start from recalling some basic facts on the derived categories of algebraic stacks.

For an algebraic stack $\Y$, let $\D(\Y)$ be the (unbounded) derived category of modules over $\cO_\Y$.
Let $\D_{\mathrm{qc}}(\Y)$ be the triangulated full subcategory of complexes with quasi-coherent cohomologies.
Similarly, if $\Y$ is noetherian, let $\D_{\mathrm{c}}(\Y)$ (resp.\! $\D_{\mathrm{c}}(\Qcoh \Y)$) be the triangulated full subcategory of $\D(\Y)$ (resp.\! $\D(\Qcoh \Y)$) consisting of complexes with coherent cohomologies.

The following fact is standard.

\begin{prop}
Let $\Y$ be a quasi-compact separated algebraic stack with finite stabilizers. Then the natural functor
\begin{align*} 
\D(\Qcoh \Y) \to \D_{\mathrm{qc}}(\Y)
\end{align*}
gives an equivalence of categories.
If $\Y$ is in addition noetherian, then the natural functors
\begin{align*}
\D^\ast(\coh \Y) \to \D_{\mathrm{c}}^\ast(\Qcoh \Y) \to \D_{\mathrm{c}}^\ast(\Y) 
\end{align*}
are equivalences, for $\ast \in \{- ,\mathrm{b} \}$.
\end{prop}

\begin{proof}
This is proved in \cite[Remark 2.12 and Proposition A.1]{BLS}.
Note that the claim for bounded above categories implies the one for bounded categories.
\end{proof}

We say that a quasi-compact quasi-separated algebraic stack $\Y$ has \textit{finite cohomological dimension} if there exists an integer $d$ such that $H^i(\Y, \cF) = 0$ for all $\cF \in \Qcoh(\Y)$ and $i > d$.
This class of algebraic stacks has very good characterization for compact objects in $\D_{\mathrm{qc}}(\Y)$.

\begin{dfn}
Let $\Y$ be a quasi-compact quasi-separated algebraic stack.
An object of $\D_{\mathrm{qc}}(\Y)$  is \textit{perfect} if it is smooth-locally isomorphic to a bounded complex of free $\cO_\Y$-modules of finite rank.
\end{dfn}

\begin{prop}
Let $\Y$ be a quasi-compact quasi-separated algebraic stack that has finite cohomological dimension.
Then an object $P \in \D_{\mathrm{qc}}(\Y)$ is perfect if and only if it is compact.
\end{prop}

\begin{proof}
This is proved in \cite[Lemma 4.5 and Remark 4.6]{HR}.
\end{proof}

\begin{dfn}
A \textit{tilting complex} on a quasi-compact quasi-separated algebraic stack $\Y$ is a perfect complex $\cT$ that satisfies the following two conditions.

\begin{enumerate}
\item[(1)] $\cT$ generates $\D_{\mathrm{qc}}(\Y)$, i.e.\! the kernel of the functor $\RHom(\cT,-)$ is zero.
\item[(2)] $\Ext^i(\cT,\cT) = 0$ for all $i \neq 0$.
\end{enumerate}
A \textit{tilting bundle} is  a vector bundle that is a tilting complex.
\end{dfn}

Let $\Y$ be a quasi-compact separated algebraic stack with finite stabilizers, and assume that $\Y$ has finite cohomological dimension.
If a tilting complex $\cT$ exists on $\Y$, it is a compact generator of $\D_{\mathrm{qc}}(\Y) \simeq \D(\Qcoh \Y)$,
and then the functor $\RHom(\cT,-)$ induces an equivalence
\[ \RHom(\cT,-) \colon \D_{\mathrm{qc}}(\Y) \xrightarrow{\sim} \D(\Mod \End_{\Y}(\cT)). \]

The following assertion is famous for schemes, and the stack version is explained in \cite[Section 2.1]{svdbtoric2}.

\begin{prop}
Let $\Y$ be a smooth separated Deligne-Mumford

that admits a tilting complex $\cT$. 
Then it induces an equivalence
\[ \RHom(\cT,-) \colon \Db(\coh \Y) \xrightarrow{\sim} \Db(\fmod \End_{\Y}(\cT)). \]
\end{prop}

When a (stacky) crepant resolution of a Gorenstein ring $R$ admits a tilting bundle, then the endomorphism algebra of the tilting bundle gives an NCCR of $R$. 
To state this precisely, we introduce the following ad hoc notation.
For a morphism $f \colon \X \to Y$ from a stack $\X$ to a scheme $Y$, define the functor $f_\sharp$ by the formula 
\[ f_\sharp(\cF) \defeq (f_*\cF)^{**}, \]
where $(-)^{**}$ denotes the double-dual.

\begin{lem} \label{pushdown}
Let $f \colon \Y \to \Spec R$ be a stacky crepant resolution, where $R$ is normal Gorenstein. 
For two coherent sheaves $\cF, \cE$ on $\Y$, assume that $\cF$ is torsion free and $\Hom_{\Y}(\cF, \cE)$ is reflexive as an $R$-module.
Then the natural $R$-linear morphisms
\begin{align*}
F &\colon \Hom_{\Y}(\cF, \cE) \to \Hom_R(f_*\cF, f_*\cE), ~ \text{and}\\
D &\colon \Hom_R(f_*\cF, f_*\cE) \to \Hom_R(f_\sharp\cF, f_\sharp\cE)
\end{align*}
are isomorphisms.
\end{lem}

\begin{proof}
Denote by $U$  the smooth locus of $\Spec R$, and by $i \colon U \to \Spec R$ the inclusion.
Put $M \defeq \Hom_{\Y}(\cF, \cE)$ and $N \defeq \Hom_R(f_*\cF, f_*\cE)$.
Since $M$ is reflexive by our assumption, the natural morphism $a \colon M \to i_*i^*M$ is an isomorphism.
Since $f$ is an isomorphism over $U$, $i^*M \simeq i^*N$ holds, and thus the adjunction gives
\[ M \xrightarrow[F]{} N \xrightarrow[\mathrm{adj}]{} i_*i^*M \xrightarrow[a^{-1}]{\sim} M. \]
Note that the composite of the above three morphisms is the identity when it is restricted to $U$, and since $M$ is reflexive, it is actually the identity of $M$.
Therefore, the morphism $F \colon M \to N$ is a splitting injection, and $N \simeq M \oplus T$ for some torsion $R$-module $T$.

To prove $F$ is an isomorphism, it is enough to show that $N$ is torsion free.
Since $\cE$ is a vector bundle and $f_*\cO_{\Y} \simeq R$, $f_*\cE$ is a torsion free $R$-module.
Now taking a surjection of $R^{\oplus r} \twoheadrightarrow f_*\cF$ for some $r \in \bZ_{> 0}$ gives an injection
\[ N \defeq \Hom_R(f_*\cF, f_*\cE) \hookrightarrow \Hom_R(R^{\oplus r}, f_*\cE) \simeq  (f_*\cE)^{\oplus r}. \]
Thus $N$ is a submodule of a torsion free $R$-module, which implies that $N$ is also torsion free.

Furthermore, since $\Hom_R(f_\sharp\cF, f_\sharp\cE)$ is reflexive and isomorphic to $M$ on $U$ by $D \circ F$, the composite $D \circ F$ and hence $D$ are isomorphisms.
\end{proof}

\begin{thm}
Let $R$ be a normal Gorenstein algebra and $f \colon \Y \to \Spec R$ is a 
stacky crepant resolution.
Assume that $\Y$ admits a tilting bundle $\cT$.
Then the reflexive $R$-module $f_\sharp \cT$ gives an NCCR of $Y$, which is naturally isomorphic to $\End_{\Y}(\cT)$.
\end{thm}

\begin{proof}
First, the derived equivalence associated to $\cT$ implies that $\End_\Y(\cT)$ has finite global dimension.
Second, the same proof as in \cite[Corollary 3.4]{svdbtoric1} shows that $\End_\Y(\cT)$ is Cohen-Macaulay as an $R$-module.
Finally, Lemma \ref{pushdown} gives the isomorphism $\End_\Y(\cT) \simeq \End_R(f_\sharp \cT)$.
Since $f_\sharp \cT$ is reflexive by definition, this completes the proof.
\end{proof}
 
The next proposition means that, under a nice geometric context, mutations of an NCCR and the associated equivalences can be described using tilting bundles.
This is a generalisation of \cite[Lemma 2]{hara22}.

\begin{prop} \label{tilting mutation}
Let $f \colon \Y \to \Spec R$ be a stacky crepant resolution, where $R$ is a $d$-sCY.
Given a vector bundle $\cL$ and an exact sequence
\begin{equation} \label{es 2.21}
  0 \to \cK \to \cE \to \cC \to 0  
\end{equation}  
of vector bundles,
assume that
\begin{enumerate}
\item[$({\rm a})$] $\cK \oplus \cL$ and $\cC \oplus \cL$ are tilting bundles, and that
\item[$({\rm b})$] $\cE \in \add \cL$.
\end{enumerate}
Let us write $K\defeq f_{\sharp}\cK$, $E\defeq f_{\sharp} \cE$, $C\defeq f_{\sharp} \cC$ and $L\defeq f_{\sharp} \cL$. Then, the following hold.
\begin{enumerate}
\item[$(1)$] The exact sequence 
\[ 0 \to K \to E \to C \]
gives a right $(\add L)$-approximation of $C$, and hence $L\oplus K \cong \mu_{L}(L\oplus C)$ holds up to additive closure.
\item[$(2)$] There is an isomorphism $L\oplus C\cong \mu_L^-(L\oplus K)$ up to additive closure.
\item[$(3)$] The mutation functor
\[ \Phi_{L} \colon \Db(\fmod \End_R(L \oplus K)) \to \Db(\fmod \End_R( L \oplus C)) \]
associated to the left mutation in $(2)$ is given by a tilting module $\Hom_{\Y}(\cL \oplus \cK, \cL \oplus \cC)$.
In particular, the following diagram of functors commutes.
\[  \begin{tikzpicture}[auto,->]
\node (a) at (0,0) {$\Db(\coh \Y)$};
\node (b) at (5,0) {$\Db(\fmod \End_R(L \oplus K))$};
\node (c) at (5,-1.2) {$\Db(\fmod \End_R(L \oplus C))$};

\draw (a) -- node {$\scriptstyle \RHom(\cL \oplus \cK,-)$} (b);
\draw (a) -- node[swap] {$\scriptstyle \RHom(\cL \oplus \cC,-)$} (node cs:name=c,anchor=north west);
\draw (b) -- node {$\scriptstyle \Phi_{L}$} (c);
\end{tikzpicture}
\]
\item[$(4)$] If $\codim \Sing(R) \geq3$, then the natural morphism  
\[ \Hom_{\Y}(\cL\oplus\cK,\cL\oplus\cC) \to \Hom_R(L\oplus K,L\oplus C) \]
is an isomorphism.
In particular, $T \defeq \Hom_R(L\oplus K,\mu_L^-(L\oplus K))$ is a tilting $\End_R(L \oplus K)$-module such that $\End_{\End_R(L \oplus K)}(T) \simeq \End_R(\mu_L^-(L\oplus K))$.
\end{enumerate}
\end{prop}

\begin{proof}
Since $\End_\Y(\cC \oplus \cL)$ is an NCCR and hence a Cohen-Macaulay $R$-module, the assumptions (a) and (b) imply that $\Hom_{\Y}(\cL, \cE)$ and $\Hom_{\Y}(\cL, \cC)$ are reflexive.
Thus Lemma \ref{pushdown} yields the following commutative diagram
\[ \begin{tikzcd}
\Hom_{\Y}(\cL, \cE) \arrow[r] \arrow[d, "\simeq"] & \Hom_{\Y}(\cL, \cC) \arrow[r] \arrow[d, "\simeq"] & \Ext^1_{\Y}(\cL, \cK) = 0 \\
 \Hom_R(L, E) \arrow[r] & \Hom_R(L, C) &
\end{tikzcd} \]
Thus $\Hom_R(L, E) \to \Hom_R(L, C)$ is surjective, which proves (1), and (2) follows from (1) and Theorem \ref{IW mutation}\,(1).
Indeed, \cite[Propositon 6.4\,(3)]{IW} implies that the dual sequence
\[ 0 \to C^* \to E^* \to K^* \]
is exact and the morphism $E^* \to K^*$ is an $(\add L^*)$-approximation.

For (3), put
\begin{align*} 
V &\defeq 
\mathrm{Cok}\bigl( \Hom_R(L \oplus K, K) \to \Hom_R(L \oplus K,E) \bigr),
\end{align*}
and $Q \defeq \Hom_R(L \oplus K, L) \simeq \Hom_{\Y}(\cL \oplus \cK, \cL)$.
Then as recalled after Theorem \ref{IW mutation}, the equivalence is given by the tilting module $V \oplus Q$.
We now show that $V \simeq \Hom_{\Y}(\cL \oplus \cK, \cC)$.
To see this, consider the following commutative diagram of exact sequences
\[ \begin{tikzcd}
0 \arrow[r] & \Hom_{\Y}(\cL \oplus \cK, \cK) \arrow[r, hook] \arrow[d, "\simeq"] & \Hom_{\Y}(\cL \oplus \cK, \cE) \arrow[r, twoheadrightarrow] \arrow[d, "\simeq"] & \Hom_{\Y}(\cL \oplus \cK, \cC) \arrow[d] \arrow[r] & 0 \\
0 \arrow[r] & \Hom_R(L \oplus K, K) \arrow[r, hook] & \Hom_R(L \oplus K, E) \arrow[r] & \Hom_R(L \oplus K, C).  & {}
\end{tikzcd}\]
Note that the left and middle vertical arrows are isomorphisms by Lemma \ref{pushdown}.
Thus the diagram shows $V \simeq \Hom_{\Y}(\cL \oplus \cK, \cC)$ as desired.

Since $\Hom_{\Y}(\cL \oplus \cK, \cL \oplus \cC) \simeq \RHom_{\Y}(\cL \oplus \cK, \cL \oplus \cC)$ by assumption, the functorial isomorphisms
\begin{align*}
&\Phi_L \circ \RHom_{\Y}(\cL \oplus \cK, -) \\
= &\RHom_{\End_R(L \oplus K)}\left(\Hom_{\Y}(\cL \oplus \cK, \cL \oplus \cC), \RHom_{\Y}(\cL \oplus \cK, -)\right) \\
\simeq & \RHom_{\End_R(L \oplus K)}\left(\RHom_{\Y}(\cL \oplus \cK, \cL \oplus \cC), \RHom_{\Y}(\cL \oplus \cK, -)\right) \\
\simeq & \RHom_{\Y}(\cL \oplus \cC,  -)
\end{align*}
imply the commutativity of the diagram.

To prove (4), by Lemma \ref{pushdown}, it is enough to show that $V \simeq \Hom(\cL \oplus \cK, \cC) \simeq Rf_*((\cL \oplus \cK)^{*} \otimes \cC) $ is reflexive.
In order to check this, pick a prime ideal $\mathfrak{p} \in \Spec R$.
First, assume that $\mathrm{ht}(\mathfrak{p}) \leq 2$.
Since the assumption implies that the localization $f_{\mathfrak{p}}$ is an isomorphism, the stalk $V_{\mathfrak{p}}$ is free, and hence reflexive as an $R_{\mathfrak{p}}$-module.

Second, consider the case when $\mathrm{ht}(\mathfrak{p}) \geq 3$.
The computation using the Grothendieck duality gives isomorphisms
\begin{align*}
\Ext^i_R(V, R) & \simeq \Ext^i(Rf_*((\cL \oplus \cK)^{*} \otimes \cC), R) \\
& \simeq \Ext^i(((\cL \oplus \cK)^{*} \otimes \cC), f^!R) \\
& \simeq \Ext^i(\cC, \cL \oplus \cK).
\end{align*}
Applying the functor $\RHom(-, \cL \oplus \cK)$ to the sequence \eqref{es 2.21} shows that $\Ext^i_R(V, R) \simeq \Ext^i(\cC, \cL \oplus \cK) = 0$ for all $i \geq 2$.
In particular, since $h \defeq \mathrm{ht}(\mathfrak{p}) \geq 3$, the vanishings
\[ \Ext^{h}_{R_{\mathfrak{p}}}(V_{\mathfrak{p}}, R_{\mathfrak{p}}) = \Ext^{h-1}_{R_{\mathfrak{p}}}(V_{\mathfrak{p}}, R_{\mathfrak{p}}) = 0 \]
hold.
Thus \cite[Corollary 3.5.11]{bh} yields that $\mathrm{depth}_{R_{\mathfrak{p}}}(V_{\mathfrak{p}}) \geq 2$.
Now applying \cite[Tag 0AVA]{stacks} shows that $V$ is reflexive as an $R$-module.
\end{proof}

\begin{rem} \label{counterexample for 2.21}
In the proposition above, the additional assumption on the singular locus $\Sing(R)$ in (4) cannot be removed.

For example, consider the minimal resolution $f' \colon Y' \to \Spec R = \bA^2/G$, where $G = \bZ/2\bZ \subset \mathrm{SL}(2, \Bbbk)$, and the associated $n$-dimensional crepant resolution $f \colon Y \defeq Y' \times \bA^{n-2} \to \bA^2/G \times \bA^{n-2}$.
Let $\cO_Y(1)$ be the ample generator, and consider the bundles 
$\cK = \cO_Y(-1)$, $\cE \defeq \cO_Y^{\oplus 2}$, $\cC \defeq \cO_Y(1)$, $\cL \defeq \cO_Y$, and the exact sequence
\[ 0 \to \cO_Y(-1) \to \cO_Y^{\oplus 2} \to \cO_Y(1) \to 0. \]
They satisfy the general assumption in Proposition \ref{tilting mutation}.
It is not difficult to see that $K \simeq C$.
Thus $\Hom_R(L \oplus K, L \oplus C) \simeq \End_R(L \oplus K) \simeq \End_Y(\cL \oplus \cK).$
However, Since $\RHom(\cL \oplus \cK, -)$ is an equivalence,
$\cC \not\simeq \cK$ implies $\Hom(\cL \oplus \cK, \cL \oplus \cC) \not\simeq \End(\cL \oplus \cK)$.
Hence $\Hom_R(L \oplus K, L \oplus C) \not\simeq \Hom(\cL \oplus \cK, \cL \oplus \cC)$.
\end{rem}

\section{Quasi-symmetric representation and GIT quotient}
\subsection{Quasi-symmetric representations and magic windows}\label{section: qsym}
This section recalls fundamental properties of derived categories of GIT quotients arising from quasi-symmetric representations, which are established in \cite{hl-s} and \cite{svdb}.
We freely use notation from Section \ref{notation}.
\begin{dfn} 
A  representation  $X$  of $G$ is said to be  {\it quasi-symmetric} if every one dimensional subspace $L\subseteq \M_{\bR}$ satisfies $\sum_{\upbeta_i\in L}\upbeta_i=0$, where   $\upbeta_1,\cdots,\upbeta_d \in \M$ 
\label{not: upbata} \linelabel{line not: upbata}
is the set of $T$-weights of $X$.
\end{dfn}

\begin{dfn}[\cite{svdb}] Let $X$ be a quasi-symmetric $d$-dimensional representation of $G$, and let $\upbeta_1,\hdots,\upbeta_d$ denote the  $T$-weights of $X$. 
Then the associated $W$-invariant convex region is defined as
{\label{not: bisigma} \linelabel{line not: bsigma}}
\begin{equation*} 
\bsigma\defeq\left\{\sum_{i=1}^da_i\upbeta_i\middle| a_i\in[0,1]\right\}\subset \M_{\bR}. 
\end{equation*}
By the quasi-symmetry, the equality
\[ \bsigma=\left\{\sum_{i=1}^da_i\upbeta_i\middle| a_i\in[-1,0]\right\} \]
holds.
\end{dfn}

\begin{rem}
In the paper \cite{svdb}, $\Sigma$ denotes the interior of  the above $\bsigma$, and our $\bsigma$ is denoted by $\overline{\Sigma}$. In \cite{svdb, hl-s}, $\beta_i$ denotes the weights of the dual representation $X^{*}$.
\end{rem}

In the remainder of this section,  $X$ denotes a quasi-symmetric $d$-dimensional representation of $G$ with $T$-weights $\upbeta_1,\hdots,\upbeta_d\in \M$, and we always assume that $\bsigma$ spans $\M_{\bR}$.
The  assumption  that $\bsigma$ spans $\M_{\bR}$ is equivalent to the assumption that the $T$-action on $X$ has generically finite stabilizers.  Let $\scrA$ denote  the real hyperplane arrangement  in $\M_{\bR}$ defined to be the set of codimension one linear subspaces  that are parallel to some facet of $\bsigma$.  
Set
\label{not: M_R^W} \linelabel{ line not: M_R^W}
\[
(\M_{\bR}^{W})_{\rm gen}\defeq \M_{\bR}^W\backslash\bigcup_{A\in\scrA}\bigl(A\cap \M_{\bR}^{W}\bigr).
\]
A $W$-invariant element $\ell\in \M_{\bR}^W$ is said to be {\it generic} if $\ell\in (\M_{\bR}^W)_{\rm gen}$. 
Since $\M^W$ is canonically isomorphic to $\Pic BG$, an element $\ell\in \M^W$ defines a character of $G$, which is denoted by $\chi_{\ell}\colon G\to \bG_m$. Thus $\ell$ induces a $G$-equivariant invertible sheaf $\cO(\ell)$, and 
the $G$-invariant part $\Gamma(X,\cO(\ell))^G$ of its global section is the set of $\chi_{\ell}$-semi-invariant regular functions.
Here, for a character $\chi\colon G\to \bG_m$, a regular function $f$  is said to be {\it $\chi$-semi-invariant} if $f(gx)=\chi(g)f(x)$ for all $g\in G$ and $x\in X$.  
For an element $\ell\in \M^W$, 
let $X^{\rm ss}(\ell)$ be the semistable locus with respect to $\ell$, which is defined by 
\[
X^{\rm ss}(\ell)\defeq\left\{ x\in X\relmiddle| \exists k>0 \mbox{ and }\exists f\in \Gamma(X,\cO(k\ell))^G \mbox{ such that } f(x)\neq0 \right\},
\] 
and then it associates the GIT quotient stack $[X^{\rm ss}(\ell)/G]$.
\begin{rem}\label{HM criterion}
The Hilbert--Mumford numerical criterion shows that $x\in X^{\rm ss}(\ell)$ if and only if $\l\ell,\lambda\r\geq0$ for all one-parameter subgroup $\lambda\colon \bG_m\to G$ such that $\lim_{t\to0}\lambda(t)x$ exists.
\end{rem}
Assuming that $\bsigma$ spans $\M_{\bR}$ gives the following.

\begin{prop}[{\cite[Proposition 2.1]{hl-s}}]\label{dm stack}
If $\ell\in \M_{\bR}^W$ is generic, then $[X^{\rm ss}(\ell)/G]$ is a Deligne--Mumford stack.
\end{prop}

\begin{dfn} \label{def bnabla}
Put $\bL\defeq [X^{*}]-[\g^{*}]\in K_0(\rep T)$. 
For a one-parameter subgroup $\lambda\in \N$ and a representation $V\in \rep T$, 
write $V^{\lambda>0}\subseteq V$ for the direct sum of $T$-weights $\chi$ in $V$ such that $\l\chi,\lambda\r>0$. This extends to a linear map $(-)^{\lambda>0}\colon K_0(\rep T)\to K_0(\rep T)$, and 
using this, set $\eta_{\lambda}\defeq\l\bL^{\lambda>0},\lambda\r$ and
{\label{not: bnabla} \linelabel{line not: bnabla}}
\[
\bnabla\defeq\left\{\chi\in\M_{\bR}\relmiddle| -\frac{\eta_{\lambda}}{2}\leq \l\chi,\lambda\r\leq \frac{\eta_{\lambda}}{2} \mbox{ for all $\lambda\in \N$}\right\}\subset\M_{\bR}.
\]
Let $\scrH$ 
\label{not: scrH} \linelabel{line not: scrH}
denote the set of  hyperplanes  of the form $m+B$ for some $m\in \M$ and some hyperplane $B$ such that $B\cap\bnabla$ is a facet of $\bnabla$.
\end{dfn}

For $H\in\scrH$ that does not contain $\M_{\bR}^W$, set
$H^W\defeq H\cap \M_{\bR}^W$. 
If $(\M_{\bR}^W)_{\rm gen}\neq \emptyset$, a $(\M^W)$-periodic hyperplane arrangement $\scrH^W$ in $\M_{\bR}^W$ is defined by \label{not: H^W} \linelabel{line not: H^W}
\[
\scrH^W\defeq \{H^W\mid H\in\scrH \},
\]
and the complement $\M_{\bR}^W \backslash \bigcup_{H^W \in \scrH^W} H^W$ is simply denoted by  $\M_{\bR}^W\backslash\scrH^W$. 
A connected component of $\M_{\bR}^W\backslash\scrH^W$ is called a {\it chamber} of $\M_{\bR}^W\backslash\scrH^W$. We say that two elements $\delta,\delta'\in\M_{\bR}^W\backslash\scrH^W$ are {\it equivalent} if they lie in the same chamber. 
The symbol $[\delta]$ denotes the equivalence class of $\delta \in\M_{\bR}^W\backslash\scrH^W$.

\begin{prop}\label{nabla sigma} Notation is same as above. Assume that $(\M_{\bR}^W)_{\rm gen}\neq \emptyset$.
\begin{itemize}
\item[$(1)$] 
$\bnabla\cap\M^+_{\bR}=(-\rho+(1/2)\bsigma)\cap\M^+_{\bR}.$  In particular,  
\[\bnabla=\bigcup_{w\in W}w\left((-\rho+(1/2)\bsigma)\cap\M^+_{\bR}\right).\]
\item[$(2)$] For $\delta\in\M_{\bR}^W$, $\delta\in \M_{\bR}^W\backslash\scrH^W$ if and only if $\partial(\delta+\bnabla)\cap \M=\emptyset$.
\end{itemize}
\begin{proof}
(1) This follows from \cite[Lemma 2.9]{hl-s} (see also \cite[proof of Corollary 2.10]{hl-s}).\vspace{1mm}\\
(2) This is \cite[Lemma 3.3]{hl-s}. 
\end{proof}
\end{prop}

For any  $\delta\in \M^W_{\bR}$, the corresponding finite subset $\scrC_{\delta}\subset \M^+$ of dominant weights is defined by \linelabel{line x delta}
\begin{equation}\label{x delta}
\scrC_{\delta}\defeq (\delta+\bnabla)\cap \M^+.
\end{equation}
\begin{lem}\label{independence}
Let $\delta, \tilde{\delta}\in \M_{\bR}^W\backslash\scrH^W$ with $[\delta]=[\tilde{\delta}]$.
\begin{itemize}
\item[(1)] $\scrC_{\delta}=\scrC_{\tilde{\delta}}$.
\item[(2)] $(\delta+(1/2)\bsigma)\cap (\M^++\rho)=(\tilde{\delta}+(1/2)\bsigma)\cap (\M^++\rho)$.
\end{itemize}
\begin{proof}
(1) This follows from Proposition \ref{nabla sigma}\,(2).\vspace{1mm}\\
(2) It is enough to prove $(\delta+(1/2)\bsigma)\cap (\M^++\rho)\subseteq(\tilde{\delta}+(1/2)\bsigma)\cap (\M^++\rho)$. Let $\chi\in \M^+$ such that $\chi+\rho= \delta+(1/2)\sigma$ for some $\sigma\in\bsigma$. By Proposition \ref{nabla sigma}\,(1),  $\chi=\delta-\rho+(1/2)\sigma\in \scrC_{\delta}$. By (1) there exists $\tilde{\sigma}\in \bsigma$ such that $\chi=\tilde{\delta}-\rho+(1/2)\tilde{\sigma}$. Then $\chi+\rho=\tilde{\delta}+(1/2)\tilde{\sigma}\in (\tilde{\delta}+(1/2)\bsigma)\cap (\M^++\rho)$. This shows that $(\delta+(1/2)\bsigma)\cap (\M^++\rho)\subseteq(\tilde{\delta}+(1/2)\bsigma)\cap (\M^++\rho)$. 
\end{proof}
\end{lem}

\begin{figure}
\centering
\begin{tikzpicture}[baseline=(current bounding box.north), x=5mm, y=5mm]

          \draw[fill=gray!20, line width=1pt]   (5,5) --(5,1) --(3,-3) -- (-1, -5) --(-5, -5) --cycle;
          \draw[line width=1pt]   (-5,-5) --(-5,-1) --(-3,3) -- (1, 5) --(5, 5);
          \draw[line width=1pt]   (-6,-6) -- (6,6);
                  
           \draw[step=.5cm, gray,very thin] (-6,-6) grid (6,6); 
           
               \foreach \i in {-5,..., 5}{
      \foreach \j in {-5,...,5}{ 
      \ifodd\j\relax\else 
      \ifodd\i\relax\else      
        \ifnum\numexpr \i - \j > -1 
        \ifnum\numexpr \i - \j < 5

        \draw (\i,\j) node {$\bullet$};
      \fi
       \fi
       \fi
       \fi
       }
              }
              
        \foreach \i in {-5,..., 5}{
     
      \ifodd\i

        \draw (\i,\i) node {\textcolor{red}{$\bullet$}};
      \fi
       }

	\draw (0 ,0) node[above=6pt, fill=white, inner sep=1pt]  {$\delta$};
	\draw (4 ,-4) node[fill=white, inner sep=1pt]  {$\bullet \in \scrC_\delta$}; 
	\draw (4 ,-5) node[fill=white, inner sep=1pt]  {$\textcolor{red}{\bullet} \in \scrH^W$};

          \end{tikzpicture}
          
          \caption{A picture for $\scrC_\delta$ when $G = \mathrm{GL}_2(\Bbbk)$ and $X = \mathrm{Sym}^3(\Bbbk^2) \oplus \mathrm{Sym}^3(\Bbbk^2)^*$.  }
          
\end{figure}

Following \cite{hl-s},
for each $\delta\in\M^W_{\bR}\backslash\scrH^W$,
the corresponding {\it magic window}
\label{not: magic window} \linelabel{line not: magic window}
\[
\scrM(\delta+\bnabla)\subset\Db(\coh[X/G])
\]
is defined to be 
the thick subcategory of $\Db(\coh[X/G])$ generated by the objects of the form $V(\chi)\otimes \cO_X$ with $\chi\in(\delta+\bnabla)\cap\M^+$. 
Note that Proposition \ref{nabla sigma} implies
\begin{equation}\label{window generation}
\scrM(\delta+\bnabla)=\left\l V_X(\chi)\relmiddle| \chi \in \left(\delta-\rho+(1/2)\bsigma\right)\cap\M^+\right\r.
\end{equation}
The following result says that each magic window can be identified  with the derived category of the  GIT quotient stack $[X^{\rm ss}(\ell)/G]$ for a generic $\ell$.

\begin{thm}[{\cite[Theorem 3.2]{hl-s}}]\label{hl-s}
Let $\ell\in(\M_{\bR}^{W})_{\rm gen}$ and $\delta\in \M^W_{\bR}\backslash\scrH^W$. Then the restriction functor 
\[
\res\defeq i^*\colon \scrM(\delta+\bnabla)\to \Db(\coh [X^{\rm ss}(\ell)/G])
\]
is an equivalence, where $i\colon X^{\rm ss}(\ell)\hookto X$ is a natural open immersion. 
\end{thm}

\subsection{Fundamental group action on magic windows.}  
The present section recalls fundamental group actions of complexified K\"ahler moduli spaces  of quasi-symmetric representations constructed in \cite[Section 6]{hl-s}. Let us keep the 
same notation and assumption for the previous section. 
Throughout this section, assume that $(\M_{\bR}^W)_{\rm gen}\neq \emptyset$.

Let $\delta,\delta'\in \M_{\bR}^W\backslash \scrH^W$ be two elements. A {\it separating hyperplane} of $\delta$ and $\delta'$ is a hyperplane in $\scrH^W$ such that it meets the line segment with end points $\delta$ and $\delta'$.
Set \label{not: sep hyp set} \linelabel{line not: sep hyp set}
\[
\scrH(\delta,\delta')\defeq\{H^W\in \scrH^W\mid\mbox{$H^W$ is a separating hyperplane of $\delta$ and $\delta'$}\}\subset \scrH^W.
\]
The number of elements in $\scrH(\delta,\delta')$ is called the {\it distance} between $\delta$ and $\delta'$, and it is denoted by $d(\delta,\delta')$. 
We say that  $\delta$ and $\delta'$   are {\it adjacent} if $d(\delta,\delta')=1$.

Let ${\rm Arr}(\scrH^W)$ denote the set of labeled arrows 
\[
[\delta]\xrightarrow{\ell}[\delta']
\]
such that $\delta$ and $\delta'$ are elements in $(\M_{\bR}^W)\backslash \scrH^W$  and $\ell\in (\M_{\bR}^W)_{\rm gen}$. A {\it path} in $\M_{\bR}^W\backslash \scrH^W$ is a finite sequence 
\begin{equation}\label{path}
p\colon 
[\delta_0]\xrightarrow{\ell_1}[\delta_1]\xrightarrow{\ell_2}\cdots\xrightarrow{\ell_m}[\delta_m]
\end{equation}
of labeled arrows in ${\rm Arr}(\scrH^W)$. 
The equivalence classes $[\delta_0]$ and $[\delta_m]$ in \eqref{path} are called the {\it source} and the {\it target} of $p$, 
and denoted by $s(p)\defeq[\delta_0]$ and $t(p)\defeq[\delta_m]$, respectively. 
If $s(p)=[\delta]$ and $t(p)=[\delta']$, the path $p$ is simply described as $p\colon [\delta]\to[\delta']$.
If $p\colon[\delta]\to [\delta']$ and $q\colon [\delta']\to[\delta'']$ are two  paths with $t(p)=s(q)$,  the {\it composition} $q\circ p$ of $p$ and $q$ is naturally defined. 

Let $\Gamma(\scrH^W)$\label{not: gamma gpd} \linelabel{line not: gamma gpd}
denote the groupoid whose objects are  equivalence classes $[\delta]$ with $\delta\in (\M_{\bR}^W)\backslash \scrH^W$, and morphisms from $[\delta]$ to $[\delta']$ are paths $p\colon [\delta]\to[\delta']$ modulo the equivalence relation generated by the following equivalences:
\begin{itemize}
\item[(R1)] A labeled arrow $[\delta]\xrightarrow{\ell}[\delta]$ is equivalent to the identity morphism on $[\delta]$ for any $\ell$.
\item[(R2)] The composition of labeled arrows $[\delta]\xrightarrow{\ell}[\delta']$ and $[\delta']\xrightarrow{\ell}[\delta'']$ with a  common label $\ell$ is equivalent to $[\delta]\xrightarrow{\ell}[\delta'']$.
\item[(R3)] Labeled arrows $[\delta]\xrightarrow{\ell}[\delta']$ and $[\delta]\xrightarrow{\ell'}[\delta']$ are equivalent if $\ell$ has the same orientation as $\ell'$ with respect to every hyperplane in $\scrH(\delta,\delta')$.
\end{itemize}
The complexification of the real hyperplane arrangement $\scrH^W$ is denoted by $\scrH_{\bC}^W$.

\begin{prop}[{\cite[Proposition 6.2]{hl-s}}]\label{groupoid equiv}
There is an equivalence of groupoids
\[
\Gamma(\scrH^W)\cong \Pi_1(\M_{\bC}^W\backslash \scrH^W_{\bC}),
\]
where $\Pi_1(\M_{\bC}^W\backslash \scrH^W_{\bC})$ denotes the fundamental groupoid of  $\M_{\bC}^W\backslash \scrH^W_{\bC}$.
\end{prop}

In what follows, we recall the fact that the groupoid $\Gamma(\scrH^W)$ can be identified with the groupoid completion of positive paths in $(\M_{\bR}^W)\backslash \scrH^W$.
The precise statement for this identification will be given using the morphism in (\ref{equiv of path}).

A labeled arrow $[\delta]\xrightarrow{\ell}[\delta']\in {\rm Arr}(\scrH^W)$ is said to be {\it positive} if  $[\delta]\neq[\delta']$ and $\ell$ has the same orientation as $\delta'-\delta$ with respect to every hyperplane in $\scrH(\delta,\delta')$,
and the subset of ${\rm Arr}(\scrH^W)$ consisting of all positive labeled arrows is denoted by ${\rm Arr}^+(\scrH^W)$. 
A path $p\colon 
[\delta_0]\xrightarrow{\ell_1}[\delta_1]\xrightarrow{\ell_2}\cdots\xrightarrow{\ell_m}[\delta_m]$  is said to be {\it positive} if each labeled arrow $[\delta_{i-1}]\xrightarrow{\ell_i}[\delta_i]$ is positive.

\begin{lem}\label{separating hyperplane} Let $\delta, \delta',\delta''\in \M_{\bR}^W\backslash \scrH^W$.
There is an equality
\[
\scrH(\delta,\delta'')=\bigl(\scrH(\delta,\delta')\cup\scrH(\delta',\delta'')\bigr)\backslash\bigl(\scrH(\delta,\delta')\cap\scrH(\delta',\delta'')\bigr).
\] 
In particular, it holds that $d(\delta,\delta'')\leq d(\delta,\delta')+d(\delta',\delta'')$, with equality if and only if $\scrH(\delta,\delta'')=\scrH(\delta,\delta')\sqcup\scrH(\delta',\delta'')$.
\begin{proof}
If $H\in\scrH(\delta,\delta'')$, then $H$ separates $\delta$ and $\delta''$. If $H\not\in \scrH(\delta,\delta')\cup\scrH(\delta',\delta'')$, then $\delta$, $\delta'$ and $\delta''$ lie in the same half space separated by $H$, which is a contradiction. If $H\in \scrH(\delta,\delta')\cap\scrH(\delta',\delta'')$, then $\delta$ and $\delta''$ must lie in the same half space separated by $H$, which is also a contradiction. This proves the inclusion $(\subseteq)$ of the first assertion. Conversely, let $H\in \bigl(\scrH(\delta,\delta')\cup\scrH(\delta',\delta'')\bigr)\backslash\bigl(\scrH(\delta,\delta')\cap\scrH(\delta',\delta'')\bigr)$. Then, we may assume that $H\in \scrH(\delta,\delta')\backslash\scrH(\delta',\delta'')$, and so  $H$ separates $\delta$ and $\delta'$, but $\delta'$ and $\delta''$ lie in the same half space separated by $H$. This implies that $\delta$ and $\delta''$ are separated by $H$, which completes the proof of the first assertion.  The second assertion follows from the first one. 
\end{proof}
\end{lem}

A positive path $p\colon 
[\delta_0]\xrightarrow{\ell_1}[\delta_1]\xrightarrow{\ell_2}\cdots\xrightarrow{\ell_m}[\delta_m]$ in $\M_{\bR}^W\backslash \scrH^W$ is said to be {\it minimal} if 
$d(\delta_0,\delta_m)=\sum_{i=1}^md(\delta_{i-1},\delta_i)$. The following is elementary.

\begin{prop}\label{minimality}
Let $p\colon 
[\delta_0]\xrightarrow{\ell_1}[\delta_1]\xrightarrow{\ell_2}\cdots\xrightarrow{\ell_m}[\delta_m]$ be a positive path in $\M_{\bR}^W\backslash \scrH^W$.
Then the following are equivalent.

\begin{itemize}
\item[(1)] $p$ is minimal
\item[(2)] $\scrH(\delta_0,\delta_m)=\bigsqcup_{i=0}^{m-1}\scrH(\delta_i,\delta_{i+1})$.
\item[(3)] $\scrH(\delta_i,\delta_{i+1})\cap\scrH(\delta_j,\delta_{j+1})=\emptyset$ for any $i\neq j$.
\item[(4)] Each $\ell_i$ has the same orientation as $\delta_m-\delta_0$ with respect to  hyperplanes in $\scrH(\delta_{i-1},\delta_i)$.
\end{itemize}
\begin{proof}
By Lemma \ref{separating hyperplane},  an inclusion 
\begin{equation}\label{inclusion 2.11}
\scrH(\delta_0,\delta_m)\subseteq\bigcup_{i=0}^{m-1}\scrH(\delta_i,\delta_{i+1})
\end{equation}
holds, which implies (1)$\Leftrightarrow$(2). (2)$\Rightarrow$(3) is obvious.

Assume that (3) holds, and let $H\in \scrH(\delta_i,\delta_{i+1})$. Denote by $H_+$ the half space separated by $H$ such that $\delta_{i+1}\in H_+$. Then the assumption (3) implies that $\delta_j\in H_+$ if and only if $ i+1\leq j\leq m$. In particular, $\delta_0$ and $\delta_m$ are separated by $H$, and so  $H\in  \scrH(\delta_0,\delta_m)$. This and the inclusion \eqref{inclusion 2.11} shows that (2) holds. Moreover, $\delta_{i+1}-\delta_i$ and $\delta_{m}-\delta_0$ have the same orientation with respect to $H$. This shows that (4) holds, since $p$ is a positive path.

It only remains to prove (4)$\Rightarrow$(3). Assume (4) holds, but (3) does not hold. Then there is an element $H\in \scrH(\delta_i,\delta_{i+1})\cap \scrH(\delta_j,\delta_{j+1})$ for some $i\neq j$. 
By replacing $j$ if necessary, we may assume that $H\notin \scrH(\delta_k,\delta_{k+1})$ for all $i<k<j$.  Then $\delta_{i+1}-\delta_i$ and $\delta_{j+1}-\delta_j$ has the opposite orientation with respect to $H$. This contradicts to (4). 
\end{proof}
\end{prop}

Consider the category $\dsG^+(\scrH^W)$ whose objects are the same as $\Gamma(\scrH^W)$, and whose morphisms are identity maps or positive paths  modulo  the smallest equivalence relation identifying minimal positive paths with the same source and target. 
Let $\dsG(\scrH^W)$ be the groupoid completion of the category $\dsG^+(\scrH^W)$, that is, it is the localization of $\dsG^+(\scrH^W)$ by the set of all morphisms. If a path $p\colon 
[\delta_0]\xrightarrow{\ell_1}[\delta_1]\xrightarrow{\ell_2}\cdots\xrightarrow{\ell_m}[\delta_m]$ is minimal and $\ell$ is a generic element with the same orientation as $\delta_m-\delta_0$ with respect to every hyperplane in $\scrH(\delta_0,\delta_m)$, then $p$ is equivalent to the labeled arrow $[\delta_0]\xrightarrow{\ell}[\delta_m]$ in $\Gamma(\scrH^W)$ by Proposition \ref{minimality} and the equivalence relations (R2) and (R3). 
This implies that minimal positive paths with the same source and target are equivalent in $\Gamma(\scrH^W)$. Thus there is a natural functor
$\dsG^+(\scrH^W)\to\Gamma(\scrH^W)$, which induces a functor 
\begin{equation}\label{equiv of path}
\dsG(\scrH^W)\to \Gamma(\scrH^W)
\end{equation}
by the universal property of the localization. By the argument in the proof of \cite[Proposition 6.2]{hl-s}, the functor \eqref{equiv of path} is an equivalence.

By the definition of the hyperplane arrangement $\scrH^W$, it is stable under the translation by the lattice $\M^W$. Thus the lattice $\M^W$ acts on $\M^W_{\bR}\backslash\scrH^W$ by translations. We define a groupoid $\widetilde{\Gamma}(\scrH^W)$, whose objects are the same as ${\Gamma}(\scrH^W)$, by formally adding an arrow
\[
[\delta]\rightsquigarrow[\delta+m]
\]
for every $\delta\in \M^W_{\bR}\backslash\scrH^W$ and every $m\in \M^W$, where we impose the following  additional equivalence relations:
\begin{itemize}
\item[(R4)] The composition $[\delta]\xrightarrow{\ell}[\delta']\rightsquigarrow[\delta'+m]$ is equivalent to the composition $[\delta]\rightsquigarrow[\delta+m]\xrightarrow{\ell}[\delta'+m]$ for any $\delta, \delta'\in \M^W_{\bR}\backslash\scrH^W$, $\ell\in (\M_{\bR}^W)_{\rm gen}$ and $m\in \M^W$. 
\item[(R5)] The composition $[\delta]\rightsquigarrow[\delta+m]\rightsquigarrow[\delta+m+m']$ is equivalent to $[\delta]\rightsquigarrow[\delta+m+m']$ for any $\delta\in \M^W_{\bR}\backslash\scrH^W$ and $m,m'\in \M^W$.
\end{itemize}
The following is an extension of Proposition \ref{groupoid equiv}.

\begin{prop}[{\cite[Proposition 6.5]{hl-s}}]\label{ext groupoid equiv}
There is an equivalence
\[
\widetilde{\Gamma}(\scrH^W)\simto \Pi_1\left(\bigl(\M_{\bC}^W\backslash \scrH^W_{\bC}\bigr)\relmiddle/\M^W\right)
\]
extending the equivalence in Proposition \ref{groupoid equiv}.
\end{prop}

Now we are ready to define a groupoid action on magic windows. For $[\delta]\in {\Gamma}(\scrH^W)$, put
$\uprho([\delta])\defeq\scrM(\delta+\bnabla)$, and for a positive labeled arrrow $[\delta]\xrightarrow{\ell}[\delta']$, the corresponding functor is defined by
\[
\uprho\left([\delta]\xrightarrow{\ell}[\delta']\right)
\defeq \scrM(\delta+\bnabla)\xrightarrow{\res} \Db(\coh [X^{\rm ss}(\ell)/G])\xrightarrow{\res^{-1}}\scrM(\delta'+\bnabla).
\]
By \cite[Proposition 6.6]{hl-s}, this defines a groupoid action of ${\Gamma}(\scrH^W)$ on the groupoid of magic windows. In particular, by Proposition \ref{groupoid equiv}, this action defines a group action 
\begin{equation}\label{fund action}
\uprho\colon \pi_1(\M_{\bC}^W\backslash \scrH^W_{\bC})\to \Auteq \bigl(\scrM(\delta+\bnabla)\bigr)
\end{equation}
for any $\delta\in\M_{\bC}^W\backslash \scrH^W_{\bC}$. 
Furthermore, by \cite[Proposition 6.6]{hl-s},  
this groupoid action can be extended to a groupoid action  of $\widetilde{\Gamma}(\scrH^W)$ by setting
\[
\widetilde{\uprho}\bigl([\delta]\rightsquigarrow[\delta+m]\bigr)\defeq \scrM(\delta+\bnabla)\xrightarrow{(-)\otimes\cO(m)}\scrM(\delta+m+\bnabla),
\]
where $\cO(m)$ denotes the corresponding $G$-equivariant line bundle on $X$. By Proposition \ref{ext groupoid equiv}, this  groupoid action defines a group action
\begin{equation}\label{ext fund action}
\widetilde{\uprho}\colon\pi_1\left(\bigl(\M_{\bC}^W\backslash \scrH^W_{\bC}\bigr)\relmiddle/\M^W\right)\to \Auteq\bigl(\scrM(\delta+\bnabla)\bigr)
\end{equation}
for any $\delta \in\M_{\bC}^W\backslash \scrH^W_{\bC}$, which extends  \eqref{fund action}.

\subsection{Combinatorial structures of \texorpdfstring{$(1/2)\bsigma$}{TEXT} and \texorpdfstring{$\bnabla$}{TEXT}} \label{section combinatiorics}
This section studies combinatorial structures of the polytopes $(1/2)\bsigma$ and $\bnabla$, which are necessary for later use. Throughout this section, assume that $\bsigma$ spans $\M_{\bR}$ and $(\M_{\bR}^W)_{\rm gen}\neq \emptyset$.

Recall that a {\it supporting hyperplane} of a convex set $A\subset \bR^n$ is a  hyperplane $B\subset \bR^n$ such that $A$ is contained in one of  the two  closed half spaces bounded by $B$ and that $B\cap \partial A\neq\emptyset$. 
Let $\scrS(A)$ denote the set of supporting hyperplanes of $A$. 
For  a convex polytope $P$,  we set 
\label{not: B(P) and F(P)} \linelabel{line not: B(P) and F(P)}
\begin{align*}
\scrB(P)&\defeq\{ B\in\scrS(P)\mid \mbox{$B\cap P$ is a facet of $P$}\}, \\
\scrF^k(P)&\defeq\{\mbox{codimension $k$ face of $P$}\},
\end{align*}
and $\scrF(P)\defeq\bigcup_{k=1}^{n}\scrF^k(P)$.
\label{not: F(P)} \linelabel{line not: F(P)}
A polytope $P\subset \bR^n$ is {\it centrally symmetric} if it has a point $c\in P$, called the {\it center}, such that for $a\in \bR^n$, $c+a\in P$ if and only if $c-a\in P$, or equivalently,  for $a\in \bR^n$, $a\in P$ if and only if $-a+2c\in P$. 
For $a\in P$, put $a^*\defeq -a+2c$, and this is called the dual point of $a$ with respect to $P$. 
Let $P\subset \bR^n$ be a centrally symmetric convex polytope with center $c\in \bR^n$. 
Then, for $B\in\scrB(P)$ and $F\in\scrF^k(P)$, the set of dual points defines the duals
\linelabel{line not: dual facet}
\begin{align}\label{dual facet}
B^*&\defeq-B+2c\in\scrB(P), ~\text{and}\\
F^*&\defeq-F+2c\in \scrF^k(P).
\end{align}
\begin{rem}\label{rem:central}
For any $\delta\in \M_{\bR}$, the convex polytope $\delta+\bnabla$ is centrally symmetric with center $\delta$, and so a hyperplane $B\subset \M_{\bR}$ lies in $\scrB(\delta+\bnabla)$ if and only if $B^*$ lies in $\scrB(\delta+\bnabla)$. 
\end{rem}
Let $P\subset \M_{\bR}$ be a centrally symmetric convex polytope.
For $B\in\scrS(P)$, let $B_+$  (resp. $B_-$) denote the closed half space bounded by $B$ such that $P\subset B_+$ (resp. $P\not\subset B_{-}$). 
The half-space $B_+$ is called a {\it supporting half-space} of $P$.  For $B\in\scrB(P)$, the length one normal vector to $B$ with the same orientation as $b_+-b_-$ for some $b_+\in B_+\backslash B$ and $b_-\in B_-\backslash B$ is denoted by $\lambda_B\in\N_{\bR}$. 
For $F\in\scrF(P)$, let $\scrB_F(P)\subset \scrB(P)$ be the subset consisting of $B\in \scrB(P)$ that contains $F$, and using this, set 
\begin{equation}\label{lambdas}
\scrL_{F}\defeq \{\lambda_B\mid B\in \scrB_F(P)\}\subset \N_{\bR}.
\end{equation}

\begin{dfn}\label{upbeta}
Let $F\in\scrF(\delta+(1/2)\bsigma)$  for some $\delta\in\M_{\bR}$.
We define \linelabel{line not: W_F}
 \begin{align} \label{not: W_F}
 \scrW^+_F&\defeq\left\{\upbeta_i\relmiddle|\mbox{$\l\upbeta_i,\lambda\r>0$ for some $\lambda\in\scrL_F$}\right\},\\
  \scrW^0_F&\defeq\left\{\upbeta_i\relmiddle|\mbox{$\l\upbeta_i,\lambda\r=0$ for all $\lambda\in\scrL_F$}\right\}
 \end{align}
 and put \linelabel{line not: upbeta_F^+}
\begin{equation} \label{not: upbeta_F^+}
\upbeta_F^+\defeq \sum_{\upbeta_i\in \scrW_{F}^+}\upbeta_i.
\end{equation}
Similarly, put $\scrW_{F}^-\defeq\left\{\upbeta_i\relmiddle|\mbox{$\l\upbeta_i,\lambda\r<0$ for some $\lambda\in\scrL_F$}\right\}$ and  $\upbeta_{F}^-\defeq\sum_{\upbeta_i\in \scrW_{F}^-}\upbeta_i$. 
Moreover, denote by
\label{not: d_F^+} \linelabel{line not: d_F^+}
\begin{align}
d_F^+&\defeq \#\scrW_F^+\\
d_F^-&\defeq \#\scrW_F^-
\end{align}
 the number of elements in each set $\scrW_F^{\pm}$.
\end{dfn}


\begin{lem}\label{property of beta} Let $F\in\scrF(\delta+(1/2)\bsigma)$ and $w\in W$. 
\begin{itemize} 
\item[$(1)$]  $\upbeta_{F^*}^+=-\upbeta_F^+$.
\item[$(2)$] The action $w\colon \M_{\bR}\simto\M_{\bR}$ of $w$ restricts to the bijection $\scrW_{F}^+\simto \scrW_{w(F)}^+$.
\item[$(3)$]  $\upbeta_{w(F)}^+=w(\upbeta_F^+)$.
\end{itemize}
\begin{proof}
(1) Since $\scrL_{F^*}=-\scrL_F$, it holds that $\scrW_{F^*}^+=\scrW_F^-$. 
Let $L_1,\hdots, L_k$ be the complete set of  one dimensional subspaces in $\M_{\bR}$  containing some $\upbeta_i\in \scrW_F^+$.  
Then $\upbeta_F^+=\sum_{i=1}^k(\sum_{\upbeta_j\in (L_i\cap\scrW_F^+)}\upbeta_j)$. 
The assumption that $X$ is quasi-symmetric gives
\[
\sum_{\upbeta_j\in (L_i\cap\scrW_F^+)}\upbeta_j=-\left(\sum_{\upbeta_j\in (L_i\cap\scrW_F^-)}\upbeta_j\right).
\] 
Hence $\upbeta_{F^*}^+=\sum_{i=1}^k(\sum_{\upbeta_j\in (L_i\cap\scrW_{F^*}^+)}\upbeta_j)=-\sum_{i=1}^k(\sum_{\upbeta_j\in (L_i\cap\scrW_{F}^+)}\upbeta_j)=-\upbeta_F^+$.\vspace{1mm}\\
(2) Since the pairing $\l-,-\r$ is $W$-invariant, the equality $\scrL_{w(F)}=\{w(\lambda)\,|\,\lambda\in\scrL_F\}$ holds, and the result follows from the following:
\begin{align*}
w(\upbeta_i)\in\scrW_{w(F)}^+& \Longleftrightarrow \l w(\upbeta_i), \lambda'\r>0 \mbox{ for some } \lambda'\in\scrL_{w(F)}\\
& \Longleftrightarrow \l w(\upbeta_i), w(\lambda)\r>0 \mbox{ for some } \lambda\in\scrL_{F}\\
& \Longleftrightarrow \l \upbeta_i, \lambda\r>0 \mbox{ for some } \lambda\in\scrL_{F}\\
& \Longleftrightarrow \upbeta_i\in\scrW_F^+.
\end{align*}
(3) This follows from (2).
\end{proof}
\end{lem}

The following is elementary, but we give a proof for the convenience of the  reader.

\begin{lem}\label{reflection}
 Let $\delta\in \M_{\bR}$ and $F\in \scrF^k(\delta+(1/2)\bsigma)$.  
 \begin{itemize}
 \item[$(1)$]  
 \[
 F=\left\{\delta-\frac{1}{2}\upbeta_{F}^++\sum_{\upbeta_{i_j}\in\scrW_F^0}a_j\upbeta_{i_j}\relmiddle| -\frac{1}{2}\leq a_j\leq 0\right\}.
 \]
 \item[$(2)$] 
$
F^*=F+\upbeta_{F}^+.
$
 \end{itemize}
\begin{proof}
We may assume that $\delta$ is the origin of $\M_{\bR}$.\vspace{1mm}\\
(1) Since $F$ is of codimension $k$, there exist $k$ facets $F_1,\hdots,F_k\in \scrF^1((1/2)\bsigma)$ such that $F=\cap_{i=1}^k F_i$. It is easy to see that  
\begin{align*}
F_i&=\left\{-\frac{1}{2}\upbeta_{F_i}^++\sum_{\upbeta_{i_j}\in\scrW_{F_i}^0}a_j\upbeta_{i_j}\relmiddle| -\frac{1}{2}\leq a_j\leq 0\right\}\\
&=\left\{\sum_{i=1}^da_i\upbeta_i\relmiddle| -\frac{1}{2}\leq a_i\leq 0, \mbox{$\displaystyle a_i=-\frac{1}{2}$ if $\upbeta_i\in \scrW_{F_i}^+$ and $a_i=0$ if $\upbeta_i\in \scrW_{F_i}^-$}\right\}.
\end{align*}
Therefore, 
\begin{align*}
F&=F_1\cap\cdots\cap F_k\\
&=\left\{\sum_{i=1}^da_i\upbeta_i\relmiddle| -\frac{1}{2}\leq a_i\leq 0, \mbox{$\displaystyle a_i=-\frac{1}{2}$ if $\displaystyle\upbeta_i\in \bigcup_{i=1}^k\scrW_{F_i}^+$ and $a_i=0$ if $\displaystyle\upbeta_i\in \bigcup_{i=1}^k\scrW_{F_i}^-$}\right\}.
\end{align*}
Since $\scrW_{F}^{\pm}=\bigcup_{i=1}^k\scrW_{F_i}^{\pm}$, 
the equality above yields
$F=\left\{-\frac{1}{2}\upbeta_{F}^++\sum_{\upbeta_{i_j}\in\scrW_F^0}a_j\upbeta_{i_j}\relmiddle| -\frac{1}{2}\leq a_j\leq 0\right\}$.\vspace{1mm}\\
(2)  It is enough to prove 
\[
-F=F+\upbeta_{F}^+.
\]
The assumption that $X$ is quasi-symmetric provides $\upbeta_{(-F)}^+=-\upbeta_F^+$. 
Since $F$ and $-F$ are parallel, it also holds that $\scrW_F^0=\scrW_{(-F)}^0$. Thus  
\[
-F=\left\{\frac{1}{2}\upbeta_{F}^++\sum_{\upbeta_{i_j}\in\scrW_F^0}a_j\upbeta_{i_j}\relmiddle| -\frac{1}{2}\leq a_j\leq 0\right\}=F+\upbeta_F^+.
\]
\end{proof}
\end{lem}

Let $P$ be a $W$-invariant convex polytope in $\M_{\bR}$. A face $F\in \scrF(P)$  is said to be {\it dominant}, if $F\subset \M_{\bR}^+$.  
Denote the sets of dominant faces by  \label{not: dom faces} \linelabel{line not: dom faces}
\begin{align*}
 \scrF^k_+(P)&\defeq \{F\in \scrF^k(P)\mid \mbox{$F$ is dominant}\}\\
 \scrF_+(P)&\defeq \{F\in \scrF(P)\mid \mbox{$F$ is dominant}\}.
\end{align*}
For $\delta\in \M_{\bR}^W$ and $\chi\in \partial(\delta+\bnabla)\cap\M_{\bR}^+$, let 
\[
F_{\chi}(\delta)\in \scrF_+(\delta+(1/2)\bsigma)
\]
denote the maximal codimensional face of $\delta+(1/2)\bsigma$ containing $\rho+\chi$. 

For $\chi\in \M_{\bR}$ and $w\in W$, set \label{not: ast action} \linelabel{line not: ast action}
\[
w\ast \chi\defeq w(\rho+\chi)-\rho\in \M.
\]
If $\chi\in\M$ satisfies $w\ast(\rho+\chi)\neq \rho+\chi$ for any $w\in W\backslash\{1\}$, there exists a unique $w^+\in W$ such that $w^+\ast\chi\in \M^+$. 
This specific element is denoted by 
\label{not: dom chi} \linelabel{line not: dom chi}
$\chi^+\defeq w^+\ast\chi\in \M^+$.

\begin{lem}\label{boundary}
Let $\delta\in\M_{\bR}^W$ and $\chi\in\partial(\delta+\bnabla)\cap\M^+_{\bR}$.
\begin{itemize}
\item[$(1)$]  $w(\rho+\chi+\upbeta_{F_{\chi}(\delta)}^+)=\rho+\chi+\upbeta_{F_{\chi}(\delta)}^+$ implies $w=1$.
\item[$(2)$]  $w_0(\rho+\chi+\upbeta_{F_{\chi}(\delta)}^+)\in \M^+_{\bR}$.
\item[$(3)$] Assume that $\chi\in\M^+$. Then it holds that
\[
\bigl(\chi+\upbeta_{F_{\chi}(\delta)}^+\bigr)^+=w_0\ast(\chi+\upbeta_{F_{\chi}(\delta)}^+),
\] 
and that $\bigl(\chi+\upbeta_{F_{\chi}(\delta)}^+\bigr)^+$ lies in  $\partial(\delta+\bnabla)\cap\M^+$.
 \end{itemize}
 \begin{proof}
 If $W$ is trivial, the results are obvious. Thus, assume that $W\neq1$. \vspace{1mm}\\
 (1) We may assume $\delta=0$. 
Since $\rho+\chi\in F_{\chi}(\delta)$, Lemma \ref{reflection} implies that $\rho+\chi+\upbeta_{F_{\chi}(\delta)}^+\in -F_{\chi}(\delta)$. 
Thus there exists $\chi'\in F_{\chi}(\delta)$ such that $\rho+\chi+\upbeta_{F_{\chi}(\delta)}^+=-\chi'$. 
It is enough to show that $\chi'\defeq-(\rho+\chi+\upbeta_{F_{\chi}(\delta)}^+)\in \M^+_{\bR}$ is strongly dominant. 
Since $F_{\chi}(\delta)$ is the face containing $\rho+\chi$ of maximal codimension, $\rho+\chi$ lies in $F_{\chi}(\delta)^{\circ}$, where $F_{\chi}(\delta)^{\circ}$ denotes the relative interior of $F_{\chi}(\delta)$. 
Since the bijective map $(-)+\upbeta_{F_{\chi}(\delta)}^+\colon \M_{\bR}\to \M_{\bR}$ maps $F_{\chi}(\delta)$ to $-F_{\chi}(\delta)$, it holds that $\rho+\chi+\upbeta_{F_{\chi}(\delta)}^+\in -F_{\chi}(\delta)^{\circ}$, 
and thus $\chi'\in F_{\chi}(\delta)^{\circ}$ since $(1/2)\bsigma$ is centrally symmetric. 
Since $(1/2)\bsigma$ is $W$-invariant and $F_{\chi}(\delta)$ is a dominant face containing a strongly dominant interior point $\rho+\chi$, 
every interior point in $F_{\chi}(\delta)$ is strongly dominant, and in particular $\chi'$ is strongly dominant.
 \vspace{1mm}\\
 (2) We may assume $\delta=0$. 
 The argument above implies that $\rho+\chi+\upbeta_{F_{\chi}(\delta)}^+\in -\M_{\bR}^+$, and hence $w_0(\rho+\chi+\upbeta_{F_{\chi}(\delta)}^+)\in\M_{\bR}^+$.
 \vspace{1mm}\\
 (3)  Since $\rho+\chi+\upbeta_{F_{\chi}(\delta)}^+\in \partial(\delta+(1/2)\bsigma)$ and the polytope $\delta+(1/2)\bsigma$ is $W$-invariant, $w_0(\rho+\chi+\upbeta_{F_{\chi}(\delta)}^+)\in \partial(\delta+(1/2)\bsigma)$. The statement of (2) and Proposition \ref{nabla sigma} show that $w_0(\rho+\chi+\upbeta_{F_{\chi}(\delta)}^+)-\rho\in \partial(\delta+\bnabla)\cap\M^+$.  
 In particular, $w_0\ast(\chi+\upbeta_{F_{\chi}(\delta)}^+)\in\M^+$, and thus the desired equality $\bigl(\chi+\upbeta_{F_{\chi}(\delta)}^+\bigr)^+=w_0\ast(\chi+\upbeta_{F_{\chi}(\delta)}^+)$
 holds.
  \end{proof}
\end{lem}

For each dominant face $F\in \scrF^k_+(\delta+(1/2)\bsigma)$, define a face $F^\dag$ by
\begin{equation}\label{dag}
F^{\dag}\defeq w_0(F^*)\in \scrF^k(\delta+(1/2)\bsigma).
\end{equation}
Since $F^*$ lies in the anti-dominant cone, the face $F^{\dag}$ is dominant.

\begin{lem} \label{hara syukudai 1}
Let $\delta\in \M_{\bR}^W$ and $\chi\in \partial(\delta+\bnabla)\cap \M^+_{\bR}$, and write $\chi'\defeq(\chi+\upbeta_{F_{\chi}(\delta)}^+)^+$. Then   
\[
F_{\chi}(\delta)^{\dag}=F_{\chi'}(\delta).
\] 
\begin{proof}
It follows from Lemma \ref{reflection}\,(2) that $F_{\chi}(\delta)^* = F_{\chi}(\delta) + \beta_{F_{\chi}(\delta)}^+$, which implies that $F_{\chi}(\delta)^*$ contains $\chi + \rho + \beta_{F_{\chi}(\delta)}^+$ in its relative interior.
Thus $w_0\bigl(F_{\chi}(\delta)^*\bigr)$ is the maximal codimensional face of $\delta + (1/2)\bsigma$ that contains $w_0\bigl(\chi + \rho + \beta_{F_{\chi}(\delta)}^+\bigr)$.
Now Lemma \ref{boundary}\,(3) gives $w_0(\chi + \rho + \beta_{F_{\chi}(\delta)}^+) = \chi' + \rho$,
and thus $F_{\chi}(\delta)^{\dag} = w_0(F_{\chi}(\delta)^*)$ coincides with the maximal codimensional face of $\delta + (1/2)\bsigma$ that contains $\chi' + \rho$, which is $F_{\chi'}(\delta)$. 
\end{proof}
\end{lem}

 \begin{dfn} An ordered pair $(\delta,\delta')$ of elements in $\M_{\bR}^W\backslash\scrH^W$ is called an {\it adjacent pair} if $\delta$ and $\delta'$ are adjacent. Two adjacent pairs $(\delta_1,\delta_2)$ and $(\delta'_1,\delta_2')$ are {\it equivalent} if  $[\delta_i]=[\delta'_i]$ holds for each $i=1,2$.
\end{dfn}

In what follows, $(\delta,\delta')$ is an adjacent pair in $\M^W_{\bR}\backslash \scrH^W$, and let $H^W\in \scrH(\delta,\delta')$ be the unique element. Denote by $\delta_0\in H^W$  the point that meets the line segment with endpoints $\delta$ and $\delta'$.

\begin{lem}\label{chi boundary}
Let $\chi\in\scrC_{\delta}$. Then $\chi\notin\scrC_{\delta'}$ if and only if $\chi\in\partial(\delta_0+\bnabla)$.
\begin{proof}
If $\chi\notin\scrC_{\delta'}$, there exists $0\leq t<1$ such that $\chi\in \partial(\delta+t(\delta'-\delta)+\bnabla)$.
But by Proposition \ref{nabla sigma}\,(2)   $\delta+t(\delta'-\delta)=\delta_0$ holds. Conversely, assume that $\chi\in \scrC_{\delta'}$. Then  $\chi$  lies in $(\delta+\bnabla)^{\circ}\cap(\delta'+\bnabla)^{\circ}$ by Proposition \ref{nabla sigma}\,(2), and thus $\chi\in(\delta_0+\bnabla)^{\circ}$ since $\bnabla$ is convex. 
\end{proof}
\end{lem}

 Set 
 \label{not: Fdeltadelta} \linelabel{line not: Fdeltadelta}
 \begin{align*}
 \scrF_{(\delta,\delta')}\defeq\left\{ F\in \scrF_+(\delta_0+(1/2)\bsigma)\relmiddle| \exists \chi\in \scrC_{\delta} \mbox{ such that }F=F_{\chi}(\delta_0)\right\}.
\end{align*}
Note that $\scrF_{(\delta_1,\delta_2)}=\scrF_{(\delta'_1,\delta_2')}$ if $(\delta_1,\delta_2)$ is equivalent to $(\delta'_1,\delta'_2)$.

\begin{lem}\label{orientation lem}
Any $F\in\scrF_{(\delta,\delta')}$ and any $\lambda\in \scrL_{F}$ satisfy $\l\delta'-\delta,\lambda\r>0$.
\begin{proof}
Let $F=F_{\chi}(\delta_0)\in \scrF_{(\delta,\delta')}$. Since the sign of $\l\delta'-\delta,\lambda\r$ does not change under small perturbation of $(\delta,\delta')$, the assumption $(\M_{\bR}^W)_{\rm gen}\neq \emptyset$ allows us to assume that $\delta'-\delta$ is not parallel to any hyperplane  in $ \scrB_F(\delta+\bnabla)$. Let $B_0\in \scrB_F(\delta_0+(1/2)\bsigma)$, and write $B\defeq B_0+\delta-\delta_0$ and $B'\defeq B_0+\delta'-\delta_0$. Then, since $\chi+\rho$ lies in the upper hald space $B_+$ and $\chi+\rho\in B_0$,   $\chi+\rho$ does not lie in the upper half space $B'_+$. In particular, it follows that $B'_+\subsetneq B_+$, which shows that $\lambda_B$ and $\delta'-\delta$ has the same orientation with respect to $B$. Therefore $\l\delta'-\delta,\lambda_B\r>0$.
\end{proof}
\end{lem}

\begin{lem}\label{lattice points}
Let $F=F_{\chi}(\delta_0)\in \scrF_{(\delta,\delta')}$ for some $\chi\in \scrC_{\delta}$. Then 
\[
\bigl(F^{\circ}\cap (\M^++\rho)\bigr)\bigcup \bigl(F^{\circ}\cap (\M^++\rho)\bigr)^*\subset \delta+(1/2)\bsigma,\]
where $\bigl(F^{\circ}\cap (\M^++\rho)\bigr)^*\defeq\left\{-x+2\bigl(\delta_0-(1/2)\upbeta^+_{F}\bigr)\relmiddle| x\in F^{\circ}\cap (\M^++\rho)\right\}$.
\begin{proof}
By Lemma \ref{independence}\,(2) we may assume that $\delta$ is sufficiently near the point $\delta_0$.  By Lemma \ref{reflection}\,(1) $F$ is centrally symmetric with center $\delta_0-(1/2)\upbeta_F^+$, and so $\bigl(F^{\circ}\cap (\M+\rho)\bigr)^*\subset F^{\circ}$.  
The fact that $\chi+\rho$ lies in the interior of the polytope $\delta+(1/2)\bsigma$ implies that $\dim\bigl(F\cap (\delta+(1/2)\bsigma)\bigr)=\dim F$. Thus the finite set $\bigl(F^{\circ}\cap (\M^++\rho)\bigr)\bigcup \bigl(F^{\circ}\cap (\M^++\rho)\bigr)^*\subset F^{\circ}$ is contained in $\delta+(1/2)\bsigma$, since $\delta$ is sufficiently near $\delta_0$ and $F\cap (\delta_0+(1/2)\bsigma)=F$. 
\end{proof}
\end{lem}

\begin{figure}
\centering
\begin{tikzpicture}[baseline=(current bounding box.north), x=7mm, y=7mm]

          \draw[fill=red!20] (3.25,3) --(3.25, 0) --(2.75, 0) --(2.75,3) arc(180:0:0.25) --cycle; 
          \draw[fill=blue!20] (3, 0) circle[radius=0.25];
          
          \draw[fill=red!20] (-3,-3.25) --(0,-3.25) --(0,-2.75) --(-3,-2.75) arc(90:270:0.25) --cycle;
          \draw[fill=blue!20] (0,-3) circle[radius=0.25];
           
           \draw[step=.35cm, gray,very thin] (-4,-4) grid (4,4); 
           
            \draw[line width=1pt]                (3,0) --(3,3) --(0,3) -- (-2, 2) --(-3, 0) --(-3,-3) --(0, -3) --(2,-2) --cycle;
          
          \draw (-0.5, -0.5) node[red]  {$\bullet$};
           \draw (0,0) node[teal] {$\bullet$};
           \draw (0.5,0.5) node[blue]  {$\bullet$};
             
           \draw (-0.5, -0.5) node[below=2pt, red]  {$\delta$};
           \draw (0,0) node[below=2pt, teal]  {$\delta_0$};
           \draw (0.5,0.5) node[below=2pt, blue]  {$\delta'$};

          \node (chi1) at (-2,-3) {$\bullet$};
          \node (chi2) at (-1,-3) {$\bullet$};
          \node (chi3) at (0,-3) {$\bullet$};
          
          \draw (-1.5 ,-3) node[below=8pt,  inner sep=1pt, red, fill=white]  {$F_1$};
           \draw (0,-3) node[below=8pt, inner sep=1pt, blue, fill=white]  {$F_2$};
           
           \node (chi'1) at (3,2) {$\bullet$};
          \node (chi'2) at (3,1) {$\bullet$};
          \node (chi'3) at (3,0) {$\bullet$};
          
          \draw (3 ,1.5) node[right=8pt, inner sep=1pt, red, fill=white]  {$F_1^{\dag}$};
           \draw (3,0) node[right=8pt,  inner sep=1pt, blue, fill=white]  {$F_2^{\dag}$};
           
           \draw (-2.5, 2.7) node[fill=white, inner sep=1pt] {$\delta_0 + \frac{1}{2}\bsigma$};
           
           \draw (0, -4.25) node {$\scrF_{(\delta, \delta')} = \{ F_1, F_2 \}, ~ \scrF_{(\delta', \delta)} = \{ F_1^{\dag}, F_2^{\dag}\}$ };

          \end{tikzpicture}
          
          \caption{A picture for $\scrF_{(\delta, \delta')}$ and  $\scrF_{(\delta', \delta)}$ when $G = \mathrm{GL}_2(\Bbbk)$ and $X = \mathrm{Sym}^3(\Bbbk^2) \oplus \mathrm{Sym}^3(\Bbbk^2)^*$.  }
  \end{figure}

 For $\chi\in\scrC_{\delta}\backslash\scrC_{\delta'}$, define
 the element $\mu_{(\delta,\delta')}(\chi)$ by the formula
 \label{not: mudeltadelta} \linelabel{line not: mudeltadelta}
\[
\mu_{(\delta,\delta')}(\chi)\defeq(\chi+\upbeta_{F_{\chi}(\delta_0)}^+)^+\in\M^+.
\]
\begin{prop} \label{crossing separrationg hyperplane}
Let $\chi\in\scrC_{\delta}\backslash\scrC_{\delta'}$. Then $\mu_{(\delta,\delta')}(\chi)\in\scrC_{\delta'}\backslash\scrC_{\delta}$.

\begin{proof} By Lemma \ref{chi boundary}, $\chi\in \partial(\delta_0+\bnabla)$. Thus Lemma \ref{boundary}\,(3) implies that   $\mu_{(\delta,\delta')}(\chi)\in \partial(\delta_0+\bnabla)\cap\M^+$. Again by Lemma \ref{chi boundary}, it is enough to prove that $\mu_{(\delta,\delta')}(\chi)\in\scrC_{\delta'}$.  

Since $\chi+\rho\in \partial(\delta_0+(1/2)\bsigma)$, $F\defeq F_{\chi}(\delta_0)$ exists. Since $\chi+\rho$ lies in the relative interior of $F$, Lemma \ref{lattice points} shows that there exists $\sigma\in \bsigma$ such that 
\begin{equation}\label{eq:2.18}
-(\chi+\rho)+2(\delta_0-(1/2)\upbeta_F^+)=\delta+(1/2)\sigma.
\end{equation}
By Lemma \ref{independence}\,(1), we may assume that $\delta_0=(\delta+\delta')/2$. Then the equation \eqref{eq:2.18} is equivalent to the equation
\[
\rho+\chi+\upbeta_F^+=\delta'-(1/2)\sigma.
\]
This implies $\rho+\chi+\upbeta_F^+\in \delta'+(1/2)\bsigma$ because $-\bsigma=\bsigma$. Since $\delta'+(1/2)\bsigma$ is $W$-invariant, $w_0(\rho+\chi+\upbeta_F^+)\in \delta'+(1/2)\bsigma$ holds.
Thus
\begin{align*}
\mu_{(\delta,\delta')}(\chi)= w_0(\rho+\chi+\upbeta_F^+) - \rho \in \bigl(\delta' + (-\rho + (1/2)\bsigma)\bigr) \cap \M^+ = \scrC_{\delta'},
\end{align*}
as desired.
\end{proof}
\end{prop}

By the previous result, $\mu_{(\delta,\delta')}(-)$ defines a map
\begin{equation}\label{mu map}
\mu_{(\delta,\delta')}\colon \scrC_{\delta}\backslash\scrC_{\delta'}\to\scrC_{\delta'}\backslash\scrC_{\delta}.
\end{equation}
The following result proves that this map is bijective. 

\begin{prop} \label{CSW2}
Let $\chi\in\scrC_{\delta}\backslash\scrC_{\delta'}$ and put $\chi'\defeq\mu_{(\delta,\delta')}(\chi)$. 
The following hold.
\begin{itemize}
\item[$(1)$]  $\upbeta_{F_{\chi'}(\delta_0)}^+=-w_0(\upbeta_{F_{\chi}(\delta_0)}^+)$.
\item[$(2)$]  $\mu_{(\delta',\delta)}\circ \mu_{(\delta,\delta')}=\id_{\scrC_{\delta}\backslash\scrC_{\delta'}}$.
\end{itemize}
In particular, the map $\mu_{(\delta,\delta')}\colon \scrC_{\delta}\backslash\scrC_{\delta'}\simto\scrC_{\delta'}\backslash\scrC_{\delta}$ is bijective.
\begin{proof}
(1) Since $\rho+\chi+\upbeta_{F_{\chi}(\delta_0)}^+$ lies in the relative interior of the face $F_{\chi}(\delta_0)^*$, 
the element $\rho+\chi'=w_0(\rho+\chi+\upbeta_{F_{\chi}(\delta_0)}^+)$ lies in the relative interior of the face $w_0(F_{\chi}(\delta_0)^*)$. 
This implies that the face $w_0(F_{\chi}(\delta_0)^*)$  is the maximal codimensional face containing $\rho+\chi'$, and thus $F_{\chi'}(\delta_0)=w_0(F_{\chi}(\delta_0)^*)$. 
Therefore, by Lemma \ref{property of beta}, 
\[
\upbeta_{F_{\chi'}(\delta_0)}^+=\upbeta_{w_0(F_{\chi}(\delta_0)^*)}^+=w_0(\upbeta_{F_{\chi}(\delta_0)^*}^+)=-w_0(\upbeta_{F_{\chi}(\delta_0)}^+).
\]
(2) We show that $\mu_{(\delta',\delta)}(\chi')=\chi$. This follows from the following equalities
\begin{align*}
\mu_{(\delta',\delta)}(\chi')&=(\chi'+\upbeta_{F_{\chi'}(\delta_0)}^+)^+\\
&=w_0(\rho+\chi'+\upbeta_{F_{\chi'}(\delta_0)}^+)-\rho\\
&=w_0\Bigl(\rho+w_0\bigl(\rho+\chi+\upbeta_{F_{\chi}(\delta_0)}^+\bigr)-\rho-w_0\bigl(\upbeta_{F_{\chi}(\delta_0)}^+\bigr)\Bigr)-\rho\\
&=w_0(w_0(\rho+\chi))-\rho\\
&=\chi,
\end{align*}
where the third equality follows from (1). 
\end{proof}
\end{prop}

\begin{figure}
\centering
\begin{tikzpicture}[x=5mm, y=5mm]
\tikzset{cross/.style={preaction={-,draw=white,line width=3pt}}} 

        \draw[red, line width=1pt]                (4,0) --(4,4) --(0,4) -- (-4, 2) --(-6, -2) --(-6,-6) --(-2, -6) --(2,-4) --(4,0);
        \draw[shift={(1,1)},teal, line width=1pt] (4,0) --(4,4) --(0,4) -- (-4, 2) --(-6, -2) --(-6,-6) --(-2, -6) --(2,-4) --(4,0);
        \draw[shift={(2,2)},blue, line width=1pt] (4,0) --(4,4) --(0,4) -- (-4, 2) --(-6, -2) --(-6,-6) --(-2, -6) --(2,-4) --(4,0);
        
        \draw[purple, line width=1pt] (-8,-8) --(8,8);

    \foreach \i in {-7,..., 7}
      \foreach \j in {-7,...,7}{
        
        \ifodd \i
        \ifodd \j
        \draw (\i,\j) node {$\cdot$};
        
        \fi
        \fi
              };
              
               \draw[step=.5cm, gray,very thin] (-8,-8) grid (8,8); 
              
              \node[shape=circle, line width=1pt, draw] (a) at (-5,-5) {$\bullet$};
              \node[shape=circle, line width=1pt, draw] (A) at (5,3) {$\bullet$};
              \draw[-Straight Barb, line width=1pt, cross] (a) to[out=75,in=195] (A);

              \node[shape=circle, line width=1pt, draw] (b) at (-3,-5) {$\bullet$};
              \node[shape=circle, line width=1pt, draw] (B) at (5,5) {$\bullet$};
              \draw[-Straight Barb, line width=1pt, cross] (b) to[out=75,in=195] (B);
              
        \node[shape=circle, line width=1pt, draw] (c) at (-1,-5) {$\bullet$};
              \node[shape=circle, line width=1pt, draw] (C) at (5,1) {$\bullet$};
              \draw[-Straight Barb, line width=1pt, cross] (c) to[out=80,in=190] (C);
              
               \node[red] (o1) at (-1,-1) {$\bullet$};
               \node[teal] (o2) at (0,0) {$\bullet$}; 
               \node[blue] (o3) at (1,1) {$\bullet$};
               
               \draw (o1) node[below=2pt, red]  {$\delta$};
               \draw (o2) node[below=2pt, teal]  {$\delta_0$};
               \draw (o3) node[below=2pt, blue]  {$\delta'$};

             \draw (a) node[below=10pt, fill=white, inner sep=1pt]  {$\chi_1$};
             \draw (A) node[right=10pt,fill=white, inner sep=1pt]  {$\chi'_1 = \mu_{\delta, \delta'}(\chi_1)$};
             
             \draw (b) node[below=10pt, fill=white, inner sep=1pt]  {$\chi_2$};
             \draw (B) node[right=10pt,fill=white, inner sep=1pt]  {$\chi'_2 = \mu_{\delta, \delta'}(\chi_2)$};
             
             \draw (c) node[below=10pt, fill=white, inner sep=1pt]  {$\chi_3$};
             \draw (C) node[right=10pt,fill=white, inner sep=1pt]  {$\chi'_3 = \mu_{\delta, \delta'}(\chi_3)$};

             \draw (-7,-7) node[right=10pt, inner sep=1pt, purple, fill=white]  {$\M _{\mathbb{R}}^W$};
             
             \draw  (0,-5) node[right=5pt, inner sep=1pt, red, fill=white]  {$\delta + \bnabla$};
             \draw  (3,-3) node[right=5pt, inner sep=1pt, teal, fill=white]  {$\delta_0 + \bnabla$};
             \draw  (5,-0) node[below=5pt, inner sep=1pt, blue, fill=white]  {$\delta' + \bnabla$};
  \end{tikzpicture}
  
   \caption{A picture for the map $\mu_{\delta, \delta'}$ when $G = \mathrm{GL}_2(\Bbbk)$ and $X = \mathrm{Sym}^3(\Bbbk^2) \oplus \mathrm{Sym}^3(\Bbbk^2)^*$.  }
     
  \end{figure}
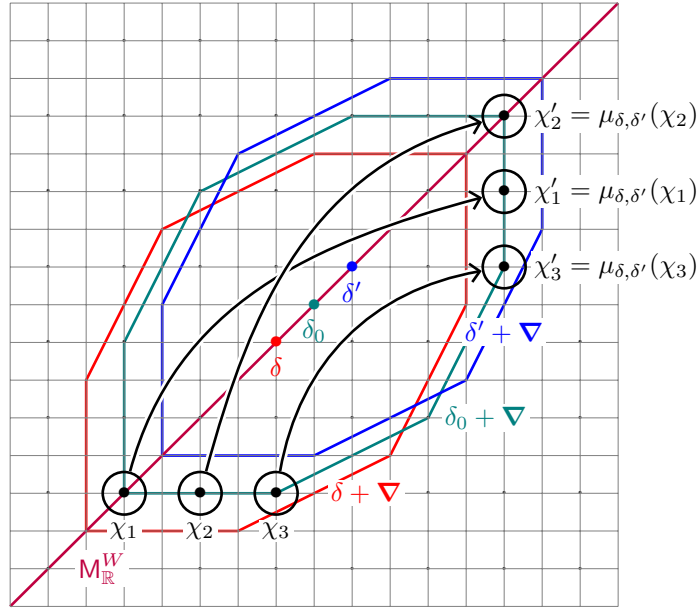

\subsection{Complexes associated to dominant faces}
This section gives the construction of complexes associated to dominant faces of the polytope $\delta+(1/2)\bsigma$, which are dual versions and minor generalizations of the complexes constructed in \cite[Section 3.2]{hl-s} (see also \cite[Section 10.2]{svdb}). 
In this section, $x_i\in X$ denotes the eigenvector with weight $\upbeta_i$. 

Let $F\in\scrF_+(\delta+(1/2)\bsigma)$ be a dominant face. 
Let $X_F^{\leq0}\subset X$ denote the subspace  spanned  by eigenvectors $x_i$ such that  $\l\upbeta_i,\lambda\r\leq0$ for all $\lambda\in\scrL_F$. 
Since $F$ is dominant, every one-parameter subgroup in  $\scrL_F$ is anti-dominant, and  thus  the subspace $X_F^{\leq0}$ is a $B$-submodule of $X$. 
Let $X_F^+\subset X$ be the subspace  generated by eigenvectors $x_i$ with $\upbeta_i\in\scrW_F^+$. 
Then this $X_F^+$ is naturally isomorphic to $X/X_F^{\leq0}$, and this gives a structure of $B$-representation on $X_F^+$. 
Let $S_F$ be the $G$-equivariant vector bundle over $G/B$ defined by $S_F\defeq G\times^B X_F^{\leq0}$, which is a $G$-subbundle of a trivial vector bundle $V\defeq G\times^B X\cong G/B\times X$. 
For $\chi\in\M$, let $k_{\chi}\in\rep B$ denote the  one dimensional representation of $B$ induced by $\chi$, 
and  write
$L_{\chi}\defeq G\times^Bk_{\chi}$ for the $G$-equivariant line bundle on $G/B$. Let $\cL(\chi)$ be the $G$-equivariant invertible sheaf such that $L_{\chi}\cong \underline{\Spec}({\rm Sym}(\cL(\chi)^*))$.
For later use, the following recalls the Borel--Weil--Bott theorem.

\begin{thm}[Borel--Weil--Bott]\label{bwb}
If $\chi+\rho$ has a non-trivial stabilizer in $W$, then for any $i\in \bZ$, the vanishing $H^i(G/B,\cL(\chi))\cong0$ holds. If there is  $w\in W$ such that $w\ast \chi$ is dominant, then $H^{\ell(w)}(G/B,\cL(\chi))\cong V(\chi^+)$ and $H^i(G/B,\cL(\chi))\cong0$ for all $i\neq \ell(w)$. 
\end{thm}

Put $E_F\defeq G\times^BX_F^+$, 
and let $\cE_F$ be the $G$-equivariant locally free sheaf such that $E_F\cong \underline{\Spec}({\rm Sym}(\cE_F^{*}))$.   
Let $s\defeq (x_i)\colon \cO_V\to p^*\cE_F$ be a regular section locally defined by the regular functions $x_i$ with $\upbeta_i\in \scrW_F^+$. Then $S_F$ is the zero scheme of $s$, where $p\colon V\to G/B$ is the projection to the first factor. Thus the following resolution of $\cO_{S_F}$ exists:
\begin{equation}\label{koszul cpx}
0\to \bigwedge^{d_F}p^*(\cE_F^{*})\to\cdots\to p^*(\cE_F^{*})\xrightarrow{s^{*}} \cO_V\to\cO_{S_F}\to0.
\end{equation}
This resolution is the one so called Koszul resolution (or complex).
Taking the dual of the above Koszul complex and then tensoring with the line bundle $\cL(\chi)$ gives a complex 
\[
K_{F,\chi}^{\bullet}\colon 0\to \cL(\chi)\to \cL(\chi)\otimes p^*\cE_F\to\cdots\to\cL(\chi)\otimes \bigwedge^{d_F}p^*\cE_F\to0.
\]
Let $q\colon V\to X$ denote the other projection. 
Then the direct image $q_*\colon \coh_GV\to\coh_GX$ is left exact, and the $\cO_X$-module $\dR^iq_*\bigl(\cL(\chi)\otimes\bigwedge^{j}p^*\cE_F\bigr)\cong H^i\bigl(G/B,\cL(\chi)\otimes\bigwedge^{j}p^*\cE_F\bigr)\otimes\cO_X$ is projective. Hence, by \cite[Theorem 11.3]{svdb}, there exists a  complex $A_{F,\chi}^{\bullet}\in\Db(\coh_GX)$ such that each term is given by 
\begin{equation}\label{A cpx}
A_{F,\chi}^{i}\defeq \bigoplus_{j\in\bZ}H^j\left(G/B,\cL(\chi)\otimes\bigwedge^{i-j}p^*\cE_F\right)\otimes\cO_X,
\end{equation}
and that it is quasi-isomorphic to $\dR q_*(K_{F,\chi}^{\bullet})$. Set 
\begin{align*}
\gwedge^m\scrW_F^+&\defeq \left\{\upbeta_{i_1}+\cdots+\upbeta_{i_m}\relmiddle|\upbeta_{i_j}\in \scrW_F^+ \mbox{ and } i_1<\cdots<i_m \right\}\\
\gwedge^*\scrW_F^+&\defeq \bigcup_{1\leq m\leq d_F^+-1}\gwedge^m\scrW_F^+.
\end{align*}

\begin{prop}\label{key cpx} 
Assume that $\chi$ is dominant and $\ell\in\M^W$ is generic.
Then the following hold.
\begin{itemize}
\item[$(1)$] If $\l\ell,\lambda\r>0$ for some $\lambda\in\scrL_F$, then the restriction of $A_{F,\chi}^{\bullet}$ to  $[X^{\rm ss}(\ell)/G]$ is acyclic.
\item[$(2)$] $A_{F,\chi}^0=V(\chi)\otimes\cO_X$ and $A_{F,\chi}^{d_F^++\ell(w_0)}=V\left((\chi+\upbeta_F^+)^+\right)\otimes\cO_X$.
\item[$(3)$] The terms $A_{F,\chi}^{1},\hdots, A_{F,\chi}^{d_F^++\ell(w_0)-1}$ are direct sums of $G$-equivariant locally free sheaves of the form
\[
V\left((\chi+\upbeta)^+\right)\otimes\cO_X,
\]
where $\upbeta\in\gwedge^*\scrW_F^+$.
\item[$(4)$] $A_{F,\chi}^i\not\cong0$  only if $0\leq i\leq d_F^++\ell(w_0)$.
\end{itemize}
\begin{proof}
(1) Recall that $x\in X^{\rm ss}(\ell)$ if and only if $\l\ell,\lambda\r\geq0$ for all $\lambda$ such that $\lim_{t\to0}\lambda(t) x$ exists. If $x\in X_F^{\leq0}$, then $\lim_{t\to0}(-\lambda)(t) x$ exists for all $\lambda\in\scrL_F$. Thus  $\l\ell,\lambda\r>0$ for some $\lambda\in\scrL_F$ implies that $x\notin X^{\rm ss}(\ell)$. Hence $X_F^{\leq0}\subseteq X\backslash X^{\rm ss}(\ell)$. Since   $\Supp(A_{F,\chi}^{\bullet})\subseteq q(S_F)=X_F^{\leq0}$,  the restriction $A_{F,\chi}^{\bullet}|_{\X}$ is acyclic. \vspace{1mm}\\
(2) Since $H^i(G/B, \cL(\chi)\otimes \gwedge^m p^*\cE_F)\not\cong0$  only if $0\leq i\leq \ell(w_0)$ and $0\leq m\leq d_F^+$, the formula (\ref{A cpx}) gives $A_{F,\chi}^0=H^0(G/B,\cL(\chi))\otimes \cO_X$ and $A_{F,\chi}^{d_F^++\ell(w_0)}=H^{\ell(w_0)}(G/B,\cL(\chi)\otimes\det(p^*\cE_F))\otimes\cO_X$. Thus the result follows from Borel--Weil--Bott theorem.\vspace{1mm}\\
(3)  Fix $0\leq m\leq d_F^+$. Weights of $k_{\chi}\otimes\gwedge^m X_F^+$ are  of the form $\chi+\upbeta$ for some $\upbeta\in \gwedge^m\scrW_F^+$. Since $B$ is solvable, there is a full flag of the $B$-representation $k_{\chi}\otimes\gwedge^m X_F^+$, and this flag induces a filtration 
\begin{equation}\label{filtration}
0\subset \cE_1\subset \cdots\subset \cE_r=\cL(\chi)\otimes\gwedge^m p^*\cE_F
\end{equation}
whose factors $\cL_i\defeq \cE_i/\cE_{i-1}$ are isomorphic to $\cL(\chi+\upbeta)$ for some $\upbeta\in\gwedge^m\scrW_F^+$.

 We claim that for each $1\leq n\leq r$ and each $i\in \bZ$, the cohomology group $H^i(G/B, \cE_n)$ is the direct sum of $G$-representations of the form $V((\chi+\upbeta)^+)$ for some $\upbeta\in\gwedge^m\scrW_F^+$. 
Let us prove this claim by induction on $n$. The claim is obvious when $n=1$.  
Consider the long exact sequence  
\begin{equation}\label{exact seq}
\cdots\to H^{i-1}(\cL_n)\to H^i(\cE_{n-1})\to H^i(\cE_n)\to H^i(\cL_n)\to H^{i+1}(\cE_{n-1})\to\cdots
\end{equation}
obtained from the short exact sequence $0\to \cE_{n-1}\to\cE_n\to \cL_n\to0$, where $H^i(-)$ denotes $H^i(G/B,-)$ for short.
If $H^i(\cL_n)\cong 0$ for all $i\in \bZ$, \eqref{exact seq} implies isomorphisms $H^i(\cE_{n})\cong H^i(\cE_{n-1})$ for all $i\in \bZ$. 
In this case, the claim follows from inductive hypothesis. 
If $H^{\ell}(\cL_n)\neq 0$ for some $\ell$,   Borel--Weil--Bott theorem shows that  $H^{i}(\cL_n)\cong0$ for all $i\neq\ell$.  Then \eqref{exact seq} implies isomorphisms $H^i(\cE_{n-1})\cong H^i(\cE_n)$ for any $i\neq \ell, \ell+1$, and there is an exact sequence 
\[
0\to H^{\ell}(\cE_{n-1})\to H^{\ell}(\cE_n)\xrightarrow{\alpha} H^{\ell}(\cL_n)\to H^{{\ell}+1}(\cE_{n-1})\to H^{{\ell}+1}(\cE_{n})\to0.
\]
Since $H^{\ell}(\cL_n)$ is an irreducible representation, the morphism $\alpha$ is either surjective or the zero map. In either case, $H^{\ell}(\cE_n)$ and $H^{\ell+1}(\cE_n)$ lie in the additive closure of the representations $H^{\ell}(\cE_{n-1})$, $H^{\ell+1}(\cE_{n-1})$ and $H^{\ell}(\cL_n)$. Since $H^{\ell}(\cL_n)$ is a representation of the from $V((\chi+\upbeta_{i_1}+\cdots+\upbeta_{i_m})^+)$, the inductive hypothesis implies that $H^{\ell}(\cE_n)$ and $H^{\ell+1}(\cE_n)$ are direct sums of representations of the form $V((\chi+\upbeta_{i_1}+\cdots+\upbeta_{i_m})^+)$.  This proves the claim.

Since $\chi$ is dominant,  Borel--Weil--Bott theorem shows that $H^i(\cL(\chi)\otimes \gwedge^0 p^*\cE_F)$ vanish for all $i\neq 0$, and $H^i(\cL(\chi)\otimes \gwedge^{d_F^+} p^*\cE_F)$ vanish for all $i\neq \ell(w_0)$. 
This implies that the summands $H^i(\cL(\chi)\otimes \gwedge^0 p^*\cE_F)$ and $H^i(\cL(\chi)\otimes \gwedge^{d_F^+} p^*\cE_F)$ in the terms $A_{F,\chi}^{1},\hdots, A_{F,\chi}^{d_F^++\ell(w_0)-1}$ vanish. Thus the claim proves (2).\\
(4) By Borel--Weil--Bott theorem, $H^i(\cL_n)=0$ for all $i>\ell(w_0)$ and $1\leq n\leq r$, where $\cL_n$ is the factor $\cE_n/\cE_{n-1}$ of the filtration \eqref{filtration}. Therefore, by  induction on $n$,  the vanishing $H^i(\cE_n)=0$ holds for all $i>\ell(w_0)$ and $1\leq n\leq r$. In particular, 
$H^i(G/B, \cL(\chi)\otimes \gwedge^m p^*\cE_F)\not\cong0$  only if $0\leq i\leq \ell(w_0)$. Thus, since $\cL(\chi)\otimes \gwedge^m p^*\cE_F\not\cong 0$ only if $0\leq m\leq d_F^+$,  $A_{F,\chi}^i$ does not vanish only if $0\leq i\leq d_F^++\ell(w_0)$.
\end{proof}
\end{prop}


\section{Main results}
\subsection{Wall crossing and tilting equivalence}

This section shows that wall-crossings of magic windows correspond to equivalences that are induced by  tilting modules. 

Let $X$ be a quasi-symmetric representation of $G$ with $T$-weights $\upbeta_1,\hdots,\upbeta_d\in \M$, and  assume that $\bsigma$ spans $\M_{\bR}$ and $(\M_{\bR}^W)_{\rm gen}\neq \emptyset$.
In what follows, fix a generic $\ell$, and set
\label{not: GIT stack} \linelabel{line not: GIT stack}
\[
\X\defeq[X^{\rm ss}(\ell)/G].
\]
 By Remark \ref{HM criterion}, for $x\in X^{\rm us}(\ell)\defeq X\backslash X^{\rm ss}(\ell)$, there exists $\lambda\colon \bG_m\to G$ such that $\lim_{t\to0}\lambda(t)x$ exists and $\l\ell,\lambda\r<0$. Moreover, there are finitely many one-parameter subgroups $\lambda_1,\hdots,\lambda_r$ with $\l\ell,\lambda_i\r<0$ such that each $\lambda_i$ defines a certain $G$-invariant  subvariety $S_{\lambda_i}\subset X^{\rm us}(\ell)$ and there is  a disjoint union
\begin{equation}\label{stratification}
X^{\rm us}(\ell)=\bigsqcup_{i=1}^rS_{\lambda_i}.
\end{equation}
Put $\eta_{i}\defeq \eta_{g\lambda_ig^{-1}}$, where $g\in G$ is an element such that $g\lambda_ig^{-1}\in\N$. 
The following vanishing is a consequence of Teleman's quantization theorem  \cite{tel} (see also \cite[Theorem 3.29]{hl} for more general version).

\begin{thm}[{\cite{tel}}]\label{quantization}  Let $V$ be a representation of $G$ such that,  for all $1\leq i\leq r$, every $\lambda_i$-weight $\chi$ of $V$ satisfies  $\l\chi,\lambda_i\r<\eta_{i}$.  Then, for all $k\neq 0$, the following vanishing holds:
\[
H^k(\X,  V\otimes \cO_{\X})\cong 0.
\]
\begin{proof}
By Teleman's quantization theorem \cite{tel}, for all $k\in \bZ$, the natural restriction map induces an isomorphism 
\[
H^k([X/G],V\otimes\cO_X)\cong H^k(\X, V\otimes \cO_{\X}).
\]
The assumption that $G$ is reductive and that $X$ is affine imply
$H^k([X/G],V\otimes\cO_X)\cong 0$ for all $k\neq0$.
\end{proof}
\end{thm}

If $\delta\in \M^W_{\bR}\backslash\scrH^W$, by \eqref{window generation} and Theorem \ref{hl-s}, the objects 
\label{tilting bdl T} \linelabel{line: tilting bdl T} 
$\scrT_{\delta}\defeq \bigoplus_{\chi\in\scrC_{\delta}}V_X(\chi)\in\coh[X/G]$ and \linelabel{line: tilting bdl}
\begin{equation}\label{tilting bdl}
\scrV_{\delta}\defeq \bigoplus_{\chi\in\scrC_{\delta}}V_{\X}(\chi)=\res(\scrT_{\delta})\in \coh \X
\end{equation}
are split generators of $\scrM(\delta+\bnabla)$ and $\Db(\coh \X)$ respectively. Moreover, since $G$ is reductive and $X$ is affine, $\scrT_{\delta}$ and $\scrV_{\delta}$ are  tilting objects. 
Define the corresponding endomorphism algebras by
\linelabel{line not: NCCR}
\begin{equation}\label{end alg}
\Lambda_{\delta}\defeq \End_{[X/G]}(\scrT_{\delta})\cong \End_{\X}(\scrV_{\delta}),
\end{equation}

which provides a natural equivalence
\[
\RHom(\scrT_{\delta},-)\colon \scrM(\delta+\bnabla)\simto \Db(\fmod \Lambda_{\delta}).
\]

\begin{lem}\label{vanishing}
 Let $\delta,\delta'\in \M_{\bR}^W$ be two distinct points such that either $\delta$ or $\delta'$ does not lie in $\scrH^W$. If $\delta'-\delta$ has the same orientation as $\ell$ with respect to all separating hyperplanes of $\delta$ and $\delta'$, then 
\[
\Ext^k_{\X}(V_{\X}(\chi),V_{\X}(\chi'))\cong0
\]
for any $\chi\in \scrC_{\delta}$, $\chi'\in \scrC_{\delta'}$  and $k\neq0$.
\begin{proof} 
We only consider the case when   $\delta\not\in \scrH^W$, since the other case can be proved similarly.  Then  $\scrC_{\delta}$ does not change under small perturbations of $\delta$ inside the chamber  of $\M_{\bR}^W\backslash \scrH^W$ containing $\delta$. Therefore, by the assumption on the orientation of $\delta'-\delta$, we may assume that $\delta'=\delta+t\ell$ for  some $t>0$. Let $\alpha$ and $\alpha'$ be arbitrary weights  of $V(\chi)$ and $V(\chi')$ respectively.   Let $\lambda_1,\hdots,\lambda_r$ be the one-parameter subgroup defining the decomposition \eqref{stratification}. Since  $\alpha\in \delta+\bnabla$ and $\alpha'\in\delta'+\bnabla$, 
\begin{align*}
\l\alpha,\lambda_i\r&\in \l\delta,\lambda_i\r+[-\eta_i/2,\eta_i/2]\\
\l\alpha',\lambda_i\r&\in \l\delta,\lambda_i\r+t\l\ell,\lambda_i\r+[-\eta_i/2,\eta_i/2]
\end{align*}
for each $1\leq i\leq r$.  These show that $\l\alpha'-\alpha,\lambda_i\r<\eta_i$ for each $1\leq i\leq r$, since $\l\ell,\lambda_i\r<0$. Hence Theorem \ref{quantization} implies that  \[H^k(\X,V_{\X}(\chi')\otimes V_{\X}(\chi)^{*})\cong0\] for all $k\neq0$, and this completes the proof.
\end{proof}
\end{lem}

For an ordered pair $(\delta,\delta')$ of points in $\M_{\bR}^W\backslash \scrH^W$, put
\[
T_{(\delta,\delta')}\defeq \Hom_{\X}(\scrV_{\delta},\scrV_{\delta'})\in \fmod\Lambda_{\delta}.
\]
\begin{thm}\label{wall crossing} Let $\delta, \delta'\in\M_{\bR}^W\backslash \scrH^W$, and  
assume that $\delta'-\delta$ has the same orientation as $\ell$ with respect to all separating hyperplanes of $\delta$ and $\delta'$.
Then the following hold:
\begin{itemize}
    \item[$(1)$] $T_{(\delta,\delta')}$ is a tilting $\Lambda_{\delta}$-module. 
    \item[$(2)$]  The  diagram 
\[
\begin{tikzpicture}[xscale=1.3]
\node (A0) at (2,0) {$\Db(\fmod\Lambda_{\delta})$};
\node (A5) at (8,0) {$\Db(\fmod\Lambda_{\delta'})$};
\node (B0) at (2,1.5) {$\scrM(\delta+\bnabla)$};
\node (B4) at (4.75,1.5) {$\Db(\coh\X)$};
\node (B5) at (8,1.5) {$\scrM(\delta'+\bnabla)$};
\draw[->] (A0) -- node[above] {$\scriptstyle\RHom_{}(T_{(\delta,\delta')},-)$}(A5);
\draw[->] (B0) -- node[above] {$\scriptstyle\res$}(B4);
\draw[->] (B4) -- node[above] {$\scriptstyle\res^{-1}$}(B5);
\draw[->] (B0) -- node[left] {$\scriptstyle\RHom_{}(\scrT_{\delta},-)$}(A0);
\draw[->] (B5) -- node[right] {$\scriptstyle\RHom_{}(\scrT_{\delta'},-)$}(A5);
\end{tikzpicture}
\]
of equivalences is commutative.
\end{itemize}
\begin{proof}
(1)  
The adjunction gives an isomorphism
\[
\RHom_{}(\scrT_{\delta}, \res^{-1}\res(\scrT_{\delta'}))\cong \RHom_{\X}(\scrV_{\delta},\scrV_{\delta'}).
\]
By Lemma \ref{vanishing}, the vanishing $\Ext^k(\scrV_{\delta},\scrV_{\delta'})\cong 0$ holds for all $k\neq 0$, and thus 
\begin{equation}\label{equation:comm}
\RHom_{}(\scrT_{\delta},\res^{-1}\res(\scrT_{\delta'}))\cong T_{(\delta,\delta')}.
\end{equation}
Since  $T_{(\delta,\delta')}$ is  the image of the tilting object $\scrT_{\delta'}$ by the equivalence $\RHom_{}(\scrT_{\delta},\res^{-1}\res(-))$, it is also a tilting object. This proves (1).\vspace{1mm}\\
(2)  Since all the functors in the diagram can be lifted to suitable dg functors, there exist functor isomorphisms 
\begin{align}
\RHom(\scrT_{\delta'},\res^{-1}\res(-))&\cong \RHom(\res^{-1}\res(\scrT_{\delta'}),-)\label{eqn:1}\\
\RHom_{}(T_{(\delta,\delta')},\RHom(\scrT_{\delta},-))&\cong \RHom_{}(T_{(\delta,\delta')}\otimes^{\dL}\scrT_{\delta},-).\label{eqn:2}
\end{align}
Therefore we only need to prove that the right hand sides of \eqref{eqn:1} and \eqref{eqn:2} are isomorphic functors.
But this follows from a natural isomorphism 
\[
\res^{-1}\res(\scrT_{\delta'})\cong T_{(\delta,\delta')}\otimes^{\dL}\scrT_{\delta}
\]
 obtained by applying the equivalence $(-)\otimes^{\dL}\scrT_{\delta}$ to the isomorphism \eqref{equation:comm}.
\end{proof}
\end{thm}


\subsection{Exchanges/mutations of modules from magic windows}
Let $X$ be a quasi-symmetric representation of $G$ with $T$-weights $\upbeta_1,\hdots,\upbeta_d\in \M$, and  assume that $\bsigma$ spans $\M_{\bR}$ and $(\M_{\bR}^W)_{\rm gen}\neq \emptyset$. 
A quasi-symmetric representation $X$ is said to be {\it generic} if ${\rm codim}(X\backslash X^{\rm ts},X)\geq 2$, where $X^{\rm ts}$ is a $G$-invariant open subspace of $X$ defined by 
\[
X^{\rm ts}\defeq\{x\in X\mid \mbox{the orbit $G\cdot x$ is closed and the stabilizer $G_x$ is trivial}\}.
\]
Write  $R\defeq \sfk[X]^G\subseteq \sfk[X]$ for  the subring of  $G$-invariant polynomials. Then $\Spec R$ is the affine quotient of $X$, and  taking $G$-invariant part defines a functor
\begin{equation}\label{inv part}
(-)^G\colon \fmod_G\sfk[X]\to \fmod R\semicolon M\mapsto M^G.
\end{equation}
Denote by $\refl_G\sfk[X]\subset \fmod_G\sfk[X]$ and by $\refl R\subset \fmod R$ the subcategories of reflexive modules. Then the functor  \eqref{inv part} restricts to a functor $(-)^G\colon \refl _G \sfk[X]\to \refl R$. 

\begin{lem}[{\cite[Lemma 3.3]{svdb}}]
Notation is same as above. If the $G$-representation $X$ is generic, the functor \[(-)^G\colon \refl _G \sfk[X]\to \refl R\] is an equivalence of symmetric monoidal categories.
\end{lem}

For a $G$-representation $V\in \rep G$, an $R$-module $M(V)$ is defined by 
\[
M(V)\defeq (V\otimes_{\sfk} \sfk[X])^G,
\]
which is called a {\it module of covariants} associated to $V$.
Then $M(-)$ defines a functor 
$
M(-)\colon \rep G\to \fmod R
$. If $V=V(\chi)$ is an irreducible representation with highest weight $\chi\in \M^+$, 
put \label{not: MCM ass to chi} \linelabel{line not: MCM ass to chi}
\[
M(\chi)\defeq M(V(\chi)).
\]
 \begin{lem}
Let $\delta\in\M_{\bR}^W$, and $\chi,\chi'\in \scrC_{\delta}$. If either $\chi$ or $\chi'$ lies in the interior of $\delta+\bnabla$, then the $R$-module $\Hom_R(M(\chi),M(\chi'))$ is Cohen-Macaulay.
\begin{proof}
There is a natural isomorphism $\Hom_R(M(\chi),M(\chi'))\cong M(V(\chi)^*\otimes V(-w_0(\chi'))^*)$. By \cite[Proposition 3.14]{svdb}, it is enough to show that $\chi-w_0(\chi')$ lies in the interior of $-2\rho+\bsigma$. We may assume that $\delta$ is the origin, and then we may write $\chi=-\rho+(1/2)\sum_{i}a_i\upbeta_i$ and 
$\chi'=-\rho+(1/2)\sum_{i}a'_i\upbeta_i$ for some $a_i, a'_i\in [0,1]$. Since $\bsigma=-\bsigma$, $-w_0(\chi')=-\rho-(1/2)\sum_ia'_i(w_0(\upbeta_i))$ lies in $-\rho+(1/2)\bsigma$. Hence there are $a''_i\in [0,1]$ such that $-w_0(\chi')=-\rho+\sum_ia''_i\upbeta_i$, and then   
\begin{align*}
\chi-w_0(\chi')&=-2\rho+\sum_i\left(\frac{1}{2}a_i+\frac{1}{2}a''_i\right)\upbeta_i.
\end{align*}
If $\chi'$ lies in the interior of $\bnabla$, so does $-w_0(\chi')$.
Therefore, if either $\chi$ or $\chi'$ lies in the interior of $\bnabla$, we may assume that either all $a_i\neq1$ or all $a''_i\neq1$, and then $\chi-w_0(\chi')$ lies in the interior of $-2\rho+\bsigma$.
\end{proof}
 \end{lem}

Let  $\ell\in \M^W_{\rm gen}$ be a generic element. Then the associated GIT quotient stack $\X=[X^{\rm ss}(\ell)/G]$ is a smooth Deligne--Mumford stack by Proposition \ref{dm stack}. Since $\ell$ is generic, the natural map
$
X^{\rm ss}(\ell) \to X^{\rm ss}(\ell)/\!\!/G
$
is a geometric quotient, and let 
\[\pi\colon \X\to X^{\rm ss}(\ell)/\!\!/G\]
be the natural map. Then  $\pi$ is proper, and the restriction $\pi|_{U}\colon [U/G]\to U/\!\!/G$ to  the open subspace $U\defeq X^{\rm ts}\subseteq X^{\rm ss}(\ell)$ is an isomorphism.  
Let $f\colon X^{\rm ss}(\ell)/\!\!/G\to \Spec R$ be the natural quotient map, and put
\[
\varphi\defeq f\circ \pi\colon \X\to \Spec R.
\]

\begin{prop}[{\cite[Lemma 4.16]{vdb}}]
If the representation $X$ is generic, the morphism $\varphi\colon \X\to \Spec R$ is a (stacky) crepant resolution. 
\end{prop}

 Since $\scrV_{\delta}$ is a tilting bundle on $\X$, 
there is a natural derived equivalence \[\Db(\coh \X)\cong \Db(\fmod\Lambda_{\delta}).\] 
For each $\delta\in \M_{\bR}^W$, set \label{not: M_delta} \linelabel{ line not: M_delta}
\[
M_{\delta}\defeq\bigoplus_{\chi\in\scrC_{\delta}}M(\chi) \in\fmod R.
\]
If $X$ is generic,  the algebra $\Lambda_{\delta}$, defined in \eqref{end alg}, is isomorphic to the $R$-algebra $\End_R(M_{\delta})$, and there is a natural isomorphism
 \begin{equation}\label{gen isom}
 M(\chi)\cong \varphi_*(V_{\X}(\chi)).
 \end{equation}

\begin{prop}[{\cite[Theorem 4.6]{vdb}}]
If $X$ is generic, $\Lambda_{\delta}$ is an NCCR of $R$.
\end{prop}

 In the remainder of this section, assume that $X$ is generic, and fix  an adjacent pair $(\delta,\delta')$  in $\M_{\bR}^W\backslash\scrH^W$ separated by a hyperplane $H^W\in \scrH^W$.  
 Denote by $\delta_0\in H^W$  the point that meets line segment with endpoints $\delta$ and $\delta'$. 
The rest of this section discusses a relationship between $M_{\delta}$ and $M_{\delta'}$.

For $F\in \scrF_{(\delta,\delta')}$, put
\label{not: C^F_delta} \linelabel{line not: C^F_delta}
$
\scrC_{\delta}^F\defeq\left\{ \chi \in \scrC_{\delta}\relmiddle| F=F_{\chi}(\delta_0)\right\},
$
and set \label{not: M^F_delta} \linelabel{line not: M^F_delta}
\begin{align*}
M_{\delta}^F&\defeq\bigoplus_{\chi\in \scrC_{\delta}^F}M(\chi).
\end{align*}
Then, by Lemma \ref{chi boundary}, there is a direct sum decomposition 
\begin{equation}\label{decomp}
M_{\delta}=M_{\delta\cap\delta'}\oplus \bigoplus_{F\in \scrF_{(\delta,\delta')}}M_{\delta}^F,
\end{equation}
where 
\label{not: M_common} \linelabel{line not: M_common} 
 $M_{\delta\cap\delta'}\defeq \bigoplus_{\chi\in \scrC_{\delta}\cap \scrC_{\delta'}}M(\chi)$.
 Recall from \eqref{mu map} that there is a bijective map
 \[
 \mu_{(\delta,\delta')}\colon \scrC_{\delta}\backslash\scrC_{\delta'}\simto\scrC_{\delta'}\backslash\scrC_{\delta}.
 \]
\begin{lem}\label{mu lem}
Notation is same as above. 
\begin{itemize}
\item[(1)] $\scrF_{(\delta',\delta)}=\left\{ F^{\dag}\relmiddle|F\in \scrF_{(\delta,\delta')}\right\}$.
\item[(2)] Let $F\in \scrF_{(\delta,\delta')}$. The map $\mu_{(\delta,\delta')}$ restricts to a bijection $\mu_{(\delta,\delta')}\colon \scrC_{\delta}^F\simto \scrC_{\delta'}^{F^{\dag}}$.
\item[(3)] $M_{\delta'}^{F^{\dag}}\cong \bigoplus_{\chi\in \scrC_{\delta}^F} M(\mu_{(\delta,\delta')}(\chi))$.
\end{itemize}
\begin{proof}
(1) For $F \in \scrF_{(\delta,\delta')}$,  we show that $F^{\dag} \in \scrF_{(\delta',\delta)}$.
By  definition of $\scrF_{(\delta,\delta')}$, there exists $\chi \in \scrC_{\delta}$ such that $F = F_{\chi}(\delta_0)$.
Then Lemma \ref{hara syukudai 1} shows that $F^{\dag} = F_{\chi'}(\delta_0)$, where $\chi' = (\chi + \beta_{F_{\chi}(\delta_0)}^+)^+$.
Since $\chi'=\mu_{(\delta,\delta')}(\chi) \in \scrC_{\delta'}$ by Proposition \ref{crossing separrationg hyperplane}, the definition of $F_{\chi'}(\delta_0)$ implies that $F^{\dag} \in \scrF_{(\delta',\delta)}$.
Now we have checked that $\scrF_{(\delta',\delta)} \supset \left\{ F^{\dag}\relmiddle|F\in \scrF_{(\delta,\delta')}\right\}$.
Swapping $\delta$ with $\delta'$ gives a sequence of inclusions
\[ \left\{ (F_1^{\dag})^{\dag} \relmiddle|F_1 \in \scrF_{(\delta',\delta)}\right\} \subset \left\{ F_2^{\dag} \relmiddle| F_2 \in \scrF_{(\delta,\delta')}\right\} \subset \scrF_{(\delta',\delta)}. \]
Now Proposition \ref{CSW2}\,(2) shows that $(F_1^{\dag})^{\dag} = F_1$ for all $F_1 \in  \scrF_{(\delta',\delta)}$.
Thus all inclusions above are equalities, which completes the proof of (1).\vspace{1mm}\\
(2) This also follows from Lemma \ref{hara syukudai 1} and the fact that $\mu_{\delta,\delta'}$ is a bijection.\vspace{1mm}\\
(3) This is a consequence of (2).
\end{proof}
\end{lem}

For each $F\in \scrF_{(\delta,\delta')}$, set 
\label{not: LN} \linelabel{line not: LN}
 \begin{align*}
 L^F_{\delta}&\defeq \bigoplus_{\chi\in \scrC_{\delta}^F}\left( \bigoplus_{{\upbeta\in{\tiny \gwedge}^*\scrW_F^+}} M((\chi+\upbeta)^+)\right),\\
 N^F_{\delta}&\defeq \bigoplus_{\chi\in \scrC_{\delta}\backslash \scrC^F_{\delta}}M(\chi)\cong M_{\delta}/M_{\delta}^F.
 \end{align*}
 By \eqref{decomp} and Lemma \ref{mu lem}\,(1), $M_{\delta'}$ can be described as follows:
\begin{equation}\label{delta' decomp}
M_{\delta'}=M_{\delta\cap\delta'}\oplus \bigoplus_{F\in \scrF_{(\delta,\delta')}}M_{\delta'}^{F^{\dag}}.
\end{equation}
The following result shows that $M_{\delta}$ can be obtained from $M_{\delta'}$ by iterated exchanges of  the summands $M_{\delta'}^{F^{\dag}}$.

\begin{thm}\label{main thm}
Notation is same as above.

\begin{itemize}
\item[(1)] Let $F\in \scrF(\delta,\delta')$, and $\chi\in \scrC^F_{\delta}$. Then 
\[
M(\chi)\in \cE_{(L_{\delta}^F, N_{\delta}^F)}^{d_{F}^++\ell(w_0)-1}\Bigl(M\bigl(\mu_{(\delta,\delta')}(\chi)\bigr)\Bigr).
\]
\item[(2)] For each $F\in \scrF(\delta,\delta')$, 
\[
M_{\delta}^F\in\cE_{(L^F_{\delta},N^F_{\delta})}^{d_{F}^++\ell(w_0)-1}(M_{\delta'}^{F^{\dag}}).
\]
\item[(3)] The following   holds.
\[
M_{\delta}\in M_{\delta\cap\delta'}\oplus \bigoplus_{F\in \scrF_{(\delta,\delta')}} \cE_{(L^F_{\delta},N^F_{\delta})}^{d_{F}^++\ell(w_0)-1}(M_{\delta'}^{F^{\dag}}).
\]
\end{itemize}
\begin{proof}
(1) 
Let $\ell\defeq \delta'-\delta$, and  $\chi\in\scrC_{\delta}^F$. Then $\ell$ is generic, and $\X\defeq[X^{\rm ss}(\ell)/G]$ is a Deligne--Mumford stack. Consider the restriction  $B_{F,\chi}^{\bullet}\defeq A_{F,\chi}^{\bullet}|_{\X}$ of  the complex $A_{F,\chi}^{\bullet}$ in \eqref{A cpx} to $\X$. Then, by Proposition \ref{key cpx}, the complex $B_{F,\chi}^{\bullet}$ is a bounded complex of the form
\begin{equation}\label{long ex seq 4.9}
 0\to V_{\X}(\chi)\xrightarrow{\delta^0} \cdots\xrightarrow{\delta^{d_F^++\ell(w_0)-1}} V_{\X}\left((\chi+\upbeta_F^+)^+\right)\to0,
\end{equation}
and it is exact by Proposition \ref{key cpx}\,(1) and Lemma \ref{orientation lem}.  
Now putting $\cE_\chi^i\defeq \Ker(\delta^{i+1})$ gives the short exact sequences
\begin{equation}\label{ex seq}
0\to \cE_\chi^{i-1}\to B_{F,\chi}^i\xrightarrow{\beta^i} \cE_\chi^{i}\to0
\end{equation}
for all $1\leq i\leq d_F^++\ell(w_0)-1$,
which are spliced to be the original sequence (\ref{long ex seq 4.9}).
Set $M^i\defeq \varphi_*(\cE_\chi^i)$ and $K^i\defeq \varphi_*(B_{F,\chi}^i)$. Since $\varphi_*$ is left exact, there is an exact sequence
\[
0\to M^{i-1}\to K^i\xrightarrow{\alpha^i} M^{i},
\]
where $\alpha^i\defeq\varphi_*(\beta^i)$. 
Now Proposition \ref{key cpx}\,(2) and the isomorphism $M(\chi)\cong \varphi_*(V_{\X}(\chi))$ in \eqref{gen isom} give an isomorphism $M^0\cong M(\chi)$.    Since $F=F_{\chi}(\delta_0)$,  $(\chi+\upbeta_F^+)^+=\mu_{(\delta,\delta')}(\chi)$ holds, and thus $M^{d_F^++\ell(w_0)-1}\cong M(\mu_{(\delta,\delta')}(\chi))$. 
Therefore, it is enough to show that the morphism $\alpha^i\colon K^i\to M^{i}$ is an $(\add L_{\delta}^F)_{N_{\delta}^F}$-approximation for each $1\leq i\leq d_F^++\ell(w_0)-1$.  

Note that $K^i\in \add L_{\delta}^F$ follows from Proposition \ref{key cpx}\,(2), and thus it suffices to show that the morphism
\begin{equation}\label{exchange seq}
\alpha^i\circ(-)\colon \Hom_R(N_{\delta}^F,K^i)\to \Hom_R(N_{\delta}^F,M^{i})
\end{equation}
is surjective. 
Putting $\cN\defeq \bigoplus_{\chi\in \scrC_{\delta}\backslash\scrC_{\delta}^F}V_{\X}(\chi)$ gives $\varphi_*(\cN)\cong N_{\delta}^F$. Applying $\Hom_{\X}(\cN,-)$ to the exact sequence \eqref{ex seq} for $i=1$ induces the exact sequence
\begin{equation}\label{hom seq}
\Ext_{\X}^m(\cN,B_{F,\chi}^1)\to \Ext_{\X}^m(\cN,\cE_\chi^{1})\to\Ext_{\X}^{m+1}(\cN,\cE_\chi^0)
\end{equation}
for all $m\in \bZ$. Since $L_{\delta}^F\in \add M_{\delta_0}$, Theorem \ref{vanishing} implies  vanishing
$
\Ext_{\X}^m(\cN,B_{F,\chi}^i)\cong 0
$
for all $m>0$ and all $1\leq i\leq d_F^++\ell(w_0)-1$. Since $\cE_\chi^0\cong V_{\X}(\chi)$ and $\chi\in \scrC_{\delta}$, Theorem \ref{vanishing} also shows that $\Ext_{\X}^m(\cN,\cE_\chi^0)\cong 0$ for all $m>0$, and thus  $\Ext_{\X}^m(\cN,\cE_\chi^1)\cong 0$ for all $m>0$. 
Repeating similar arguments shows that
\[
\Ext_{\X}^m(\cN,\cE_\chi^i)\cong 0
\]
for all $m>0$ and all $1\leq i\leq d_F^++\ell(w_0)-1$. Since $X$ is generic, the functor $\varphi_*\colon \refl \X\to \refl R$ is fully faithful. Consider the following  commutative diagram
\[\begin{tikzcd}
\Hom_{\X}(\cN, B_{F,\chi}^i)\arrow[rr,"\beta^i\circ(-)"]\arrow[d,"\mbox{$\varphi_*$}"']&&\Hom_{\X}(\cN,\cE_\chi^{i})\arrow[d, "\mbox{$\varphi_*$}"]\arrow[rr]&&\Ext^1_{\X}(\cN,\cE_\chi^{i-1})=0\\
\Hom_R(N_{\delta}^F,K^i)\arrow[rr, "\alpha^i\circ(-)"]&&\Hom_R(N_{\delta}^F,M^{i}),&&
\end{tikzcd}\]
where the vertical arrows are isomorphisms, and the top sequence is exact. This shows that the morphism \eqref{exchange seq} is surjective, and hence $\alpha^i$ is an $(\add L_{\delta}^F)_{N_{\delta}^F}$-approximation. \vspace{2mm}\\
(2) This follows from (1),  Lemma \ref{mu lem}\,(3) and Lemma \ref{exchange rem}\,(3).\vspace{2mm}\\
(3) This follows from (2) and \eqref{decomp}.
\end{proof}
\end{thm}

\begin{rem}
By the assumption that $X$ is generic, for each $F\in \scrF_{(\delta,\delta')}$, automatically $d_F^++\ell(w_0)-1\geq1$ holds. Indeed, if $G$ is a torus, since  $X$ is quasi-symmetric and $\bsigma$ spans $\M_{\bR}$, $d_F^+=1$ implies that $G=\bG_m$. But then $X$ is not generic, since the divisor $\{x_i=0\}$ is contained in $X\backslash X^{\rm ts}$, where $x_i$ is an eigenvector with weight $\upbeta_i$ such that $\scrW_F^+=\{\upbeta_i\}$.
If $G$ is not a torus, $\ell(w_0)\geq1$ and $d_F^+\geq1$ holds, and thus $d_F^++\ell(w_0)-1\geq1$.
\end{rem}

\subsection{Mutations in toric case.} \label{mut toric} This section gives applications of the results in the previous section to torus actions. 
Keep notation and assumptions in the previous sections, and assume that $G$ is an algebraic torus. Under this assumption,  $\M^+=\M$ and $\bnabla=(1/2)\bsigma$ hold, and $W$ is trivial. 
Let $(\delta,\delta')$  be an adjacent pair in  $\M_{\bR}\backslash\scrH$ separated by a hyperplane $H\in \scrH$, and denote by $\delta_0$ the point that meets the line segment with endpoints $\delta$ and $\delta'$. Since $H\in \scrH$, there is a unique facet $F\in \scrF(\delta_0+\bnabla)$ such that $F$ is parallel to $H$ and $F\cap (\delta+\bnabla)\neq \emptyset$. 

\begin{lem}\label{toric boundary} With notation  same as above,  the following hold.
\begin{itemize}
\item[(1)] $F=F_{\chi}(\delta_0)$ for any $\chi\in \scrC_{\delta}\backslash \scrC_{\delta'}$.
\item[(2)] $\scrF_{(\delta,\delta')}=\{F\}$ holds.
\end{itemize}
\begin{proof}
(1) Let $\chi\in  \scrC_{\delta}\backslash \scrC_{\delta'}$. By Lemma \ref{chi boundary},  the facet $F_{\chi}(\delta_0)$ exists. It suffices to show that $F_{\chi}(\delta_0)$ ia parallel to $H$. To see this,  denote by $B_0\in \scrB(\delta_0+\bnabla)$ the supporting hyperplane such that $B_0\cap (\delta_0+\bnabla)=F_{\chi}(\delta_0)$. Then $B\defeq B_0+\delta-\delta_0\in \scrB(\delta+\bnabla)$  never contain $\chi$ by Proposition \ref{nabla sigma}\,(2). If $F_{\chi}(\delta_0)$ is not parallel to $H$, by a small perturbation of the pair $(\delta,\delta')$, we may assume that $\delta'-\delta$ is parallel to $F_{\chi}(\delta_0)$. But then $\delta-\delta_0$ is parallel to $B_0$, and so $B=B_0$ holds, which contradicts to $\chi\not\in B$.
\vspace{1mm}\\
(2) By (1), $F\in \scrF_{(\delta,\delta')}$. Conversely, for any $F'\in  \scrF_{(\delta,\delta')}$, there exists $\chi\in \scrC_{\delta}\backslash\scrC_{\delta'}$ such that $F'=F_{\chi}(\delta_0)$. Again $F=F_{\chi}(\delta_0)$ by (1), and so $F'=F$. 
\end{proof}
\end{lem}

\begin{lem}\label{l delta}
For any $\chi\in \scrC_{\delta}\backslash\scrC_{\delta'}$ and any  $\upbeta\in {\tiny \gwedge}^*\scrW_F^+$, it holds that $\chi+\upbeta\in \scrC_{\delta}\cap\scrC_{\delta'}$.
\begin{proof}
Lemma \ref{toric boundary}\,(1) implies that $\chi\in F$. Thus Lemma \ref{reflection} shows that 
\[
\chi+\rho=\delta_0-\frac{1}{2}\upbeta_F^++\sum_{\upbeta_{i_j}\in\scrW_F^0}a_j\upbeta_{i_j}
\]
for some $-1/2\leq a_j\leq 0$. Let us write $\scrW_F^+=\{\upbeta'_1,\cdots,\upbeta'_{d_F^+}\}$. Then there exists a proper subset $S\subsetneq \{1,\hdots,d_F^+\}$ such that $\upbeta=\sum_{i\in S}\upbeta'_i$. Pick an  element $t\in \{1,\hdots,d_F^+\}\backslash S$. Then we may assume that there exists a sufficiently small $0<\varepsilon \ll 1$ such that $\delta_0=\delta+\varepsilon \upbeta'_{t}$, since $\upbeta'_t$ is not parallel to $H$. Therefore,
\[
\chi+\upbeta=-\rho+\delta-\frac{1}{2}\sum_{i\not\in S\cup\{t\}}\upbeta'_i+\frac{1}{2}\sum_{i\in S}\upbeta'_i+\left(\varepsilon -\frac{1}{2}\right)\upbeta'_t+\sum_{\upbeta_{i_j}\in\scrW_F^0}a_j\upbeta_{i_j}.
\]
Thus $\chi+\upbeta\in\scrC_{\delta}$. If $\chi+\upbeta\not\in \scrC_{\delta'}$, $F_{\chi+\upbeta}(\delta_0)$ exists by Lemma \ref{chi boundary}. Then Lemma \ref{toric boundary}\,(2) shows that necessarily $\chi+\upbeta\in F$, which can not occur since $(\chi+\upbeta)-\chi=\upbeta$ is not parallel to $H$. Hence $\chi+\upbeta\in \scrC_{\delta'}$.
\end{proof}
\end{lem}

\begin{lem}\label{n delta}
The equality $\scrC_{\delta}^F=\scrC_{\delta} \backslash \scrC_{\delta'}$ holds, and hence $\scrC_\delta \backslash \scrC_{\delta}^F= \scrC_{\delta} \cap \scrC_{\delta'}$.
\begin{proof}
Let $\chi\in \scrC_{\delta}^F$. Then $\chi\in F$, and so $\chi\not\in \scrC_{\delta'}$ by Lemma \ref{chi boundary}. The other inclusion $\scrC_{\delta}\backslash \scrC_{\delta'}\subset \scrC_{\delta}^F$ follows from  Lemma \ref{toric boundary}\,(1). 
\end{proof}
\end{lem}

\begin{thm} \label{mutation MMA toric}
With notation same as above, the following isomorphisms hold up to additive closure.

\begin{itemize}
\item[(1)] 
$
M_{\delta}\cong \mu_{M_{\delta\cap\delta'}}^{d_{F}^+-1}(M_{\delta'})\cong\mu_{M_{\delta\cap\delta'}}^{-d_{F^{*}}^++1}(M_{\delta'})
$.
\vspace{1mm}\item[(2)]
$
M_{\delta'}\cong \mu_{M_{\delta\cap\delta'}}^{d_{F^{*}}^+-1}(M_{\delta})\cong \mu_{M_{\delta\cap\delta'}}^{-d_{F}^++1}(M_{\delta}).
$
\end{itemize}
In particular, the iterated mutation of $M_{\delta}$ at $M_{\delta\cap\delta'}$ is periodic, and the periodicity is $d_F^+ + d_{F^*}^+ - 2$.
\begin{proof}
By Lemma \ref{n delta}, $N_{\delta}^F=M_{\delta\cap\delta'}$. By Lemma \ref{l delta}, $L_{\delta}^F\in \add M_{\delta\cap\delta'}$. Moreover by \eqref{delta' decomp}, $M_{\delta'}=M_{\delta\cap \delta'}\oplus M_{\delta'}^{F^*}$. 
Therefore, Theorem \ref{main thm} gives an isomorphism 
\begin{equation}
M_{\delta}\cong \mu_{M_{\delta\cap \delta'}}^{d_F^+-1}(M_{\delta'}),
\end{equation}
which yields an isomorphism $M_{\delta'}\cong\mu_{M_{\delta\cap\delta'}}^{-d_{F}^++1}(M_{\delta})$ by Theorem \ref{IW mutation}\,(1). The remaining isomorphisms follow from the identical argument for the swapped adjacent pair $(\delta',\delta)$.
\end{proof}
\end{thm}

\begin{prop}\label{composition mut}
Set $N\defeq M_{\delta\cap\delta'}$ and $d\defeq d_F^+$.
Assume that $\delta'-\delta$ has the same orientation as $\ell$ with respect to $H$. 
\begin{itemize}
\item[$(1)$]  There exists a tilting bundle $\scrV_i$ on $\X$ such that 
\[
\varphi_*(\scrV_i)\cong \mu_N^{-i}(M_{\delta})
\]
for each $1\leq i\leq d-1$.
\item[$(2)$] For any $i\geq 0$, the $\Lambda_i\defeq \End(\mu_N^{-i}(M_{\delta}))$-module $T_i\defeq \Hom(\mu_N^{-i}(M_{\delta}),\mu_N^{-i-1}(M_{\delta}))$ is a tilting module, and the mutation functor 
\[
\Phi_N\colon\Db(\fmod \Lambda_i)\simto \Db(\fmod \Lambda_{i+1})
\] 
is given by $\RHom(T_i,-)$.
\item[$(3)$] The following diagram commutes
\[
\begin{tikzpicture}[xscale=1.3]
\node (A0) at (3,0) {$\Db(\fmod\Lambda_{\delta})$};
\node (A5) at (7,0) {$\Db(\fmod\Lambda_{\delta'})$,};
\node (B4) at (5,1.5) {$\Db(\coh\X)$};
\draw[->] (A0) -- node[above] {$\scriptstyle\Phi_N^{d-1}$}(A5);
\draw[->] (B4) -- node[auto=right] {$\scriptstyle\RHom_{}(\scrV_{\delta},-)$}(A0);
\draw[->] (B4) -- node[auto=left] {$\scriptstyle\RHom_{}(\scrV_{\delta'},-)$}(A5);
\end{tikzpicture}
\]
where the bottom functor denotes the composition of mutation fucntors as in \eqref{iterated mutation}.
\end{itemize}
\begin{proof}
This proof adopts the same notation as in the proof of Theorem \ref{main thm}.
For the unique facet $F$, taking the direct sum of the exact sequences \eqref{long ex seq 4.9} for all $\chi \in \scrC_{\delta}^F$ gives
\begin{equation*} \label{for tilting mutation long}
0 \to 
\bigoplus_{\chi \in \scrC_{\delta}^F} V_\X(\chi) 
\to 
\bigoplus_{\chi \in \scrC_{\delta}^F} B^1_{F, \chi} 
\to \cdots \to 
\bigoplus_{\chi \in \scrC_{\delta}^F} B^{d_F^+ - 1}_{F, \chi} 
\to 
\bigoplus_{\chi \in \scrC_{\delta}^F} V_\X\left(\mu_{(\delta, \delta')}(\chi) \right) \to 0,
\end{equation*}
where $\mu_{(\delta, \delta')}(\chi) = (\chi + \upbeta^+_F)^+ = \chi + \upbeta^+_F$ as in Section \ref{section combinatiorics},
and this exact sequence is the splicing of the short exact sequences
\begin{equation} \label{for tilting mutation}
    0 \to \bigoplus \cE_\chi^{i-1} \to \bigoplus B^i_{F,\chi} \to \bigoplus \cE_\chi^i \to 0.
\end{equation}
Now Lemma \ref{n delta} implies that $\cN = \bigoplus_{\chi \in \scrC_{\delta} \cap \scrC_{\delta'}} V(\chi)$, 
and then Lemma \ref{n delta}, Proposition \ref{crossing separrationg hyperplane}, and Proposition \ref{CSW2} show that 
\[ \text{$\scrV_{\delta} = \cN \oplus \bigoplus_{\chi \in \scrC_{\delta}^F} V(\chi)$ \hspace{3mm} and \hspace{3mm} $\scrV_{\delta'} = \cN \oplus \bigoplus_{\chi \in \scrC_{\delta}^F} V(\mu_{(\delta, \delta')}(\chi))$.} \]
In addition, Lemma \ref{l delta} combined with Proposition \ref{key cpx}\,(2) shows that $B^i_{F, \chi} \in \add \cN$ for all $\chi \in \scrC_{\delta}^F$ and all $1 \leq i \leq d_F^+ -1$.
Then applying \cite[Lemma 6]{hara22}\footnote{\cite[Lemma 6]{hara22} is stated only for schemes, but the proof also works for stacks.}
implies that the bundle
\[ \scrV_i \defeq \cN \oplus \bigoplus \cE_\chi^i \]
is tilting for all $0 \leq i \leq d_F^+ -1$.
Since $N \simeq \phi_*(\cN)$ by definition, applying Proposition \ref{tilting mutation} repeatedly to the sequences \eqref{for tilting mutation} gives the results (1), (2) and (3).
\end{proof}
\end{prop}

In the following corollaries, put $N\defeq M_{\delta\cap\delta'}$ and $d\defeq d_F^+$.

\begin{cor}\label{toric main}

Assume that $\delta'-\delta$ has the same orientation as $\ell$ with respect to $H$.
\begin{itemize}
\item[$(1)$] The following diagram of equivalences commutes:
\[
\begin{tikzpicture}[xscale=1.3]
\node (A0) at (2,0) {$\Db(\fmod\Lambda_{\delta})$};
\node (A5) at (8,0) {$\Db(\fmod\Lambda_{\delta'})$};
\node (B0) at (2,1.5) {$\scrM(\delta+\bnabla)$};
\node (B4) at (4.75,1.5) {$\Db(\coh\X)$};
\node (B5) at (8,1.5) {$\scrM(\delta'+\bnabla)$};
\draw[->] (A0) -- node[above] {$\scriptstyle\Phi_N^{d-1}$}(A5);
\draw[->] (B0) -- node[above] {$\scriptstyle\res$}(B4);
\draw[->] (B4) -- node[above] {$\scriptstyle\res^{-1}$}(B5);
\draw[->] (B0) -- node[left] {$\scriptstyle\RHom_{}(\scrT_{\delta},-)$}(A0);
\draw[->] (B5) -- node[right] {$\scriptstyle\RHom_{}(\scrT_{\delta'},-)$}(A5);
\end{tikzpicture}
\]
\item[$(2)$] Write $T\defeq\Hom_R(M_{\delta},M_{\delta'})$. Then $\RHom(T,-)\cong \Phi_N^{d-1}$ holds.
\end{itemize}
\begin{proof}
(1) Since  $\scrV_{\delta}= \res(\scrT_{\delta})$  by definition, there exists a natural isomorphism
\begin{equation}\label{res tilt}
\RHom(\scrT_{\delta},-)\cong \RHom(\scrV_{\delta},-)\circ \res
\end{equation}
of functors. Therefore, the result follows from Proposition \ref{composition mut}\,(2).\vspace{1mm}\\
(2) Since $M_{\delta}\cong \varphi_*(\scrV_{\delta})$, there is an isomorphism $T\cong T_{(\delta,\delta')}=\Hom_{\X}(\scrV_{\delta},\scrV_{\delta'})$. Thus (1) and Theorem \ref{wall crossing}\,(1) give isomorphisms $\RHom(T,-)\cong \RHom(T_{(\delta,\delta')},-)\cong \Phi_N^{d-1}$.
\end{proof}
\end{cor}

The next corollary is an analogy of \cite[Theorem 4.2]{HomMMP}. 

\begin{cor} \label{HomMMP4.2 toric}
Let $\ell_1,\ell_2\in (\M_{\bR})_{\rm gen}$ be generic elements, and set $\X_i\defeq [X^{\rm ss}(\ell_i)/G]$. If $\delta'-\delta$ has the same orientation as $\ell_2$, then the following diagram commutes.
\[
\begin{tikzpicture}[xscale=1.3]
\node (A0) at (2,0) {$\Db(\fmod\Lambda_{\delta})$};
\node (A5) at (8,0) {$\Db(\fmod\Lambda_{\delta'})$};
\node (B0) at (2,1.5) {$\Db(\coh\X_1)$};
\node (B4) at (4.75,1.5) {$\scrM(\delta+\bnabla)$};
\node (B5) at (8,1.5) {$\Db(\coh\X_2)$};
\draw[->] (A0) -- node[above] {$\scriptstyle\Phi_N^{d-1}$}(A5);
\draw[->] (B0) -- node[above] {$\scriptstyle\res^{-1}$}(B4);
\draw[->] (B4) -- node[above] {$\scriptstyle\res$}(B5);
\draw[->] (B0) -- node[left] {$\scriptstyle\RHom_{}(\scrV_{\delta},-)$}(A0);
\draw[->] (B5) -- node[right] {$\scriptstyle\RHom_{}(\scrV_{\delta'},-)$}(A5);
\end{tikzpicture}
\]
\begin{proof}
Proposition \ref{composition mut}\,(2) gives an isomorphism
\[
\RHom(\scrV_{\delta'},-)\cong \Phi_N^{d-1}\circ\RHom(\scrV_{\delta},-)
\]
of functors from $\Db(\coh \X_2)$ to $\Db(\fmod\Lambda_{\delta'})$. Therefore the result follows from the following isomorphisms
\begin{align*}
\RHom(\scrV_{\delta'},-)\circ \res^{-1}\circ \res&\cong\Phi_N^{d-1}\circ\RHom(\scrV_{\delta},-)\circ \res^{-1}\circ \res\\
&\cong\Phi_N^{d-1}\circ\RHom(\scrT_{\delta},-)\circ \res\\
&\cong\Phi_N^{d-1}\circ\RHom(\scrV_{\delta},-),
\end{align*}
where the second and  third isomorphisms follow from \eqref{res tilt}.
\end{proof}
\end{cor}

\begin{rem}
With the same notations as in Proposition \ref{composition mut}, put $\X' \defeq [X^{\rm ss}(-\ell)/G]$.
The composition
\begin{align*} \phi_{\delta, \delta'} \colon 
\Db(\coh\X) \xrightarrow{\res^{-1}} \scrM(\delta+\bnabla) \xrightarrow{\res} \Db(\coh\X') \xrightarrow{\res^{-1}} \scrM(\delta'+\bnabla) \xrightarrow{\res} \Db(\coh\X) 
\end{align*}
gives a natural autoequivalence of $\Db(\coh\X)$ associated to the wall-crossing.
Similarly, there is the autoequivalence of $\Db(\coh\X')$ that is given by
\begin{align*} \phi'_{\delta', \delta} \colon 
\Db(\coh\X') \xrightarrow{\res^{-1}} \scrM(\delta' +\bnabla) \xrightarrow{\res} \Db(\coh\X) \xrightarrow{\res^{-1}} \scrM(\delta+\bnabla) \xrightarrow{\res} \Db(\coh\X').
\end{align*}

Note that Proposition \ref{composition mut} also implies that 
$M_{\delta'} \cong \mu_{M_{\delta\cap\delta'}}^{-d_{F}^++1}(M_{\delta})$.
Then Proposition \ref{composition mut} and Corollary \ref{HomMMP4.2 toric}  yield the following commutative diagram

\[
\begin{tikzpicture}[xscale=1.3]
\node (A0) at (3,0) {$\Db(\fmod\Lambda_{\delta})$};
\node (A5) at (5,0) {$\Db(\fmod\Lambda_{\delta'})$};
\node (A10) at (7,0) {$\Db(\fmod\Lambda_{\delta})$};
\node (B4) at (3,1.8) {$\Db(\coh\X)$};
\node (B8) at (5,1.8) {$\Db(\coh\X)$};
\node (B10) at (7,1.8) {$\Db(\coh\X)$};
\node (C0) at (3, -1.8) {$\Db(\coh\X')$};
\node (C1) at (5, -1.8) {$\Db(\coh\X')$};
\node (C2) at (7, -1.8) {$\Db(\coh\X')$,};

\draw[double equal sign distance] (B4) -- (B8);
\draw[double equal sign distance] (C1) -- (C2);
\draw[->] (A0) -- node[above] {$\scriptstyle\Phi_N^{d-1}$}(A5);
\draw[->] (A5) -- node[above] {$\scriptstyle\Phi_N^{d^*-1}$}(A10);

\draw[->] (B8) -- node[above] {$\scriptstyle\phi_{\delta, \delta'}$}(B10);
\draw[->] (C0) -- node[above] {$\scriptstyle\phi'_{\delta', \delta}$}(C1);

\draw[->] (B4) -- node[left] {$\scriptstyle\RHom_{}(\scrV_{\delta},-)$}(A0);
\draw[->] (B8) -- node[left] {$\scriptstyle\RHom_{}(\scrV_{\delta'},-)$}(A5);
\draw[->] (B10) -- node[left] {$\scriptstyle\RHom_{}(\scrV_{\delta},-)$}(A10);

\draw[->] (C0) -- node[left] {$\scriptstyle\RHom_{}(\scrV_{\delta},-)$}(A0);
\draw[->] (C1) -- node[left] {$\scriptstyle\RHom_{}(\scrV_{\delta'},-)$}(A5);
\draw[->] (C2) -- node[left] {$\scriptstyle\RHom_{}(\scrV_{\delta},-)$}(A10);

\end{tikzpicture}
\]
where $d \defeq d_F^+$ and $d^* \defeq d_{F^*}^+$.
Thus, under the identifications of categories
\[ \Db(\coh\X) \xrightarrow{\RHom_{}(\scrV_{\delta},-)} \Db(\fmod\Lambda_{\delta}) \xleftarrow{\RHom_{}(\scrV_{\delta},-)} \Db(\coh\X'), \]
the natural autoequivalences $\phi_{\delta, \delta'}$ and $\phi'_{\delta', \delta}$ correspond to $\Phi_N^{d + d^* -1} \in \Auteq \Db(\fmod\Lambda_{\delta})$.

In some example including threefold flops (cf.~\cite{HomMMP}) and certain higher dimensional flops (cf.~\cite{hara17, hara22}), 
it is observed that a natural autoequivalence associated to the geometry of a contraction (eg. spherical twists along the structure sheaf of the flopping locus) can be recovered as the autoequivalence associated to the periodicity of mutations of an NCCR.
The diagram above should be regarded as a version of that phenomenon in the toric GIT setting.
\end{rem}

\subsection{Descending to toric hyperK\"{a}hler varieties} \label{section: THK}
This section studies a class of varieties so called toric hyperK\"{a}hler.
A quick reminder for the geometry of toric hyperK\"{a}hler varieties is given in the following, but for more details see \cite[Section 3]{svdbhypertoric}.

Let $X$ be a symplectic representation  of an algebraic torus $G = T$.
Then $X$ admits a canonical moment map $\upmu \colon X \to \mathfrak{t}^*$, where $\mathfrak{t} \defeq \mathrm{Lie}(T)$.
Throughout let us assume that the action of $T$ on $X$ is faithful, which gives that $\upmu$ is flat and surjective.

For $\xi \in \mathfrak{t}^*$ and generic $\ell \in \M_{\bR}$, the quotient stack $[\upmu^{-1}(\xi)^{\mathrm{ss}}(\ell)/T]$ is smooth Deligne-Mumford, and 
the natural morphism 
$\varphi^{\xi} \colon [\upmu^{-1}(\xi)^{\mathrm{ss}}(\ell)/T] \to \upmu^{-1}(\xi)/\!\!/T$
gives a stacky crepant resolution of the singular symplectic variety $\upmu^{-1}(\xi)/\!\!/T$.
Algebraic varieties of the form $\upmu^{-1}(\xi)^{\mathrm{ss}}(\ell)/\!\!/T$
are called \textit{toric hyperK\"{a}hler varieites} or \textit{hypertoric varieties} in the literature.

For the inclusion $i \colon [\upmu^{-1}(\xi)^{\mathrm{ss}}(\ell)/T] \hookrightarrow [X^{\mathrm{ss}}(\ell)/T]$, \cite{svdbhypertoric} proved the following.

\begin{thm}
The restriction $i^*\cV$ of any tilting bundle $\cV$ on $[X^{\mathrm{ss}}(\ell)/T]$ to $[\upmu^{-1}(\xi)^{\mathrm{ss}}(\ell)/T]$ remains to be tilting.
\end{thm}

This holds basically because, in the toric setting, the stack $[\upmu^{-1}(\xi)^{\mathrm{ss}}(\ell)/T]$ is cut out from $[X^{\mathrm{ss}}(\ell)/T]$ by a regular sequence.
Note that, in \cite{svdbhypertoric}, the theorem is claimed only for tilting bundles constructed by \cite{{hl-s}}, but the proof does work for any tilting bundles.

With the same notation as in Proposition \ref{composition mut}, put
\[ \text{$\scrV_{\delta}^\xi \defeq i^* \scrV_{\delta}$, $\cN^{\xi} \defeq i^*\cN$, $\Lambda_{\delta}^{\xi} \defeq \End(\scrV_{\delta}^\xi)$, $M_{\delta}^\xi \defeq \varphi^\xi_* \scrV_{\delta}^\xi$, and $N^{\xi} \defeq \varphi^{\xi}_* \cN^{\xi}$.} \]

Since an exact sequence of locally free sheaves is locally split, its arbitrary restriction remains to be exact.
Thus all the arguments in the proof of Proposition \ref{composition mut} still work when one restricts everything to $[\upmu^{-1}(\xi)^{\mathrm{ss}}(\ell)/T]$.
This observation gives the following.

\begin{thm}
The same statement as in Theorem \ref{mutation MMA toric}\,(1), (2) and \ref{composition mut}\,(1), (3) holds for 
$[\upmu^{-1}(\xi)^{\mathrm{ss}}(\ell)/T]$, 
$\scrV_{\delta}^\xi$, 
$\cN^{\xi}$,
$M_{\delta}^\xi$, 
and $N^{\xi}$.
\end{thm}

This theorem generalises results in \cite[Section 5.2.2]{hara17} about Mukai flops to arbitrary toric hyperK\"{a}hler varieties.
Note that Proposition \ref{composition mut}\,(2) does not hold in general for toric hyperK\"{a}hler varieties. If the singular locus of $\upmu^{-1}(\xi)/\!/ T$ has codimension at least $3$, the same statement as in Proposition \ref{composition mut}\,(2) holds.
For more details, see Remark \ref{counterexample for 2.21}.

\section{Applications to Calabi-Yau complete intersections}
\label{section: CICY}

Let $a_1,\hdots, a_n>0$ be  positive integers.
 Write $S\defeq \bC[x_1,\hdots,x_n]$ for the $\bZ$-graded polynomial ring with $\deg(x_i)=a_i$. Let $d_1,\hdots,d_r$ be positive integers such that $\sum_{i=1}^rd_i=\sum_{i=1}^n a_i$. 
Let  $f_1,\hdots,f_r\in S$ be homogeneous elements with $\deg(f_i)=d_i$ such that these define a Calabi-Yau complete intersection
\[
Y\defeq \bigcap_{i=1}^r\{f_i=0\}\subset {\bf P}(a_1,\hdots,a_n)
\]
in the weighted projective stack ${\bf P}(a_1,\hdots,a_n)\defeq [\bA^n\backslash \{{\bf 0}\}/\bG_m]$. 
This section shows that a certain fundamental group action on $\Db(\coh Y)$ can be interpreted in terms of Iyama--Wemyss mutations via noncommutative matrix factorizations. 
See Appendix \ref{mf appendix} for fundamental definitions and  properties of  matrix factorizations.

Set  $G_i\defeq \bG_m$ for $i=1,2$,  and   define a 2-dimensional torus $G\defeq G_1\times G_2$. Let
\[
X\defeq\sfk^{n+1}_{\bf x}\oplus \sfk^{r+1}_{\bf y}
\] 
be the direct sum of  vector spaces with coordinate ${\bf x}=(x_0,\hdots,x_n)$ and ${\bf y}=( y_0,\hdots,y_r )$. 
Consider a $G$-action on $X$  defined by 
\begin{align*}
G\times X&\ni (s,t)\times \bigl((x_0,\hdots,x_n),(y_0,\hdots,y_r)\bigr)\\&\mapsto 
\bigl((sx_0, s^{a_1}x_1,\hdots,s^{a_n}x_n), (s^{-1}ty_0, s^{-d_1}ty_1,\hdots, s^{-d_r}ty_r)\bigr)\in G\times X.
\end{align*}
The $G$-representation $X$ is not quasi-symmetric, but it is a generic quasi-symmetric representation with respect to the $G_1$-action.
Then the associated hyperplane arrangement $\scrH$ in $\M_{\bR}=\Hom(G_1,\bG_m)_{\bR}\cong \bR$ is 
\[
\scrH=\begin{cases}
\bZ+(1/2)& \mbox{if } \sum_{i=1}^n a_i \mbox{ is even}.\\
\bZ & \mbox{if } \sum_{i=1}^n a_i \mbox{ is odd}.
\end{cases}
\]

Consider a $G_1$-invariant and $G_2$-semi-invariant regular function
\[
W\defeq x_0y_0+f_1y_1+\cdots+f_ry_r
\]
on $X$, which induces a regular function on $[X/G_1]$. Then $\bigl([X/G_1],\chi,W\bigr)^{G_2}$ is a gauged Landau-Ginzburg model, where $\chi\colon G_2\to \bG_m$ is the identity map.   
Define a locally free sheaf $\cE$ on ${\bf P}(1,a_1,\hdots,a_n)$ by 
\[\cE\defeq\cO(-1)\oplus\cO(-d_1)\oplus\cdots\oplus\cO(-d_r).\] 
Since $Y$ is the zero locus of  a regular section $s\defeq(x_0,f_1,\hdots,f_r)\in \Gamma({\bf P}(1,a_1,\hdots,a_n),\cE^{*})^{G_2}$, 
by Kn\"orrer periodicity \cite{isik,shipman,H2}, there is an equivalence
\begin{equation}\label{derived knorrer}
\Db(\coh Y)\cong \Dcoh_{G_2}({\rm V}(\cE),\chi,W),
\end{equation}
where ${\rm V}(\cE)$ is the vector bundle associated to $\cE$, and $G_2$-action on ${\rm V}(\cE)$ is induced from the $G_2$-equivariant structure of the $G_2$-equivariant locally free sheaf $\cE(\chi)$ on the stack ${\bf P}(1,a_1,\hdots,a_n)$ with the trivial $G_2$-action.
Denote by $\ell\colon G_1\to \bG_m$ the identity map, which is a generic element in  $\M_{\bR}$, and for this $\ell$, set $\X_{\pm}\defeq [X^{\rm ss}(\pm\ell)/G_1]$. Then there is an isomorphism ${\rm V}(\cE)\cong \X_{+}$. 

For $\delta\in \M_{\bR}\backslash \scrH$, 
a {\it magic window} $\scrM(\delta+\bnabla,W)\subset \Dcoh_{G_2}([X/G_1],\chi,W)$ is defined by 
\[
\scrM(\delta+\bnabla,W)\defeq \left\l E_1\to E_0\to E_1(\chi)\relmiddle|  \Res^{G_2}(E_i)\in \scrM(\delta+\bnabla)\right\r,
\]
where $\Res^{G_2}\colon \coh_{G_2}[X/G_1]\to \coh[X/G_1]$ is the restriction functor. Then a version of Theorem \ref{hl-s} holds for $\scrM(\delta+\bnabla,W)$ by \cite[Section 6.3]{hl-s} or \cite{vgit}:
\begin{equation}\label{mf window}
\res\defeq i^*\colon \scrM(\delta+\bnabla,W)\simto \Dcoh_{G_2}(\X_{\pm},\chi,W).
\end{equation}
Therefore, \eqref{derived knorrer} and \eqref{mf window} gives an equivalence 
\begin{equation}\label{knorrer}
\Db(\coh Y)\cong \scrM(\delta+\bnabla,W).
\end{equation}
If $[\delta]\xrightarrow{\ell}[\delta']\in \Pi_1(\bC\backslash\scrH_{\bC})$ is a positive labeled arrow,
set
\[
\uprho\left([\delta]\xrightarrow{\ell}[\delta']\right)
\defeq \scrM(\delta+\bnabla,W)\xrightarrow{\res} \Dcoh_{G_2} (\X_+,\chi,W)\xrightarrow{\res^{-1}}\scrM(\delta'+\bnabla,W),\]
and if $[\delta]\xrightarrow{-\ell}[\delta']\in \Pi_1(\bC\backslash\scrH_{\bC})$ is a positive labeled arrow,
set
\[
\uprho\left([\delta]\xrightarrow{-\ell}[\delta']\right)
\defeq \scrM(\delta+\bnabla,W)\xrightarrow{\res} \Dcoh_{G_2} (\X_-,\chi,W)\xrightarrow{\res^{-1}}\scrM(\delta'+\bnabla,W).\]
Then these assignment  define a fundamental groupoid action of $\Pi_1(\bC\backslash\scrH_{\bC})$ on magic windows \cite[Corollary 6.11]{hl-s}, and in particular there is a fundamental group action 
\begin{equation}\label{fund action 1}
\uprho\colon \pi_1(\bC\backslash\scrH_{\bC})\to \Auteq \scrM(\delta+\bnabla,W)\simto\Auteq \Db(\coh Y)
\end{equation}
on $\Db(\coh Y)$ via the equivalence \eqref{knorrer}.  
By an identical manner as the action \eqref{ext fund action}, 
it is possible to extend the action \eqref{fund action 1} to the action of $\pi_1\bigl((\bC\backslash\scrH_{\bC})/\bZ\bigr)$ \cite[Corollary 6.11]{hl-s}. 
Thus, since  $(\bC\backslash\scrH_{\bC})/\bZ\cong \bP^1\backslash \{0,1,\infty\}$,  there is a group action
\begin{equation}\label{fund action 2}
\widetilde{\uprho}\colon \pi_1\bigl(\bP^1\backslash \{0,1,\infty\}\bigr)\to \Auteq \Db(\coh Y).
\end{equation}

Since  $W\in R\defeq\sfk[X]^{G_1}$ and $W\in \Gamma(\Spec R,\cO(\chi))^{G_2}$,  it defines a noncommutative gauged LG model $(\Lambda_{\delta},\chi,W)^{G_2}$ for any $\delta\in \M_{\bR}\backslash \scrH$. 
Since the abelian category $\Qcoh_{G_2}[X/G_1]\cong \Qcoh_G X$ has finite injective resolutions, the functor 
\[\Hom_{[X/G_1]}(\scrT_{\delta},-)\colon \coh_{G_2}([X/G_1],\chi,W)\to\fmod_{G_2}(\Lambda_{\delta},\chi,W)\] defines the right derived functor
\begin{equation}\label{tilt mf}
\RHom_{[X/G]}(\scrT_{\delta},-)\colon \Dcoh_{G_2}([X/G_1],\chi,W)\to\Dmod_{G_2}(\Lambda_{\delta},\chi,W). 
\end{equation}
Similarly, the functor $\Hom_{\X_+}(\scrV_{\delta},-)\colon \coh_{G_2}(\X_+,\chi,W)\to \fmod_{G_2}(\Lambda_{\delta},\chi,W)$ induces its right derived functor
\begin{equation}\label{tilt mf 2}
\RHom_{\X_+}(\scrV_{\delta},-)\colon \Dcoh_{G_2}(\X_+,\chi,W)\to\Dmod_{G_2}(\Lambda_{\delta},\chi,W). 
\end{equation}
\begin{prop}
The restriction of \eqref{tilt mf}
\begin{equation}\label{window nccr}
\RHom_{[X/G]}(\scrT_{\delta},-)\colon \scrM(\delta+\bnabla,W)\simto\Dmod_{G_2}(\Lambda_{\delta},\chi,W)
\end{equation}
 to a magic window and the functor \eqref{tilt mf 2}
 \[
\RHom_{\X_+}(\scrV_{\delta},-)\colon \Dcoh_{G_2}(\X_+,\chi,W)\simto\Dmod_{G_2}(\Lambda_{\delta},\chi,W)
 \]
 are  equivalences.
 \begin{proof}
Since $\res\colon \scrM(\delta+\bnabla,W)\simto\Dcoh_{G_2}(\X_+,\chi,W)$ is an equivalence  and there is  a natural isomorphism $\RHom_{[X/G]}(\scrT_{\delta},-)\cong\RHom_{\X_+}(\scrV_{\delta},-)\circ\res$ of functors, it is enough to prove that the functor  \eqref{tilt mf 2} is an equivalence. Let $T_{\delta}\defeq \bigoplus_{\chi\in\scrC_{\delta}}\cO(\chi)\in \coh_GX$, and set $V_{\delta}\defeq i^*(T_{\delta})\in \coh_{G}X^{\rm ss}(\ell)$. 
Then $V_{\delta}$ corresponds to $\scrV_{\delta}$ via a natural equivalence $\coh_G X^{\rm ss}(\ell)\cong \coh_{G_2}\X_+$, and isomorphisms $\Lambda_{\delta}\cong\End_{\X_+}(\scrV_{\delta})\cong\End_{X^{\rm ss}(\ell)}(V_{\delta})^{G_1}$ hold.
Therefore there exists the following commutative diagram:
\[
\begin{tikzpicture}[xscale=1.3]
\node (A0) at (2,0) {$\Dcoh_G(X^{\rm ss}(\ell),1\times \chi,W)$};
\node (A5) at (8,0) {$\Dmod_{G_2}(\Lambda_{\delta},\chi,W)$};
\node (B4) at (2,2) {$\Dcoh_{G_2}(\X_+,\chi,W)$};
\draw[->] (A0) -- node[above] {$\scriptstyle\RHom(V_{\delta},-)^{G_1}$}(A5);
\draw[->] (B4) -- node[left] {$\cong$}(A0);
\draw[->] (B4) -- node[above] {\scriptsize \eqref{tilt mf 2}}(A5);
\end{tikzpicture}
\]
Since the bottom functor is an equivalence by Theorem \ref{H3 thm}, so is \eqref{tilt mf 2}. 
 \end{proof}
\end{prop}

Let $(\delta,\delta')$ be an adjacent pair in $\bR\,\backslash \scrH$, and let $N\defeq M_{\delta\cap\delta'}$. Then for any $i\geq 0$ the algebra $\Lambda_{i}\defeq\End(\mu_N^{-i}(M_{\delta}))$ is a $G_2$-equivariant finite $R$-algebra. By Proposition \ref{composition mut}\,(2), the $G_2$-equivariant $\Lambda_i$-module $T_i\defeq\Hom_R(\mu_N^{-i}(M_{\delta}),\mu_N^{-i-1}(M_{\delta}))$ is a tilting module, and it defines the mutation functor  
$
\Phi_N=\RHom(T_i,-)\colon\Db(\fmod\Lambda_i)\simto \Db(\fmod\Lambda_{i+1}).
$
Furthermore, since $T_i$ is a $G_2$-tilting module in the sense of Definition \ref{eq tilt},  the mutation functor 
\[
\Phi_N\defeq\RHom(T_i,-)\colon \Dmod_{G_2}(\Lambda_i,\chi,W)\to\Dmod_{G_2}(\Lambda_{i+1},\chi,W)
\]
for derived factorization categories is an equivalence by Theorem \ref{H3 thm}.

The following shows that the equivalences of magic windows generating the group action \eqref{fund action 1} correspond to mutation functors between noncommutative matrix factorizations.
\begin{thm}\label{window shift}
Let $\delta,\delta'\in \bR \,\backslash\scrH$ be an adjacent pair with $\delta<\delta'$, and write $N\defeq M_{\delta\cap\delta'}$. Then the following diagram of equivalences commutes.
\[
\begin{tikzpicture}[xscale=1.3]
\node (A0) at (1,0) {$\Dmod_{G_2}(\Lambda_{\delta},\chi,W)$};
\node (A5) at (5,0) {$\Dmod_{G_2}(\Lambda_{\delta'},\chi,W)$};
\node (A9) at (9,0) {$\Dmod_{G_2}(\Lambda_{\delta},\chi,W)$};
\node (B0) at (1,1.5) {$\scrM(\delta+\bnabla,W)$};
\node (B5) at (5,1.5) {$\scrM(\delta'+\bnabla,W)$};
\node (B9) at (9,1.5) {$\scrM(\delta+\bnabla,W)$};
\draw[->] (A0) -- node[above] {$\scriptstyle\Phi_N^{n}$}(A5);
\draw[->] (A5) -- node[above] {$\scriptstyle\Phi_N^{r}$}(A9);
\draw[->] (B0) -- node[above] {$\scriptstyle\uprho\bigl([\delta]\xrightarrow{\ell}[\delta']\bigr)$}(B5);
\draw[->] (B5) -- node[above] {$\scriptstyle\uprho\bigl([\delta']\xrightarrow{-\ell}[\delta]\bigr)$}(B9);
\draw[->] (B0) -- node[left] {$\scriptstyle\RHom_{}(\scrT_{\delta},-)$}(A0);
\draw[->] (B5) -- node[right] {$\scriptstyle\RHom_{}(\scrT_{\delta'},-)$}(A5);
\draw[->] (B9) -- node[right] {$\scriptstyle\RHom_{}(\scrT_{\delta},-)$}(A9);
\end{tikzpicture}
\]
\begin{proof}
We only show that the left square commutes, since the  commutativity of the right one follows from a similar argument. Consider the following diagram 
\[
\begin{tikzpicture}[xscale=1.3]
\node (A0) at (1,-0.5) {$\Dmod_{G_2}(\Lambda_{\delta},\chi,W)$};
\node (A5) at (7,-0.5) {$\Dmod_{G_2}(\Lambda_{\delta'},\chi,W)$.};
\node (B0) at (1,3) {$\scrM(\delta+\bnabla,W)$};
\node (B5) at (7,3) {$\scrM(\delta'+\bnabla,W)$};
\node (C) at (4,1.5)
{$\Dcoh_{G_2}(\X_+,\chi,W)$};
\draw[->] (A0) -- node[below] {$\scriptstyle\Phi_N^{n}$}(A5);
\draw[->] (B0) -- node[above] {$\scriptstyle\uprho\bigl([\delta]\xrightarrow{\ell}[\delta']\bigr)$}(B5);
\draw[->] (B0) -- node[left] {$\scriptstyle\RHom_{}(\scrT_{\delta},-)$}(A0);
\draw[->] (B5) -- node[right] {$\scriptstyle\RHom_{}(\scrT_{\delta'},-)$}(A5);
\draw[->] (B0) -- node[above] {$\res$}(C);
\draw[->] (B5) -- node[above] {$\res$}(C);
\draw[->] (C) -- node[auto=left] {$\scriptstyle\RHom(\scrV_{\delta},-)$\\}(A0);
\draw[->] (C) -- node[auto=right] {$\scriptstyle\RHom(\scrV_{\delta'},-)$\\}(A5);
\end{tikzpicture}
\]
Then the top triangle commutes by the definition of the map $\uprho$, and the left and right triangles are commutative since $\scrV_{\delta}=\res(\scrT_{\delta})$. Thus it is enough to show that the bottom triangle commutes. But this follows from Proposition \ref{composition mut}\,(2) and \cite[Lemma 4.10]{bdfik}.
\end{proof}
\end{thm}

Since $Y$ is Calabi--Yau,  a line bundle $\cL$ on $Y$ is a spherical object in $\Db(\coh Y)$, and it defines the {\it spherical twist} autoequivalence
\[
\ST_{\cL}\colon \Db(\coh Y)\simto \Db(\coh Y).
\]
The following shows that the spherical twist associated to a line bundle of the form $\cO_Y(m)$ corresponds to the compositions of Iyama--Wemyss mutations via noncommutative matrix factorizations.

\begin{cor}\label{cor:spherical twist}  
For $m\in\bZ$, put $\delta\defeq m+(\alpha/2)+1\in\bR\backslash\scrH$, and  set $N\defeq M_{\delta-1\cap\delta}$. Then the  diagram 
\[
\begin{tikzpicture}[xscale=1.3]
\node (A0) at (1,0) {$\Dmod_{G_2}(\Lambda_{\delta},\chi,W)$};
\node (A5) at (5,0) {$\Dmod_{G_2}(\Lambda_{\delta},\chi,W)$};
\node (B0) at (1,1.5) {$\Db(\coh Y)$};
\node (B5) at (5,1.5) {$\Db(\coh Y)$};
\draw[->] (A0) -- node[above] {$\scriptstyle\Phi_N^{n+r}$}(A5);
\draw[->] (B0) -- node[above] {$\scriptstyle\ST_{\cO(m)}$}(B5);
\draw[->] (B0) -- node[left] {$\cong$}(A0);
\draw[->] (B5) -- node[right] {$\cong$}(A5);
\end{tikzpicture}
\]
commutes, where the vertical equivalences are the compositions of \eqref{knorrer} and \eqref{window nccr}.
\end{cor}

\vspace{4mm}
This corollary can be deduced from Theorem \ref{window shift} and a result in \cite{hl-sh}. Consider one parameter subgroups $\lambda_{\pm}\colon  \bG_m\to G$ defined by $\lambda_{\pm}(t)\defeq (t^{\pm1},1)$, and define
\[
S_{\pm}\defeq\left\{x\in X\relmiddle| \lim_{t\to 0}\lambda_{\pm}(t)x=0\right\}.
\]
Then $X^{\rm ss}_{\pm}\defeq X\backslash S_{\pm}= X^{\rm ss}(\mp\ell)$ and $E\defeq X^{\rm ss}_-\cap S_+= (\sfk^{n+1}\backslash\{0\})\times \{0\}$ holds, and the fixed locus $Z$ of $\lambda_{-}$-action on $X$ equals to the origin $\{0\}\cong\Spec \sfk$. In particular, there are natural equivalences
\begin{align}
\coh_GX^{\rm ss}_-&\cong \coh_{G_2}{\rm V}(\cE)\label{X=V}\\
\coh_GE&\cong \coh_{G_2}\bfP(1,a_1,\hdots,a_n)\label{E=P},
\end{align}
and, in what follows, we tacitly use these identifications.
Write $\chi_{a,b}\colon G\to \bG_{m}$ for the character defined by $\chi_{a,b}(s,t)\defeq s^at^b$. Fix an integer $w\in \bZ$. We define
\[
\Dcoh_G(\Spec \sfk,\chi_{0,1},0)_w
\]
to be the thick subcategory of $\Dcoh_G(\Spec \sfk,\chi_{0,1},0)$ generated by factorizations such that the $(\lambda_{-})$-weights of the components is $w$, and consider the subcategory
\[
\cG_w^-\defeq \left\{(F_1\to F_0\to F_1(\chi_{0,1}))\relmiddle| \mbox{ $(\lambda_-)$-weight of  $F_i|_{Z}$ lies in $[w,w+\eta^-)$ } \right\}
\]
of $\Dcoh_G(X,\chi_{0,1},W)$, where $\eta^-$ is the $(\lambda_-)$-weight of $\det(N_{S_-}X)$. 
Then $\Dcoh_G(\Spec \sfk,\chi_{0,1},0)_w$ is generated by an exceptional object $\bigl(0\to \cO(\chi_{-w,0})\to0\bigr)$,  $\eta^-=\alpha+1$ holds, and 
\begin{align*}
\cG_w^-&=\left\{(F_1\to F_0\to F_1(\chi_{0,1}))\relmiddle| \mbox{ $G_1$-weights of  $F_i|_{Z}$ lies in $(-w-\eta^-,-w]$ } \right\}\\
&=\scrM(-w-(\alpha/2)+\bnabla,W),
\end{align*}
where $\alpha\defeq\sum_{i=1}^na_i$. 
Consider an equivalence
\[
\psi_w\colon\Dcoh_G(X^{\rm ss}_-,\chi_{0,1},W)\simto\Dcoh_G(X^{\rm ss}_+,\chi_{0,1},W)
\]
defined by the composition 
$
\Dcoh_G(X^{\rm ss}_-,\chi_{0,1},W)\xrightarrow{\res^{-1}}\cG_{w}^-\xrightarrow{\res}\Dcoh_G(X^{\rm ss}_+,\chi_{0,1},W).
$
Then the {\it window shift autoequivalence} is defined to be the composition
$\Phi_w\defeq \psi_{w+1}^{-1}\circ\psi_w\in \Auteq\Dcoh_G(X^{\rm ss}_-,\chi_{0,1},W)$.
Set 
\[\cK_w\defeq j_*\pi^*\bigl(0\to \cO(\chi_{-w,0})\to0\bigr)\in \Dcoh_G(X^{\rm ss}_-,\chi_{0,1},W),\]
where $\pi\colon E\to \Spec \sfk$ is a natural projection, and $j\colon E\to X^{\rm ss}_-$ is the zero section.

\begin{lem}[{\cite[Section 3.1]{hl-sh}\footnote[1]{Although \cite{hl-sh} only discusses complexes, there are similar functors and semi-orthogonal decompositions for matrix factorizations by \cite{vgit}, and so a similar argument as in \cite{hl-sh} works in our setting.}}]\label{spherical lem 1}
For each $F\in \Dcoh_G(X^{\rm ss}_-,\chi_{0,1},W)$, there is a functorial exact triangle
\begin{equation}
\RHom(\cK_w,F)\otimes \cK_w\to F\to \Phi_w(F)\to\RHom(\cE_w,F[1])\otimes \cE_w.
\end{equation}
\end{lem}

\vspace{3mm}
Write $\bfP\defeq\bfP(1,a_1,\hdots,a_n)$ for simplicity, and  denote by ${\rm V}(\cE)|_Y$  the pull-back of the vector bundle $q\colon{\rm V}(\cE)\to \bfP$ by the closed immersion $i_Y\colon Y\hookto \bfP$. Consider the following diagram
\[
\begin{tikzpicture}[xscale=1.3]
\node (A0) at (1,1.5) {${\rm V}(\cE)|_Y$};
\node (A5) at (4,1.5) {${\rm V}(\cE)$};
\node (B0) at (1,0) {$Y$};
\node (B5) at (4,0) {$\bfP$,};
\draw[->] (A0) -- node[above] {$\scriptstyle i$}(A5);
\draw[->] (B0) -- node[above] {$\scriptstyle i_Y$}(B5);
\draw[->] (B0) edge[bend left] node [left] {$\scriptstyle j_Y$}(A0);
\draw[->] (A0) edge[bend left] node [right] {$\scriptstyle p$}(B0);
\draw[->] (B5) edge[bend left] node [left] {$\scriptstyle j$}(A5);
\draw[->] (A5) edge[bend left] node [right] {$\scriptstyle q$}(B5);
\end{tikzpicture}
\]
where $p$ is the natural projection, $i$ is the natural closed immersion, and $j$ and $j_Y$ are the zero sections. Note that $j\colon \bfP=[E/G_1]\hookto {\rm V}(\cE)=[X_-^{\rm ss}/G_1]$ is induced from $j\colon E\hookto X^{\rm ss}_-$, and so we use the same notation.
By Kn\"orrer periodicity, the functor
\begin{equation}\label{knorrer 2}
i_*p^*\colon \Db(\coh Y)\simto \Dcoh_G(X^{\rm ss}_-,\chi_{0,1},W)
\end{equation}
is an equivalence.  
\begin{lem}
There is an  isomorphism 
 \[\bigl(0\to j_*\cO_E(\chi_{-w,0})\to0\bigr)\cong  i_*p^*\cO_Y(-w-\alpha-1)[r+1]\]
in $\Dcoh_G(X^{\rm ss}_-,\chi_{0,1},W)$.
\begin{proof}
 Recall that $s=(x_0,f_1,\hdots,f_r)\in \Gamma({\bf P},\cE^{*})^{G_2}$, and set  $t\defeq (y_0,\hdots,y_r)\in\Gamma(\bfP,\cE)^{G_2}$. Consider the following $G$-equivariant locally free sheaf on $E$:
\[
\widetilde{\cE}\defeq \cO(\chi_{-1,0})\oplus\cO(\chi_{-d_1,0})\oplus\cdots\oplus \cO(\chi_{-d_r,0}).
\]
Then $\widetilde{\cE}$ corresponds to $\cE$ via \eqref{E=P}, and pulling-back the sections $s$ and $t$ by the projection $q\colon X^{\rm ss}_-\to E$ defines $G$-invariant sections $q^*(s)\in \Gamma(X^{\rm ss}_-,(q^*\widetilde{\cE})^*)^G$ and $q^*(t)\in \Gamma(X^{\rm ss}_-,q^*\widetilde{\cE}(\chi_{0,1}))^G$.
Since $E\cong Z_{q^*(t)^*}$, there are   isomorphisms 
\begin{align}
\bigl(0\to j_*\cO_E(\chi_{-w,0})\to0\bigr)&\cong\bigl(0\to \cO_{Z_{q^*(s)}}(\chi_{-w,0})\to0\bigr)\otimes \det(q^*\widetilde{\cE})[r+1]\\
&\cong \bigl(0\to \cO_{Z_{q^*(s)}}(\chi_{-w-\alpha-1,0})\to0\bigr)[r+1]\\
&\cong \bigl(0\to i_*p^*\cO_{Y}(-w-\alpha-1)\to0\bigr)[r+1]
\end{align}
where the first isomorphism follows from Lemma \ref{koszul lem}. This finishes the proof.
\end{proof}
\end{lem}

The following is a generalization of \cite[Proposition 8.5]{ko}, which we prove by a similar argument as in loc. cit.

\begin{lem}\label{lem for cor}
The following diagram commutes.
\[
\begin{tikzpicture}[xscale=1.3]
\node (A0) at (1,0) {$\Dcoh_G(X^{\rm ss}_-,\chi_{0,1},W)$};
\node (A5) at (5,0) {$\Dcoh_G(X^{\rm ss}_-,\chi_{0,1},W)$};
\node (B0) at (1,1.5) {$\Db(\coh Y)$};
\node (B5) at (5,1.5) {$\Db(\coh Y)$};
\draw[->] (A0) -- node[above] {$\scriptstyle\Phi_{w}$}(A5);
\draw[->] (B0) -- node[above] {$\scriptstyle\ST_{\cO_Y(-w-\alpha-1)}$}(B5);
\draw[->] (B0) -- node[left] {$\scriptstyle i_*p^*$}(A0);
\draw[->] (B5) -- node[right] {$\scriptstyle i_*p^*$}(A5);
\end{tikzpicture}
\]
\begin{proof}
Let $F\in \Db(\coh Y)$. By Lemma \ref{spherical lem 1}, there is a functorial exact triangle
\begin{equation}
\RHom(\cK_w,i_*p^*F)\otimes \cK_w\to i_*p^*F\to \Phi_w(i_*p^*F)\to\RHom(\cE_w,i_*p^*F[1])\otimes \cE_w.
\end{equation}
Thus it is enough to show that there is a natural isomorphism
\begin{equation}\label{cor isom}
\RHom(\cK_w,i_*p^*F)\otimes \cK_w\cong \RHom(\cO_Y(-w-\alpha-1), F)\otimes i_*p^*\cO_Y(-w-\alpha-1).
\end{equation}
By Lemma \ref{lem for cor}, there is an isomorphism 
\[
\cK_w=\bigl(0\to j_*\cO_E(\chi_{-w,0})0\bigr)\cong i_*p^*\cO_Y(-w-\alpha-1)[r+1],
\]
and this implies the required isomorphism \eqref{cor isom} since $i_*p^*$ is an equivalence.
\end{proof}
\end{lem}

\begin{proof}[Proof of Corollary \ref{cor:spherical twist}] For simplicity, write 
\[
\cD_{\pm}\defeq\Dcoh_G(X_{\pm}^{\rm ss},\chi_{0,1},W).
\]
By Lemma \ref{lem for cor}, the spherical twist $\ST_{\cO_Y(m)}$ corresponds to 
$
\Phi_{(-m-\alpha-1)}\colon \cD_-\simto\cD_-
$
via \eqref{knorrer 2}. By the definition of $\Phi_{(-m-\alpha-1)}$, the following diagram commutes:
\begin{equation}\label{diagram for cor}
\begin{tikzpicture}[xscale=1.0]
\node (A1) at (1,0) {$\cG^-_{-m-\alpha-1}$};
\node (A3) at (3,0) {$\cD_+$};
\node (A5) at (5,0) {$\cG^-_{-m-\alpha}$};
\node (A7) at (7,0) {$\cD_-$};
\node (A9) at (9,0) {$\cG^-_{-m-\alpha-1}$};
\node (B1) at (1,1.5) {$\cD_-$};
\node (B9) at (9,1.5) {$\cD_-$};
\draw[->] (A1) -- node[above] {$\scriptstyle\res$}(A3);
\draw[->] (A3) -- node[above] {$\scriptstyle\res^{-1}$}(A5);
\draw[->] (A5) -- node[above] {$\scriptstyle\res$}(A7);
\draw[->] (A7) -- node[above] {$\scriptstyle\res^{-1}$}(A9);
\draw[->] (B1) -- node[left] {$\scriptstyle\res^{-1}$}(A1);
\draw[->] (B9) -- node[right] {$\scriptstyle\res^{-1}$}(A9);
\draw[->] (B1) -- node[above] {$\scriptstyle\Phi_{(-m-\alpha-1)}$}(B9);
\end{tikzpicture}
\end{equation}
Since $\cG^-_{-m-\alpha-1}=\scrM(\delta+\bnabla)$, $\cG^-_{-m-\alpha}=\scrM(\delta-1+\bnabla)$, and $\cD_{\pm}=\Dcoh_{G_2}(\X_{\mp},\chi,W)$, the composition of the bottom equivalences in the above diagram is nothing but the composition
\[
\begin{tikzpicture}[xscale=1.3]
\node (B0) at (1,0) {$\scrM(\delta+\bnabla,W)$};
\node (B5) at (5,0) {$\scrM(\delta-1+\bnabla,W)$};
\node (B9) at (9,0) {$\scrM(\delta+\bnabla,W)$.};
\draw[->] (B0) -- node[above] {$\scriptstyle\uprho\bigl([\delta]\xrightarrow{-\ell}[\delta-1]\bigr)$}(B5);
\draw[->] (B5) -- node[above] {$\scriptstyle\uprho\bigl([\delta-1]\xrightarrow{\ell}[\delta]\bigr)$}(B9);
\end{tikzpicture}
\]
Therefore, the assertion follows from Theorem \ref{window shift}. 
\end{proof}

\appendix

\section{Matrix factorizations}\label{mf appendix}
This appendix recalls definitions and fundamental properties of derived factorization categories.
See  \cite{posi, bfk,bdfik,H1,H3} for more details.

\subsection{Definitions} Let $\X$ be an algebraic stack, and let $G$ be an algebraic group acting on $\X$. 
A {\it gauged Landau-Ginzburg (=LG) model} is  data
\[
(\cA,\chi,W)^G
\]
such that $\cA$ is a $G$-equivariant $\X$-algebra (in the sense of \cite[Section 3]{H3}), $\chi\colon G\to \bG_m$ is a character of $G$ and $W\in \Gamma(\X,\cO(\chi))^G$. 
If $\cA$ is noncommutative, the data is called a {\it noncommutative} gauged LG model. 
If $\cA=\cO_{\X}$, write $(\X,\chi,W)^G$ for the gauged LG model, and if $\X=\Spec R$ and $A\defeq\Gamma(\X,\cA)$, let $(A,\chi,W)^G$ denote the gauged LG model.

\begin{dfn}
Let $(\cA,\chi,W)^G$ be a gauged LG model. A {\it factorization} of $(\cA,\chi,W)^G$ is a sequence
\[
E=\left(E_1\xrightarrow{\varphi_1}E_0\xrightarrow{\varphi_0}E_1(\chi)\right),
\]
where each $E_i$ is a $G$-equivariant coherent $\cA$-module, $E_1(\chi)\defeq E_1\otimes\cO(\chi)$ and each $\varphi_i$ is a $G$-homomorphism such that  $\varphi_0\circ\varphi_1=\id_{E_1}\otimes W$ and $\varphi_1(\chi)\circ\varphi_0=\id_{E_0}\otimes W$.
\end{dfn}
 
Let $E$ and $F$ be factorizations of $(\cA,\chi,W)^G$. A {\it homomorphism} from $E$ to $F$ is a pair $\alpha=(\alpha_0,\alpha_1)$ of $G$-homomorphism $\alpha_i\colon E_i\to F_i$ such that the following diagram is commutative:
\[\xymatrix{
E_1\ar[rr]^{\varphi_1^E}\ar[d]_{\alpha_1}&&E_0\ar[rr]^{\varphi_0^E}\ar[d]^{\alpha_0}&&E_1(\chi)\ar[d]^{\alpha_1(\chi)}\\
F_1\ar[rr]^{\varphi_1^F}&&F_0\ar[rr]^{\varphi_0^F}&&F_1(\chi).
 }\]
Two homomorphisms $\alpha,\beta\colon E\to F$  is said to be {\it homotopy equivalent}, denoted by $\alpha\sim\beta$, if  there exist two homomorphisms 
$h_0:E_0\rightarrow F_1\hspace{3mm}{\rm and}\hspace{3mm}h_1:E_1(\chi)\rightarrow F_0$
of $G$-equivariant $\cA$-modules such that $\alpha_0-\beta_0=\varphi_1^Fh_0+h_1\varphi_0^E$ and $\alpha_1(\chi)-\beta_1(\chi)=\varphi_0^Fh_1+(h_0(\chi))(\varphi_1^E(\chi))$.

\begin{dfn}
Let
\[
\coh_G(\cA,\chi,W)
\]
denote the category of factorizations of $(\cA,\chi,W)^G$, and its {\it homotopy category} 
\[\Kcoh_G(\cA,\chi,W)\]
is defined to be the category such that its objects are the same as $\coh_G(\cA,\chi,W)$ and homomorphisms are defined by
$
\Hom(E,F)\defeq\Hom_{\coh_G(\cA,\chi,W)}(E,F)/\sim
$. 
\end{dfn}

The homotopy category $\Kcoh_G(\cA,\chi,W)$ has a natural  structure of a triangulated category. 
For a bounded complex $E^{\bullet}$ in $\coh_G(\cA,\chi,W)$, 
its {\it totalization} $\Tot(E^{\bullet})\in\coh_G(\cA,\chi,W)$ can be defined as an analogy of totalizations of double complexes. 
We say that a factorization $E$ is {\it acyclic} if there is a bounded acyclic complex $A^{\bullet}$ in $\coh_G(\cA,\chi,W)$  such that $E\cong \Tot(A^{\bullet})$ in $\Kcoh_G(\cA,\chi,W)$. 
Write 
\[
\Acoh_G(\cA,\chi,W)\subseteq \Kcoh_G(\cA,\chi,W)
\] 
for the smallest thick subcategory of $\Kcoh_G(\cA,\chi,W)$ containing all acyclic factorizations.

\begin{dfn}
The {\it derived factorization category} of $(\cA,\chi,W)^G$ is defined by
\[
\Dcoh_G(\cA,\chi,W)\defeq \Kcoh_G(\cA,\chi,W)/\Acoh_G(\cA,\chi,W).
\]
If $\cA=\cO_{\X}$, write $\Dcoh_G(\X,\chi,W)$, and if $\X=\Spec R$ and $A=\Gamma(\X,\cA)$, write $
\Dmod_G(A,\chi,W)$ for the derived factorization categories. 
Similar alternative notations are used for $\coh_G(\cA,\chi,W)$ and $\Kcoh_G(\cA,\chi,W)$. 
\end{dfn}

\subsection{Equivariant tilting theory} Each normal subgroup $H\subseteq G$ associates a natural functor 
\[
\Res^G_H\colon \coh_G\cA\to \coh_H\cA,
\]
which is called  the {\it restriction functor}. 
The rest of this section recalls equivariant tilting theory for factorizations developed in \cite{H3} (see also \cite{ot}).

\begin{dfn}\label{eq tilt} A $G$-equivariant coherent $\cA$-module $\cT\in\coh_G\cA$ is called {\it $(G,H)$-tilting module} if $\Res^G_H(\cT)$ is a tilting object in $\D(\Qcoh_H\cA)$. We say that $\cT$ is {\it $G$-tilting module} if it is $(G,\{1\})$-tilting module.
\end{dfn}

 Let $H\subseteq G$ be a normal subgroup, and let $Y=\Spec R$ be an affine variety with a $G$-action such that the $H$-action on $Y$ is trivial. 
 Then there is an induced $G/H$-action on $Y$. 
 Let $X$ be a variety with a $G$-action, and let $f\colon X\to Y$ be a $G$-equivariant morphism such that $f_*(\cO_X)\cong R$ and the induced morphism $[f/H]\colon [X/H]\to \Spec R$ is proper. 
 Let $\cA$ be a $G$-equivariant coherent $X$-algebra, and let  $\cT\in \coh_GX$. 
Then its endomorphism algebra $\Lambda_{\cT}\defeq \End(\cT)$ is a $G$-equivariant coherent $R$-algebra, 
and its $H$-invairant part $\Lambda_{\cT}^H$ is a $G/H$-equivariant coherent $R$-algebra. 
Thus there is a left exact functor 
\begin{equation}\label{tilt functor}
\Hom(\cT,-)^H\colon \coh_G\cA\to \fmod_{G/H}\Lambda_{\cT}^H.
\end{equation}
Let $\chi\colon G/H\to \bG_m$ be a character of $G/H$, and set $\widehat{\chi}\defeq \chi\circ p\colon G\to\bG_m$, where $p\colon G\to G/H$ is the natural projection.  Let $W\in \Gamma(Y,\cO(\chi))^{G/H}=\Gamma(Y,\cO(\widehat{\chi}))^{G}$. Then these define  gauged LG models
\[
(\cA,\widehat{\chi},W)^G \hspace{3mm}\mbox{ and }\hspace{3mm} (\Lambda_{\cT}^H,\chi,W)^{G/H},
\]
where $W$ in the left LG model denotes $f^*W$ by abuse of notation, and  the functor \eqref{tilt functor}  defines the right derived functor 
\begin{equation}\label{mf tilt}
\Hom(\cT,-)^H\colon \coh_G(\cA,\widehat{\chi},W)\to \fmod_{G/H}(\Lambda_{\cT}^H,\chi,W).
\end{equation}
If the abelian category $\Qcoh_G\cA$ has finite injective resolutions, the functor \eqref{mf tilt} induces the right derived fucntor 
\[
\RHom(\cT,-)^H\colon \Dcoh_G(\cA,\widehat{\chi},W)\simto \Dmod_{G/H}(\Lambda_{\cT},\chi,W).
\]
The following shows an equivariant and factorization version of a tilting equivalence.

\begin{thm}[{\cite[Theorem 4.7]{H3}}]\label{H3 thm} Let $\cT$ be a $(G,H)$-tilting module. Assume that $\Qcoh_G\cA$ has finite injective resolutions,  $G/H$ is reductive and $\Lambda_{\cT}^H$ is of finite global dimension. Then the right derived functor 
\[
\RHom(\cT,-)^H\colon \Dcoh_G(\cA,\widehat{\chi},W)\simto \Dmod_{G/H}(\Lambda_{\cT}^H,\chi,W).
\]
of \eqref{mf tilt} is an equivalence.
\end{thm}

\subsection{Koszul factorizations} Let $(X,\chi,W)^G$ be a gauged LG model such that $X$ is a smooth variety. Let $\cE$ be a $G$-equivariant locally free sheaf on $X$ of rank $r$, and let $s\in \Gamma(X,\cE^{*})^G$ and $t\in \Gamma(X,\cE(\chi))^G$ be $G$-invariant sections such that $t\circ s=W\cdot \id_{\cE}$ and $s(\chi)\circ t=W$. The {\it Koszul factorization} $K(s,t)\in \coh_G(X,\chi,W)$ associated to $s$ and $t$ is defined by 
\[
K(s,t)\defeq\left(\bigoplus_{n\in \bZ}(\bigwedge^{2n+1}\cE)(\chi^n)\xrightarrow{\varphi_1}\bigoplus_{n\in \bZ}(\bigwedge^{2n}\cE)(\chi^n)\xrightarrow{\varphi_0}\bigoplus_{n\in \bZ}(\bigwedge^{2n+1}\cE)(\chi^{n+1})\right),
\]
where $\varphi_i\defeq t\wedge(-)\oplus s\vee(-)$. Denote by $Z_s\subset X$ the zero scheme of $s$, and write $i\colon Z_{s}\hookto X$ for the closed immersion. Then the relative canonical bundle $\omega_i\defeq \omega_{Z_s}\otimes (\omega^*_X|_{Z_s})$ is isomorphic to $\det(\cE^*)$. 
The following is a fundamental property of Koszul factoizations.
\begin{lem}[{\cite[Proposition 3.20]{bfk}\footnote[1]{There is a typo in the latter assertion in loc. cit.}}] \label{koszul lem} Assume that the sections $s$ and $t^*$ are regular. Then
 there are isomorphisms
\begin{align*}
K(s,t)\cong \bigl(0\to \cO_{Z_s}\to0\bigr)
\cong \bigl(0\to \cO_{Z_{t^*}}\to0\bigr)\otimes \omega_i[-r].
\end{align*}
in $\Dcoh_G(X,\chi,W)$.
\end{lem}

\section{List of Notations}\label{notation appendix}
\noindent
\renewcommand{\arraystretch}{1.5}
\begin{longtable}[l]{p{2.3cm} p{8.25cm} p{3cm}}

$w\ast \chi$ 
&
$\ast$-action of the Weyl group $W$ on $\M_\bR$
&
p.\pageref{not: ast action}, line \lineref{line not: ast action}.
\\

$\upbeta_1,\hdots,\upbeta_d$ 
&
weights of a quasi-symmetric representation $X$ 
&
p.\pageref{not: upbata}, line \lineref{line not: upbata}.
\\

$\upbeta_F^+$ 
&
$\defeq \sum_{\upbeta_i\in \scrW_{F}^+}\upbeta_i$
&
p.\pageref{not: upbeta_F^+}, line \lineref{line not: upbeta_F^+}.
\\

$\Gamma(\scrH^W)$ 
&
the groupoid associated to $\scrH^W$
&
p.\pageref{not: gamma gpd}, line \lineref{line not: gamma gpd}
\\

$\Lambda_{\delta}$
&
NCCR associated to $\delta$.
&
p.\pageref{end alg}, line \lineref{line not: NCCR}. 
\\

$\mu_N(N)$
&
a right mutation of $M$ at $N$ 
&
p.\pageref{not: mutation}, line \lineref{line not: mutation}. 
\\

$\mu_N^-(M)$
&
a left mutation of $M$ at $N$ 
&
p.\pageref{not: mutation}, line \lineref{line not: mutation}. 
\\

$\mu_{(\delta,\delta')}(\chi)$
&
$\defeq (\chi + \upbeta_{F_{\chi}(\delta_0)}^+)^+$
&
p.\pageref{not: mudeltadelta}, line \lineref{line not: mudeltadelta}
\\

$\rho$
&
$\defeq (1/2)\sum_{\alpha\in \Phi^+}\alpha$
&
Section \ref{notation}
\\

$\bsigma$
&
$\defeq\left\{\sum_{i=1}^da_i\upbeta_i\middle| a_i\in[-1,0]\right\}$
&
p.\pageref{not: bisigma}, line \lineref{line not: bsigma}.
\\

$\Upphi$
&
the set of roots of $G$
&
Section \ref{notation}
\\

$\Upphi^{\pm}$
&
the set of positive and negative roots of $G$
&
Section \ref{notation}
\\

$\Phi_N$
&
the equivalence associated to the mutation at $N$ 
&
p.\pageref{not: mutation equiv}, line \lineref{line not: mutation equiv}. 
\\

$\chi^+$
&
the dominant element in $W \ast \chi$
&
p.\pageref{not: dom chi}, line \lineref{line not: dom chi}
\\

$\bnabla$
&
$\defeq \left\{\chi\in\M_{\bR}\relmiddle| -\eta_{\lambda}/2\leq \l\chi,\lambda\r\leq \eta_{\lambda}/2 \mbox{ for all $\lambda\in \N$}\right\}$
&
p.\pageref{not: bnabla}, line \lineref{line not: bnabla}. 
\\

$B^*$
&
$\defeq-B+2c\in\scrB(P)$
&
p.\pageref{dual facet}, line \lineref{line not: dual facet}.
\\

$\scrB(P)$
&
$\defeq\{ B\in\scrS(P)\mid \mbox{$B\cap P$ is a facet of $P$}\}$
&
p.\pageref{not: B(P) and F(P)}, line \lineref{line not: B(P) and F(P)}.
\\

$\scrC_{\delta}$
&
$\defeq (\delta+\bnabla)\cap \M^+$
&
p.\pageref{x delta}, line \lineref{line x delta}.
\\

$\scrC_{\delta}^F$
&
$\defeq\left\{ \chi \in \scrC_{\delta}\relmiddle| F=F_{\chi}(\delta_0)\right\}$
&
p.\pageref{not: C^F_delta}, line \lineref{line not: C^F_delta}.
\\

$d_F^+$
&
$\defeq \#\scrW_F^+$
&
p.\pageref{not: d_F^+}, line \lineref{line not: d_F^+}.
\\

$\cE_{(L,N)}(\cS)$
&
the set of right exchanges
&
p.\pageref{not:set of right exchanges}, line \lineref{line not:set of right exchanges}. 
\\

$\cE_{(L,N)}^{-}(\cS)$
&
the set of left exchanges
&
p.\pageref{not:set of left exchanges}, line \lineref{line not:set of left exchanges}.
\\

$F^*$
&
$\defeq-F+2c\in\scrF^k(P)$
&
p.\pageref{dual facet}, line \lineref{line not: dual facet}.
\\

$\scrF(P)$
&
$\defeq\bigcup_{k=1}^{n}\scrF^k(P)$
&
p.\pageref{not: F(P)}, line \lineref{line not: F(P)}.
\\

$\scrF^k(P)$
&
$\defeq\{\mbox{codimension $k$ face of $P$}\}$
&
p.\pageref{not: B(P) and F(P)}, line \lineref{line not: B(P) and F(P)}.
\\

$\scrF_+(P), \scrF_+^k(P)$
&
the sets of dominant faces
&
p.\pageref{not: dom faces}, line \lineref{line not: dom faces}.
\\

$\scrF_{(\delta,\delta')}$
&
$\defeq\left\{ F\in \scrF_+(\delta_0+(1/2)\bsigma)\relmiddle| \exists \chi\in \scrC_{\delta}, F=F_{\chi}(\delta_0)\right\}$
&
p.\pageref{not: Fdeltadelta}, line \lineref{line not: Fdeltadelta}
\\

$G$
&
connected reductive group
&
Section \ref{notation}
\\

$\scrH$
&
a set of  hyperplanes
&
p.\pageref{not: scrH}, line \lineref{line not: scrH}. 
\\

$\scrH^W$
&
$(\M^W)$-periodic hyperplane arrangement
&
p.\pageref{not: H^W}, line \lineref{line not: H^W}. 
\\

$\scrH(\delta,\delta')$
&
the set of separating hyperplanes
&
p.\pageref{not: sep hyp set}, line \lineref{line not: sep hyp set}.
\\

$L^F_{\delta}$
&
$\defeq \bigoplus_{\chi\in \scrC_{\delta}^F}\left( \bigoplus_{{\upbeta\in{\tiny \gwedge}^*\scrW_F^+}} M((\chi+\upbeta)^+)\right)$
&
p.\pageref{not: LN}, line \lineref{line not: LN}
\\

$\scrM(\delta+\bnabla)$
&
magic window subcategory
&
p.\pageref{not: magic window}, line \lineref{line not: magic window}. 
\\

$M(\chi)$
&
$\defeq \left( V(\chi) \otimes_{\Bbbk} \Bbbk[X] \right)^G$
&
p.\pageref{not: MCM ass to chi}, line \lineref{line not: MCM ass to chi}. 
\\

$M_{\delta}$
&
$\defeq\bigoplus_{\chi\in\scrC_{\delta}}M(\chi)$
&
p.\pageref{not: M_delta}, line \lineref{ line not: M_delta}. 
\\

$M_{\delta}^F$
&
$\defeq\bigoplus_{\chi\in \scrC_{\delta}^F}M(\chi)$
&
p.\pageref{not: M^F_delta}, line \lineref{line not: M^F_delta}.
\\

$M_{\delta\cap\delta'}$
&
$\defeq \bigoplus_{\chi\in \scrC_{\delta}\cap \scrC_{\delta'}}M(\chi)$
&
p.\pageref{not: M_common}, line \lineref{line not: M_common} 
\\

$\M$ 
&
the character lattice of $T$
&
Section \ref{notation}
\\

$(\M_{\bR}^{W})_{\rm gen}$
&
$\defeq \M_{\bR}^W\backslash\bigcup_{A\in\scrA}\bigl(A\cap \M_{\bR}^{W}\bigr)$
&
p.\pageref{not: M_R^W}, line \lineref{ line not: M_R^W}. 
\\

$N^F_{\delta}$
&
$\defeq \bigoplus_{\chi\in \scrC_{\delta}\backslash \scrC^F_{\delta}}M(\chi)\cong M_{\delta}/M_{\delta}^F$
&
p.\pageref{not: LN}, line \lineref{line not: LN}
\\

$T$
&
fixed maximal torus of $G$
&
Section \ref{notation}
\\

$\scrT_{\delta}$
&
$\defeq \bigoplus_{\chi\in\scrC_{\delta}}V_X(\chi)\in\coh[X/G]$
&
p.\pageref{tilting bdl T}, line \lineref{line: tilting bdl T}.
\\

$\scrV_{\delta}$
&
$\defeq \bigoplus_{\chi\in\scrC_{\delta}}V_{\X}(\chi) \in \coh \X$
&
p.\pageref{tilting bdl}, line \lineref{line: tilting bdl}.
\\

$\scrW^+_F$
&
$\defeq\left\{\upbeta_i\relmiddle|\mbox{$\l\upbeta_i,\lambda\r>0$ for some $\lambda\in\scrL_F$}\right\}$
&
p.\pageref{not: W_F}, line \lineref{line not: W_F}.
\\

$\scrW^0_F$
&
$\defeq\left\{\upbeta_i\relmiddle|\mbox{$\l\upbeta_i,\lambda\r=0$ for all $\lambda\in\scrL_F$}\right\}$
&
p.\pageref{not: W_F}, line \lineref{line not: W_F}.
\\

$\X$
&
$\defeq[X^{\rm ss}(\ell)/G]$
&
p.\pageref{not: GIT stack}, line \lineref{line not: GIT stack}.

\end{longtable}
\renewcommand{\arraystretch}{1}


\begin{thebibliography}{9F}
 
\bibitem[1]{bfk} M.~Ballard, D.~Favero and L.~Katzarkov, {\it A category of kernels for equivariant factorizations and its implications for Hodge theory}. Publ. Math. Inst. Hautes $\acute{\rm E}$tudes Sci. {\bf 120}, 1--111 (2014).
 
 

 
 \bibitem[2]{vgit}  M.~Ballard, D.~Favero and L.~Katzarkov, {\it Variation of geometric invariant theory quotients and derived categories}, J. Reine Angew. Math. {\bf 746}, 235--303 (2019).
 
\bibitem[3]{bdfik} M.~Ballard, D.~Deliu, D.~Favero, M.~U.~Isik, and L.~Katzarkov, {\it Resolutions in factorization categories}. Adv. Math. {\bf 295}, 195--249 (2016).

\bibitem[4]{BLS} D.~Bergh, V.~A.~Lunts, O.~M.~Schn\"{u}rer1, {\it Geometricity for derived categories of algebraic stacks}, Selecta Math. (N.S.) {\bf 22} (2016), no. 4, 2535--2568.  

 \bibitem[5]{bri} T.~Bridgeland, {\it Flops and derived categories}, Invent.\ Math.\ {\bf 147} (2002), no.~3, 613--632. 

 \bibitem[6]{bh} W.~Bruns and J.~Herzog, {\it Cohen-Macaulay rings}. Cambridge Studies in Advanced Mathematics, {\bf 39}.

 \bibitem[7]{che}
J.-C.~Chen, \emph{Flops and equivalences of derived categories for threefolds with only terminal Gorenstein singularities}, J. Differential Geom. \textbf{61} (2002), no. 2, 227--261.
 



\bibitem[8]{hara17} W.~Hara, {\it Non-commutative crepant resolution of minimal nilpotent orbit closures of type A and Mukai flops}.
Adv. Math.{\bf 318}(2017), 355--410.

\bibitem[9]{hara22} W.~Hara, {\it On derived equivalence for Abuaf flop: mutation of non-commutative crepant resolutions and spherical twists}.
Matematiche (Catania) {\bf 77}(2022), no.2, 329--371.
    
  
\bibitem[10]{hl} D.~Halpern-Leistner, {\it The derived category of a GIT quotient}. J. Amer. Math. Soc. {\bf 28} (2015), no. 3, 871--912.
   
  \bibitem[11]{hl-s} D.~Halpern-Leistner and S.~V.~Sam, {\it Combinatorial constructions of derived equivalences}. J. Amer. Math. Soc. {\bf 33}, no. 3, 871--912  (2020).
   
\bibitem[12]{hl-sh} D.~Halpern-Leistner and I.~Shipman, {\it Autoequivalences of derived categories via geometric invariant
theory.} Adv. Math. {\bf 303}, 1264--1299 (2016).

\bibitem[13]{HigNak} A.~Higashitani and Y.~Nakajima,
{\it Conic divisorial ideals of Hibi rings and their applications to non-commutative crepant resolutions}.
Selecta Math. (N.S.) {\bf 25}(2019), no.5, Paper No. 78, 25 pp.
 
 
 \bibitem[14]{H1} Y.~Hirano, {\it Equivalences of derived factorization categories of gauged Landau-Ginzburg models}. Adv. Math. {\bf 306}, 200--278 (2017).
 
 
\bibitem[15]{H2} Y.~Hirano, {\it Derived Kn\"orrer periodicity and Orlov's theorem for gauged Landau-Ginzburg models}. Compos. Math. {\bf 153}, no. 5, 973--1007 (2017).

 
\bibitem[16]{H3} Y.~Hirano, {\it Equivariant Tilting Modules, Pfaffian Varieties
and Noncommutative Matrix Factorizations
}. SIGMA Symmetry Integrability Geom. Methods  Appl.  {\bf 17}, Paper No. 055, 43 pp (2021). 
 
 \bibitem[17]{HW} Y.~Hirano and M.~Wemyss, {\it Faithful actions from hyperplane arrangements}, Geom.\  Topol.\ \textbf{22} (2018), no.~6, 3395--3433. 

  \bibitem[18]{HW2} Y.~Hirano and M.~Wemyss, {\it Stability conditions for 3-fold flops}, Duke Math. J. \textbf{172} (2023), no. 16, 3105--3173. 

\bibitem[19]{HR} J.~Hall, D.~Rydh, {\it Perfect complexes on algebraic stacks}, Compos. Math. {\bf 153} (2017), no. 11, 2318--2367. 

 \bibitem[20]{isik} M.~U.~Isik, {\it Equivalence of the derived category of a variety with a singularity category}, Int. Math. Res. Not. IMRN (2013), no. 12, 2787--2808.
 
 \bibitem[21]{IR} O.~Iyama and I.~Reiten, {\it Fomin-Zelevinsky mutation and tilting modules over Calabi-Yau algebras}. Am. J. Math. {\bf 130} (4), 1087--1149 (2008).

 \bibitem[22]{IW}  O.~Iyama and M.~Wemyss, {\it Maximal modifications and Auslander-Reiten duality for non-isolated singularities}. Invent.\ Math.\ \textbf{197} (2014), no.~3, 521--586.
 
 \bibitem[23]{IW9}
O.~Iyama and M.~Wemyss, \emph{Tits cones intersections and applications}, \href{https://www.maths.gla.ac.uk/~mwemyss/MainFile_for_web.pdf}{\sf preprint}. 
 


\bibitem[24]{kaw} Y.~Kawamata, {\it Flops connect minimal models}, Publ. RIMS, {\bf 44}, 419--423 (2008).
 

 \bibitem[25]{ko} N.~Koseki and G.~Ouchi, {\it Perverse schobers and Orlov equivalences}, Eur. J. Math. {\bf 9}, no. 2, Paper No. 32, 38 pp (2023).
 

 


\bibitem[26]{Nak} Y.~Nakajima,
{\it Mutations of splitting maximal modifying modules: the case of reflexive polygons}.
Int. Math. Res. Not. IMRN(2019), no.2, 470--550.

 
\bibitem[27]{ot} C.~Okonek and A.~Teleman, {\it Graded tilting for gauged Landau-Ginzburg models and geometric applications}. Pure Appl. Math. Q. \textbf{17} (2021), no. 1, 185--235. 

 

 
 
\bibitem[28]{posi} L.~Positselski, {\it Two kinds of derived categories, Koszul duality, and comodule-contramodule correspondence}. Mem. Amer. Math. Soc. 212 (2011), no. 966. 
 

  


\bibitem[29]{shipman} I.~Shipman, {\it A geometric approach to Orlov's theorem}, Compos. Math. {\bf 148}, no. 5, 1365-1389  (2012).
 
\bibitem[30]{svdb}\v{S}.~\v{S}penko and M.~Van den Bergh, {\it Non-commutative resolutions of quotient singularities for reductive groups}. Invent. Math. {\bf 210}, no. 1, 3--67 (2017).

\bibitem[31]{svdbtoric1} \v{S}.~\v{S}penko and M.~Van den Bergh,  {\it
Non-commutative crepant resolutions for some toric singularities I}.
Int. Math. Res. Not. IMRN(2020), no.21, 8120--8138.

\bibitem[32]{svdbtoric2} \v{S}.~\v{S}penko and M.~Van den Bergh,  {\it Non-commutative crepant resolutions for some toric singularities. II}. J. Noncommut. Geom. {\bf 14} (2020), no. 1, 73--103.

\bibitem[33]{svdbhypertoric} \v{S}.~\v{S}penko and M.~Van den Bergh,  {\it Tilting bundles on hypertoric varieties}.
Int. Math. Res. Not. IMRN(2021), no.2, 1034--1042.

\bibitem[34]{svdbNCBO} \v{S}.~\v{S}penko and M.~Van den Bergh, J.-P.~Bell,
{\it On the noncommutative Bondal-Orlov conjecture for some toric varieties}.
Math. Z. {\bf 300}(2022), no.1, 1055--1068.

\bibitem[35]{stacks} The Stacks Project Authors, {\it Stacks Project}. \href{https://stacks.math.columbia.edu}{https://stacks.math.columbia.edu}
 
\bibitem[36]{tel} C.~Teleman, {\it The quantization conjecture revisited}.  Ann. of Math. (2) {\bf 152} (2000),
no. 1, 1--43.
 
 


 

  \bibitem[37]{vdb3dflop}
M.~Van den Bergh, \emph{Three-dimensional flops and noncommutative rings}, 
Duke Math.\ J.\ \textbf{122}  (2004), no.~3, 423--455.

 \bibitem[38]{vdbnccr} M.~Van den Bergh, {\it Non-commutative crepant resolutions.} The Legacy of Niels Henrik Abel, pp. 749--770. Springer, Berlin (2004)
 
 \bibitem[39]{vdb} M.~Van den Bergh, {\it Non-commutative crepant resolutions, an overview}. arXiv:2207.09703.



\bibitem[40]{HomMMP}
M.~Wemyss, \emph{Flops and Clusters in the Homological Minimal Model Program}, Invent.\ Math.\ \textbf{211} (2018), no.~2, 435--521.


\end{thebibliography}
\end{document}